\documentclass[12pt,a4paper]{amsart}
\usepackage[utf8]{inputenc}
\usepackage[active]{srcltx}
\usepackage{amsthm}

\makeatletter

\pdfpageheight\paperheight
\pdfpagewidth\paperwidth

\numberwithin{equation}{section}
\numberwithin{figure}{section}

\usepackage{amsmath}

\usepackage{amssymb}
\usepackage{stmaryrd}
\usepackage{amsthm}
\usepackage[mathscr]{eucal}
\usepackage{amscd}
\usepackage[all]{xy}
\usepackage{url}

\newtheorem{Thm}{Theorem}[subsection]

\usepackage{enumerate}

\usepackage{comment}

\usepackage{hyperref}
\hypersetup{%
pdftitle={Preprint},
pdfauthor={},
pdfkeywords={},
bookmarksnumbered,
pdfstartview={FitH},
breaklinks=true,
colorlinks=true,
}%
\usepackage[all]{hypcap} 

\usepackage{breakurl}

\usepackage{color}

\newtheorem{Lem}[Thm]{Lemma}
\newtheorem{Prop}[Thm]{Proposition}
\newtheorem{Cor}[Thm]{Corollary}
\newtheorem{Conj}[Thm]{Conjecture}
\newtheorem{Quest}[Thm]{Question}
\newtheorem{Eg}[Thm]{Example}
\newtheorem{Rem}[Thm]{Remark}
\newtheorem{Def}[Thm]{Definition}

\newtheorem*{Def*}{Definition}
\newtheorem*{Thm*}{Theorem}

\newtheorem*{Conj*}{Conjecture}
\newtheorem*{Assumption*}{Assumption}

\newcommand{\kk}{\Bbbk}
\newcommand{\Z}{\mathbb{Z}}
\newcommand{\N}{\mathbb{N}}
\newcommand{\Q}{\mathbb{Q}}
\newcommand{\C}{\mathbb{C}}
\newcommand{\R}{\mathbb{R}}

\renewcommand{\hat}[1]{\widehat{#1}}
\renewcommand{\tilde}[1]{\widetilde{#1}}

\newcommand{\opname}[1]{\operatorname{\mathsf{#1}}}

\newcommand{\Spec}{\operatorname{\mathsf{Spec}}}

\renewcommand{\mod}{\opname{mod}}

\newcommand{\per}{\opname{per}}

\newcommand{\pr}{\opname{pr}}

\newcommand{\proj}{\opname{proj}}

\newcommand{\add}{\opname{add}}

\newcommand{\op}{^{op}}

\newcommand{\ra}{\rightarrow}

\newcommand{\Gr}{\opname{Gr}}

\newcommand{\sign}{\opname{sign}}

\newcommand{\Hom}{\opname{Hom}}

\newcommand{\supp}{\opname{supp}}

\renewcommand{\deg}{\opname{deg}}

\newcommand{\Hf}{{\frac{1}{2}}}

\newcommand{\Rm}[1]{{\longmapsto}}
\newcommand{\Lm}[1]{{\longmapsfrom}}

\newcommand{\cA}{{\mathcal A}}

\newcommand{\cC}{{\mathcal C}}

\newcommand{\cF}{{\mathcal F}}

\newcommand{\cS}{{\mathcal S}}

\newcommand{\uM}{{\underline{M}}}

\newcommand{\uT}{{\underline{T}}}

\newcommand{\tB}{{\widetilde{B}}}

\newcommand{\tE}{{\widetilde{E}}}
\newcommand{\tF}{{\widetilde{F}}}
\newcommand{\tG}{{\widetilde{G}}}

\newcommand{\tI}{{\widetilde{I}}}

\newcommand{\tQ}{{\widetilde{Q}}}

\newcommand{\tW}{{\widetilde{W}}}

\newcommand{\tg}{{\tilde{g}}}

\newcommand{\gen}{\mathbb{L}}

\newcommand{\clAlg}{{\cA}}
\newcommand{\qClAlg}{\cA_q}

\newcommand{\diag}{{\delta}}

\newcommand{\gr}{{\mathrm{gr}}}

\usepackage{tikz}
\usetikzlibrary{positioning,shapes,shadows,arrows,snakes}

\pgfdeclarelayer{edgelayer}
\pgfdeclarelayer{nodelayer}
\pgfsetlayers{edgelayer,nodelayer,main}

\tikzstyle{none}=[inner sep=0pt]
\tikzstyle{black box}=[draw=black, fill=black!25]
\tikzstyle{white box}=[draw=black, fill=white]
\tikzstyle{black circle}=[circle,draw=black!50, fill=black!25]
\tikzstyle{red circle}=[circle,draw=red!50, fill=red!25]
\tikzstyle{blue circle}=[circle,draw=blue!50, fill=blue!25]
\tikzstyle{green circle}=[circle,draw=green!50, fill=green!25]
\tikzstyle{yellow circle}=[circle,draw=yellow!50, fill=yellow!25]

\newcommand{\thistheoremname}{}
\newtheorem*{genericthm*}{\thistheoremname}
\newenvironment{namedthm*}[1]
  {\renewcommand{\thistheoremname}{#1}%
   \begin{genericthm*}}
  {\end{genericthm*}}

\renewcommand{\diag}{{d}}

\newcommand{\Ind}{{\opname{Ind}}}

\newcommand{\fv}{\opname{f}}
\newcommand{\ufv}{\opname{uf}}
\newcommand{\codeg}{\opname{codeg}}
\newcommand{\suppDim}{\opname{suppDim}}
\newcommand{\ptSet}{\mathcal{PT}}
\newcommand{\hPtSet}{\widehat{\mathcal{PT}}}

\makeatother

\begin{document}
\newtheorem{DefLem}[Thm]{Definition-Lemma}

\renewcommand{\qClAlg}{{\clAlg_q}}
\newcommand{\tropSet}{{\mathcal{M}^\circ}}
\renewcommand{\Mc}{{M^\circ}}
\newcommand{\yCone}{{N_{\ufv}}}

\newcommand{\sol}{\mathrm{TI}}
\newcommand{\intv}{{\mathrm{BI}}}

\newcommand{\Perm}{\mathrm{P}}

\newcommand{\bideg}{\opname{bideg}}

\newcommand{\midClAlg}{{\clAlg^{\mathrm{mid}}}}

\newcommand{\upClAlg}{\mathcal{U}}

\newcommand{\canClAlg}{{\clAlg^{\mathrm{can}}}}

\newcommand{\AVar}{\mathbb{A}}
\newcommand{\XVar}{\mathbb{X}}
\newcommand{\bAVar}{\overline{\mathbb{A}}}

\newcommand{\Jac}{\hat{\mathop{J}}}

\newcommand{\wt}{\mathrm{cl}}
\newcommand{\cl}{\mathrm{cl}}

\newcommand{\domCone}{{\overline{M}^\circ}}

\newcommand{\LP}{{\mathcal{LP}}}
\newcommand{\bLP}{{\overline{\mathcal{LP}}}}

\newcommand{\bClAlg}{{\overline{\clAlg}}}
\newcommand{\bUpClAlg}{{\overline{\upClAlg}}}

\newcommand{\tree}{{\mathbb{T}}}

\renewcommand{\diag}{{d'}}

\newcommand{\img}{{\mathrm{Im}}}

\newcommand{\Id}{{\mathrm{Id}}}
\newcommand{\prin}{{\mathrm{prin}}}

\newcommand{\mm}{{\mathbf{m}}}

\newcommand{\cPtSet}{{\mathcal{CPT}}}
\newcommand{\bPtSet}{{\mathcal{BPT}}}
\newcommand{\tCPtSet}{{\widetilde{\mathcal{CPT}}}}

\newcommand{\tf}{{\tilde{f}}}
\newcommand{\ty}{{\tilde{y}}}
\newcommand{\tcS}{{\tilde{\mathcal{S}}}}

\newcommand{\frd}{{\mathfrak{d}}}
\newcommand{\frD}{{\mathfrak{D}}}
\newcommand{\frp}{{\mathfrak{p}}}
\newcommand{\frg}{{\mathfrak{g}}}
\newcommand{\frn}{{\mathfrak{n}}}
\newcommand{\frsl}{{\mathfrak{sl}}}

\newcommand{\Quot}{\mathrm{Quot}}

\title[]{Bases for upper cluster algebras and tropical points}

\dedicatory{In memory of Kentaro Nagao}

\author{Fan QIN}

\address{School of Mathematical Sciences, Shanghai Jiao Tong University, 800
	Dongchuan Road, Shanghai 200240, China.}

\email{fgin11@sjtu.edu.cn}
\begin{abstract}
It is known that many (upper) cluster algebras possess different kinds
of good bases which contain the cluster monomials and are parametrized
by the tropical points of cluster Poisson varieties. For a large class of 
upper cluster algebras (injective-reachable ones with full rank coefficients), we describe all of its bases with these properties.
Moreover, we show the existence of the generic basis for them. In
addition, we prove that Bridgeland's representation theoretic formula
is effective for their theta functions (weak genteelness).

Our results apply to (almost) all well-known cluster algebras arising from
representation theory or higher Teichm\"uller theory, including quantum
affine algebras, unipotent cells, double Bruhat cells, skein algebras
over surfaces, where we change the coefficients if necessary so that the full rank assumption holds.
\end{abstract}

\maketitle
\tableofcontents{}

\section{Introduction}

\label{sec:intro} 

\subsection{Background: good bases for cluster algebras}

Cluster algebras $\clAlg$ are commutative algebras equipped with
extra combinatorial data. Fomin and Zelevinsky \cite{FominZelevinsky02}
invented these algebras as a combinatorial approach to the dual canonical
bases of quantized enveloping algebras (\cite{Lusztig90}\cite{Lusztig91}\cite{Kashiwara90}).
They conjectured that the cluster monomials (certain monomials of
generators) of some cluster algebras are elements of the dual canonical
bases of quantized enveloping algebra. Similarly, an analogous conjecture
due to Hernandez and Leclerc \cite{HernandezLeclerc09} expected that
the cluster monomials of some other cluster algebras correspond to
simple modules of quantum affine algebras. Inspired by these conjectures,
there have been many works devoted to relate cluster algebras, their
bases and representation theory (\cite{BuanMarshReinekeReitenTodorov06}
\cite{Keller05} \cite{CalderoReineke08} \cite{DerksenWeymanZelevinsky09}
\cite{Amiot09} \cite{Plamondon10a} \cite{GeissLeclercSchroeer10}
\cite{GeissLeclercSchroeer11} \cite{GeissLeclercSchroeer10b} \cite{plamondon2013generic}
\cite{HernandezLeclerc09} \cite{Nakajima09} \cite{KimuraQin14} \cite{qin2017triangular} \cite{Kang2018}
etc...).

On the other hand, to each cluster algebra $\clAlg$, one can define
geometric objects $\AVar$ and $\XVar$ called the cluster $K2$ variety
and cluster Poisson variety respectively \cite{fock2016symplectic}.
The upper cluster algebra $\upClAlg$ is defined to be the ring of
the regular functions over $\AVar$. Furthermore, (a weaker form\footnote{Fock and Goncharov expect an additional stronger property that the
basis should have positive structure constants. For the moment, we
don't know how to pick out such positive bases from the candidates
provided in our paper.} of) a conjecture by Fock and Goncharov predicts that $\upClAlg$
possesses a basis parametrized by the tropical points of $\XVar$
associated to the Langlands dual cluster algebra \cite{FockGoncharov09}.
Gross-Hacking-Keel-Kontsevich recently verified it in many cases and
found that the conjecture does not always hold \cite{gross2018canonical}.

It is well known that the cluster algebra $\clAlg$ is contained in
the upper cluster algebra $\upClAlg$ (Laurent phenomenon \cite{FominZelevinsky02}),
and they agree in many cases, e.g. for many cluster algebras arising
from representation theory. In view of the above conjectures, it is
natural to look for good bases of (upper) cluster algebras, where
the meaning of ``good'' depends on the context. Good bases in the
literature can be divided into the following three families, see Section
\ref{sec:Preliminaries} for necessary definitions.

\begin{enumerate}

\item The generic basis in the sense of \cite{dupont2011generic}:
If the collection of the ``generic'' cluster characters from certain triangulated
category is a basis, it is called the generic basis. The existence of such bases is mostly known for the cluster algebras
arising from unipotent cells \cite{GeissLeclercSchroeer10b}, in which
case it agrees with the dual semicanonical basis of Lusztig \cite{Lusztig00}.
Also, its existence is preserved by source/sink extension \cite{Fei2017generic}.
Conjecturally, this family includes the bangle basis \cite{musiker2013bases}\cite{felikson2017bases}
of cluster algebras arising from surfaces, with the no punctured case
treated in \cite{geiss2020schemes}\cite{GLSsurface}.

\item The common triangular basis in the sense of \cite{qin2017triangular}:
It is defined using some triangular properties by \cite{qin2017triangular}
for ``injective-reachable'' quantum cluster algebras. Its existence
is known for the quantum cluster algebras arising from quantum affine
algebras, where it agrees with the basis consisting of the simple
modules \cite{qin2017triangular}. Also, its existence is known for
those arising from unipotent cells, where it agrees with the dual
canonical basis \cite{qin2017triangular}\cite{Kang2018}\cite{kashiwara2018laurent}.
Conjecturally, this family includes the band basis \cite{thurston2014positive}
of cluster algebras arising from surfaces and the Berenstein-Zelevinsky
acyclic triangular bases \cite{BerensteinZelevinsky2012}\cite{qin2019compare}.

\item The theta basis in the sense of \cite{gross2018canonical}:
It consists of the ``theta functions'' appearing in the associated
scattering diagram. It turns out to be a basis for injective-reachable
upper cluster algebras \cite{gross2018canonical}. This family includes
the greedy bases of cluster algebras of rank $2$ \cite{lee2014greedy}\cite{cheung2015greedy}.
For cluster algebras arising
from surfaces \cite{musiker2013bases}, the bracelet bases in the sense of \cite{musiker2013bases} are conjectured to be the theta bases. This conjecture will be verified in an upcoming work by Travis Mandel and the author \cite{mandel_surface}.

\end{enumerate}

The bases as listed above appear naturally from their own backgrounds\footnote{The common triangular basis is related to the (dual) canonical basis
in representation theory, which is often thought to be the best basis
for quantized enveloping algebras. The theta basis was also said to
be ``canonical'' in the original paper \cite{gross2018canonical} and is very natural from a geometric point of view.}. They are always parametrized by the tropical points and contain
all cluster monomials \cite{plamondon2013generic}\cite{qin2017triangular}\cite{gross2018canonical}.
But such good bases are known to be different even in easy toy models
\cite{ShermanZelevinsky04}. This surprising phenomenon is the main
motivation of this paper. Given there exist different bases parametrized
by the tropical points (verifying the Fock-Goncharov conjecture),
the following question arises naturally.

\begin{Quest}

How many bases are parametrized by the tropical points? How similar
and how different are they?

\end{Quest}

We shall give an answer for injective-reachable upper cluster algebras under the full rank assumption (see Remark \ref{rem:full_rank_assumption}).

\subsection{Main results and comments}\label{sec:main_results}

Let there be given a set of vertices $I$ and a partition $I=I_{\ufv}\sqcup I_{\fv}$
into unfrozen vertices and frozen vertices. A seed $t$ is a collection
$((b_{ij})_{i,j\in I},(x_{i})_{i\in I})$, where $(b_{ij})$ is a
skew-symmetrizable matrix and $x_{i}$ the cluster variables in $t$
(distinguished generators of $\clAlg$). Throughout the paper, we often make the following assumption, see Remark \ref{rem:full_rank_assumption}.
\begin{Assumption*}[Full rank assumption]
We assume $\tB(t):=(b_{ik})_{i\in I,k\in I_{\ufv}}$
to be of full rank.
\end{Assumption*}

We will work with a base ring $\kk$, which will be $\kk=\Z$ for classical (upper) cluster algebras and $\kk=\Z[q^{\pm\Hf}]$ for the quantum case, where $q$ is a formal quantum parameter.

We have the lattice of Laurent multidegrees $\Mc(t)\simeq\Z^{I}$
with the natural basis $f_{i}$, the Laurent polynomial ring $\LP(t)=\kk[x_{i}{}^{\pm}]=\kk[\Mc(t)]$,
where where $x^{f_{i}}:=x_{i}$, and the (skew-)field of fractions $\cF(t)$ (see Section \ref{sec:quantization} for the quantum case). In
\cite{qin2017triangular}, the author introduced the dominance order
$\preceq_{t}$ on $\Mc(t)$ such that $g'\preceq_{t}g$ if and only
if $g'=g+\tB(t)\cdot n$ for some $n\in\N^{I_{\ufv}}$. 

On the one hand, for any unfrozen vertex $k\in I_{\ufv}$, there is
an algorithm $\mu_{k}$ called mutation which generates a new seed
$t'=\mu_{k}(t)$ from $t$. We use $\Delta^{+}$ to denote the set
of seeds obtained by repeatedly applying mutations. In addition, there
is a corresponding isomorphism between the (skew-)fields $\mu_{k}^{*}:\cF(t')\simeq\cF(t)$.
We naturally extend these notions for seeds $t'=\overleftarrow{\mu}t$
related by a sequence of mutations $\overleftarrow{\mu}$. Recall
that the upper cluster algebra $\upClAlg$ equals $\cap_{t\in\Delta^{+}}\LP(t)$
where the fraction fields are identified.

On the other hand, on the tropical part, one has a tropical transformation
(piecewise linear map) $\phi_{t',t}:\Mc(t)\simeq\Mc(t')$. By
identifying Laurent degrees $g\in\Mc(t)$ for all seeds $t\in\Delta^{+}$
via the tropical transformations, we define the set of tropical points\footnote{We remark that $\tropSet$ should not be confused with the fixed abstract lattice $\Mc$ used in \cite{gross2013birational}.  The set $\tropSet$ in our paper is viewed as the set of equivalent classes of Laurent degree lattices. In particular, it does not have an additive structure.}
$\tropSet$ to be the set of the equivalent classes $[g]$. $\tropSet$
is equipped with many dominance orders $\preceq_{t}$ by comparing
the representatives in each seed $t$. Given any set of seeds $S$
and any tropical point $[g]\in\tropSet$, dominance orders cut out a
subset of tropical points $\tropSet_{\preceq_{S}[g]}=\{[g']|[g']\preceq_{t}[g]\ \forall t\in S\}$ 

We say a Laurent polynomial $z\in\LP(t)$ is pointed at degree $\deg^{t}z=g\in\Mc(t)$
(resp. copointed at codegree $\codeg^{t}z=g\in\Mc(t)$) if it has
a unique $\preceq_{t}$-maximal (resp. $\preceq_{t}$-minimal) Laurent
monomial with degree $g$ and coefficient $1$. We say $z\in\upClAlg$
is pointed at the tropical point $[g]$ if it is pointed at the representatives
of $[g]$ at all seeds $t\in\Delta^{+}$.

In this work, we restrict our attention to injective-reachable seeds $t$,
which means that there is a seed $t[-1]$ such that, for some permutation
$\sigma$ of $I_{\ufv}$, the cluster variables $x_{i}(t)$ have degree
$\deg^{t[-1]}(x_{i}(t))=-f_{\sigma(i)}$ modulo the frozen part $\Z^{I_{\fv}}$.

\subsubsection*{All bases}

Our first main result is a description of all bases parametrized by the
tropical points.

\begin{Thm}\label{thm:all_pointed_bases}
Consider the classical case $\kk=\Z$.
Let there be given an upper cluster algebra $\upClAlg$ with injective-reachable
seeds $t=\overleftarrow{\mu}t[-1]$ subject to the full rank assumption.

(1) For any collection $\mathcal{S}=\{s_{[g]}\in\upClAlg|[g]\in\tropSet\}$
such that $s_{[g]}$ are pointed at the tropical points $[g]$, $\mathcal{S}$
must be a $\kk$-basis of $\upClAlg$ containing all cluster monomials. 

(2) There exists at least one such basis, which we choose and denote by $\mathcal{Z}=\{z_{[g]}\}$.

(3) The set of all such bases $\mathcal{S}$ is parametrized as follows:

\begin{eqnarray*}
\prod_{[g]\in\tropSet}\kk^{\tropSet_{\prec_{\Delta^{+}}[g]}} & \simeq & \{\mathcal{S}\}\\
((b_{[g],[g']})_{[g']\in\tropSet_{\prec_{\Delta^{+}}[g]}})_{[g]\in\tropSet} & \mapsto & \mathcal{S}=\{s_{[g]}|[g]\in\tropSet\}
\end{eqnarray*}
such that $s_{[g]}=z_{[g]}+\sum_{[g']\in\tropSet_{\prec_{\Delta^{+}}[g]}}b_{[g],[g']}z_{[g']}$. In addition,  $\tropSet_{\prec_{\Delta^{+}}[g]}$
are finite sets. 

\end{Thm}
By this result, the three families of good bases in previous literature correspond to three points in this (infinite) ``moduli space" of bases. The quantum analog of Theorem \ref{thm:all_pointed_bases} is discussed in Section \ref{sec:quantum_bases}. See also Remark \ref{rem:frozen_shift_basis} for bases that factor through frozen variables.

\begin{Rem}[Deformation factors]
The main theorem shows that the set of bases $\{\mathcal{S}\}$ has
a linear structure similar to that of the solution space of a non-homogeneous
linear system, and a general basis could be obtained from a special
one by linear deformation controlled by the factors $\tropSet_{\prec_{\Delta^{+}}[g]}$, which we call the deformation factors.

These deformation factors are new mathematical objects, and further questions arise naturally, see
Section \ref{sec:Deformation-factors}. In particular, Conjecture \ref{conj:rigid_deformation_factor} would imply the open orbit conjecture for unipotent subgroups (see \cite{GeissLeclercSchroeer10}), see Remark \ref{rem:open_orbit_conjecture}.
\end{Rem}

In practice, instead of using the set $\tropSet_{\prec_{\Delta^{+}}[g]}$, it would be easier to work with the larger finite sets $\tropSet_{\prec_{\{t,t[-1]\}}[g]}$. These larger sets can be easily controlled
by computing the difference between the degrees and codegrees (called
support dimensions, or $f$-vectors following \cite{fujiwara2018duality})
(Proposition \ref{prop:bipointed_support}). Correspondingly, in Theorem
\ref{thm:bipointed_bases}, we describe the bases subject to the
weaker condition: we require the basis elements to be compatibly pointed at the seeds $t,t[-1]$ rather than compatibly pointed at all seeds (see Definition \ref{def:compatably_pointed}).

Next, we discuss how to choose one such basis for Theorem \ref{thm:all_pointed_bases}.

\subsubsection*{Generic bases}

Assume that the seeds are skew-symmetric, i.e. their matrices are skew-symmetric. It is naturally expected that the generic cluster characters give rise to bases of many (upper) cluster algebra, called the generic bases. Though the existence of such bases was verified in limited cases, such as \cite{GeissLeclercSchroeer10b}. 

Our second main result gives the existence of the generic basis at a high level of generality, which provides a good choice for the special basis $\mathcal{Z}$ in Theorem \ref{thm:all_pointed_bases}.

\begin{Thm}[Generic basis]\label{thm:generic_basis_existence}
Consider the classical case $\kk=\Z$. Let there be given a skew-symmetric injective-reachable seed $t$ subject to the full rank assumption. Then the set of the localized generic cluster
characters is a basis of $\upClAlg$, called the generic basis.

\end{Thm}
Theorem \ref{thm:generic_basis_existence} is a consequence of Theorem \ref{thm:tropical_finite_decomposition}. The latter result is a general criterion of independent interest, which states that if a collection of elements have well-behaved degrees under mutations, then they form a basis.

We refer the reader to Section \ref{sec:results_proofs} ans Section
\ref{sec:Other-applications-and} for more precise statements, generalization
and more details. Our results apply to (almost) all well-known
cluster algebras arising from representation theory or higher Teichm\"uller
theory, see Remark \ref{rem:injective_reachable_assumption}. Note that a change of coefficients will be needed for punctured surfaces, see Remark \ref{rem:full_rank_assumption}.

In particular, we obtain the existence of the generic basis with high
generality, covering all previously known cases such as \cite{GeissLeclercSchroeer10b}. This result will be used by an upcoming work \cite{GLSsurface} for studying generic bases of cluster algebras arising from surfaces.

\subsubsection*{Theta bases}

For general seeds, a good choice for the special basis $\mathcal{Z}$ in Theorem \ref{thm:all_pointed_bases} would be the theta
basis \cite{gross2018canonical}(see
Section \ref{sec:Theta-functions}).

Now, assume the seeds to be skew-symmetric again. Our last result states that Bridgeland's
representation theoretic formula for many theta functions is effective
(called weak genteelness, see Section \ref{sec:Weakly-genteelness}), which
can be viewed as a pleasant property predicted by Nagao's work \cite{Nagao10}. 

\begin{Thm}[Weak genteelness]\label{thm:genteel}
Take $\kk=\Z$. Let there be given a skew-symmetric injective-reachable seed $t$.
Then Bridgeland's representation theoretic formula is effective for
theta functions in the cluster scattering diagram. Moreover, the stability
scattering diagram and the cluster scattering diagram are equivalent.

\end{Thm}

\begin{Rem}\label{rem:update_DM}
\cite{davison2019strong} appeared soon after this work. Its results allow us to further understand and strengthen the present work.

First, an explicit topology was constructed for the Laurent polynomial ring $\LP(t)$ in \cite[Section 2.2.2]{davison2019strong}, which generalized the natural adic-topology that we will use for seeds of principal coefficients in Section \ref{sec:change_seed}. We omit the details but point out that, in view of this topology, in Definition-Lemma \ref{def:dominance_order_decomposition}, 
the dominance order decomposition is convergent and the pointed set $\cS$ is a topological basis.

Second but most importantly, for any skew-symmetric seed under the full rank assumption, \cite{davison2019strong} constructed the quantum theta functions with strong properties. In particular, when the seed is injective-reachable, such functions form the quantum theta basis for the quantum upper cluster algebra. The existence of such a basis is crucial for describing more quantum bases, see Section \ref{sec:quantum_bases}.
\end{Rem}

\begin{Rem}[Full rank assumption]\label{rem:full_rank_assumption}
It is worth noting that, if an initial seed $t_0$ satisfies the full rank assumption, so do all the seeds obtained from $t_0$ by iterated mutations, see \cite[Theorem 3.1.2]{muller2015existence}\cite[Proposition 5.1.4]{qin2017triangular}.

But the full rank assumption does not hold true for an arbitrary seed $t=((b_{ij})_{i,j\in I}, (x_i))$. Nevertheless, for studying many questions in cluster theory, one has the freedom to change the coefficients so that the assumption becomes true (i.e. change the set of frozen vertices $I_{\fv}$ and the matrix $(b_{ij})$ but keeping the principal part $(b_{ij})_{i,j\in I_{\ufv}}$ unchanged).

A change of coefficients is justified by keeping important structures in cluster theory. For example, the exchange graphs remain the same, see \cite[Proposition 3]{cao2020enough}. Moreover, if one knows the cluster expansion of cluster variables for some coefficients under the assumption, then one can deduce the cluster expansion for all coefficients \cite[Section 3]{FominZelevinsky07}. 

Similarly, if a (quantum) cluster algebra subject to the assumption possesses a good basis (as in Remark \ref{rem:frozen_shift_basis}), one can construct a spanning set for the corresponding algebra with arbitrary coefficients, using the correction technique for pointed elements (\cite[Section 9]{Qin12} \cite[Section 4]{qin2017triangular}). Moreover, under the full rank assumption, or the weaker assumption that $\tB(t)\R_{\geq 0}^{I_{\ufv}}$ is strictly convex (as used in \cite{GLSsurface}), the spanning set is again a basis. 

It is natural to ask whether the spanning set constructed above is always a basis for all choice of coefficients. But, at this moment, very little is known about bases of (upper) cluster algebras without the full rank assumption or the convexity assumption above. A progress made towards this direction was due to \cite{irelli2013linear}, which showed that the set of cluster monomials (usually a proper subset of the basis) is linearly independent.

Finally, a seed can be quantized if and only if the the assumption holds. 
Except for punctured surfaces, the well-known cluster algebras listed in Remark \ref{rem:injective_reachable_assumption} admit natural quantization and satisfy the full rank assumption. When the assumption fails, we have to choose appropriate coefficients so that the assumption becomes true, a quantization can be performed, and our results about bases become effective.
\end{Rem}

\begin{Rem}[Injective-reachable assumption]\label{rem:injective_reachable_assumption}
For deriving the main results of this paper, the injective-reachable assumption is imposed. 

This assumption implies that the associated Jacobian algebra is finitely-dimensional. The converse is not necessarily true. A counter example arising from once-punctured torus was studied in \cite[Example 4.3]{plamondon2013generic}.

The injective-reachable assumption are satisfied by the following well-known cluster algebras:
\begin{itemize}
\item coordinate rings of unipotent cells \cite{GeissLeclercSchroeer10} \cite{GeissLeclercSchroeer11}, see \cite[Section 13]{GeissLeclercSchroeer10};
\item level$-l$ categories of representations of quantum affine algebras \cite{HernandezLeclerc09}, see \cite[(52)]{qin2017triangular};
\item symmetric CGL extensions (including double Bruhat cells \cite{BerensteinFominZelevinsky05} \cite{goodearl2020berenstein}), see \cite[Main theorem III]{GY13};
\item equivariant perverse coherent
sheaves over affine Grassmannians, see \cite[Theorem 3.1, Proposition 6.2]{cautis2019cluster};
\item cluster algebras over marked surfaces (except once-punctured closed surfaces) \cite{FockGoncharov06a} \cite{FominShapiroThurston08}, see \cite[Proposition 7.10]{FominShapiroThurston08};
\item $PGL_m$ (or $SL_m$) local systems on marked surfaces (except once-punctured closed surfaces) \cite{goncharov2018donaldson} \cite{Goncharov2019localsystem}, see \cite[Theorem 1.2]{goncharov2018donaldson}.
\end{itemize}
\end{Rem}

\subsubsection*{Key points in the proofs}

As an important part of the paper, we give a systematic analysis of the tropical properties of upper cluster algebra elements, by which we mean how their degrees and codegrees change under mutations. More precisely, we introduce notions such as codegrees and support dimensions (Definitions \ref{def:copointed} \ref{def:f_dim} \ref{def:support_dimension}). We also introduce the linear map $\psi_{t[-1],t}:\Mc(t)\rightarrow\Mc(t[-1])$, which reverses the dominance orders and swap degrees and codegrees at different seeds $t$, $t[-1]$ (Definition \ref{def:deg_transform} Propositions \ref{prop:order_reverse} \ref{prop:cpt_pt_swap}). Then we derive the equivalence between being compatibly
pointed at $t,t[-1]$ (i.e., degrees are controlled by tropical transformations)
and being bipointed at $t$ with the ``correct'' support dimension
(Proposition \ref{prop:bipointed_support}). We arrive at the following interesting observation.
\begin{Lem}[Lemma \ref{lem:compatible_at_g_vector}]\label{lem:intro_compatible_at_g_vector}
If an upper cluster algebra element $Z$ and a cluster monomial $M$ share the same tropical property, then they are the same.
\end{Lem}

The parametrization of the set of the bases (Theorem \ref{thm:all_pointed_bases}(2)) is an application of the above analysis.

As another important part of the paper, we propose and prove a criterion for a given collection of elements of an upper cluster algebra to become a basis (Theorem \ref{thm:tropical_finite_decomposition}), which says that the collection suffices to have good tropical properties. This criterion immediately implies Theorem \ref{thm:all_pointed_bases}(1) as well as Theorem \ref{thm:generic_basis_existence}, since the generic cluster characters are known to have good tropical properties \cite[Theorem 1.3]{plamondon2013generic}. 

The criterion is proved by introducing and analyzing the dominance order decomposition into the given collection (Definition-Lemma \ref{def:dominance_order_decomposition}). A priori, the (possibly infinite) decomposition depends on the chosen seed. We first show that the decomposition is independent of the chosen seed (Proposition \ref{prop:invariant_decomposition}), whose proof is based on natural adic-topologies induced by principal coefficients in the sense of \cite{FominZelevinsky07}, and an application of the nilpotent Nakayama Lemma (we learn the usefulness
of the nilpotent Nakayama Lemma from the inspirational work \cite{gross2018canonical}). Then we show that the decomposition is finite by using the injective-reachability condition and conclude that the given collection is a basis.

Finally, we give a quick proof of Theorem \ref{thm:genteel} based on cluster theory and the trick of constructing opposite scattering diagrams.

\begin{Rem}
The analysis of tropical properties in this paper turn out to be very useful in upcoming works \cite{qin2020dual} \cite{qin2020analog}. In particular, the dominance order decomposition will be used in \cite{qin2020dual}, and the codegrees will be used in \cite{qin2020analog}.
\end{Rem}

\subsection{Contents}

Section \ref{sec:Preliminaries} contains necessary preliminaries.
A reader could skip the details and the content familiar to him/her.
But it is still recommended to read Section \ref{sec:cluster_notation}
which merges symbols and notions of cluster algebras of two different
styles \cite{FominZelevinsky02}\cite{gross2018canonical}. In addition, we verify the equivalence between injective-reachability and the existence of green to red sequences.

In Section \ref{sec:bidegrees_and_support}, we define and study degrees,
codegrees and support. These are the main tools that will be used
in this paper, which we develop by elementary manipulation on Laurent
polynomials/series.

In Section \ref{sec:Properties-of--decompositions}, we study properties
of the $\prec_{t}$-decompositions based on Section \ref{sec:bidegrees_and_support}
and the nilpotent Nakayama Lemma. This section provides direct proofs
for Theorem \ref{thm:all_pointed_bases}(1) and Theorem \ref{thm:generic_basis_existence}.

In Section \ref{sec:results_proofs}, we present the main results,
consequences and the proofs based on Sections \ref{sec:bidegrees_and_support}
\ref{sec:Properties-of--decompositions}.

In Section \ref{sec:Other-applications-and}, we discuss related topics
such as deformation factors, quantized version of our results, a representation theoretic formula for the theta functions (weak genteelness), and the bases for partially
compactification cases.

In Section \ref{sec:appendix_scattering_diagram},
we briefly review some content in \cite{gross2018canonical} about
scattering diagrams and theta functions. Then we present two proofs for the weak genteelness (Theorem \ref{thm:genteel}). One is conceptual
following Nagao \cite{Nagao10}. Another one uses the construction of an
opposite scattering diagram. This section is independent from most part of
the paper, but provides definitions and properties for the theta
functions.

\section*{Acknowledgments}

The author thanks Bernhard Keller and Pierre-Guy Plamondon for inspiring
discussion on representation theory. He also thanks Mark Gross, Tom
Bridgeland, Alfredo N\'ajera Ch\'avez, Dylan Rupel, Man-Wai Cheung,
Tim Magee, Christof Geiss, Daniel Labardini Fragoso and Salvatore
Stella for interesting discussion. He is grateful to Bernhard Keller
for reading a preliminary version of the paper. He thanks Travis Mandel
for interesting discussion and remarks. He thanks the anonymous referee for comments and the observation in Remark \ref{rem:monodromy}.

\section{Preliminaries\label{sec:Preliminaries}}

\subsection{Basics of cluster mutations and tropicalization\label{sec:cluster_notation}}

Throughout this paper, we shall consider cluster algebras with geometric
coefficients in the sense of \cite{FominZelevinsky07}. The cluster
algebra we defined is the same as in \cite{FominZelevinsky07}, following
the nice presentation of \cite{gekhtman2017hamiltonian}. Furthermore,
our convention is compatible with the different formalism \cite{gekhtman2017hamiltonian}\cite{gross2013birational},
so that we can easily use results and arguments form these works.

We will work with a base ring $\kk$. We usually take $\kk=\Z$ for classical (upper) cluster algebras and $\kk=\Z[q^{\pm\Hf}]$ for quantum (upper) cluster algebras, where $q^{\Hf}$ is a formal quantum parameter. Unless otherwise specified, our arguments will be equally effective for both the classical and quantum case.

\subsubsection*{Seeds and $B$-matrices}

Given the set of vertices $I=I_{\ufv}\sqcup I_{\fv}$. The vertices
in $I_{\ufv}$ and $I_{\fv}$ are called unfrozen and frozen respectively.
Suppose that there is a collection of integers $d_{i}>0$, and a matrix
$(b_{ij})_{i,j\in I}$ such that $b_{ij}\in\begin{cases}
\Q & i,j\in I_{\fv}\\
\Z & \mathrm{else}
\end{cases}$, $b_{ij}d_{j}=-b_{ji}d_{i}$.

\begin{Def}\label{def:seed}

A seed $t$ is a collection $((b_{ij}(t))_{i,j\in I},(x_{i}(t))_{i\in I},d_{i},I,I_{\ufv})$
with $x_{i}(t)$ indeterminate. The matrix $\tB(t):=(b_{ik}(t))_{i\in I,k\in I_{\ufv}}$
is called the $B$-matrix associated to $t$ and $x_{i}(t)$ the cluster
variables.

\end{Def}

For any $m=(m_{i})\in\N^{I_{\ufv}}\oplus\Z^{I_{\fv}}$, we call $x(t)^{m}:=\prod_{i\in I}x_{i}(t)^{m_{i}}$
a \emph{(localized) cluster monomial} in the seed $t$.

We usually fix $d_{i}$ and $I_{\ufv}\subset I$, and denote $t=((b_{ij}(t)),(x_{i}(t)))$
for simplicity. The symbol $t$ will be omitted when the context is
clear.

Let $d$ denote the least common multiple of $(d_{i})_{i\in I}$
and define the Langlands dual $d_{i}^{\vee}:=\frac{d}{d_{i}}$. Then
$d_{i}^{\vee}b_{ij}=-d_{j}^{\vee}b_{ji}$, and we say $(b_{ij})$
is \emph{skew-symmetrizable} by the diagonal matrix $\mathrm{diag}(d_{i}^{\vee})$.
It follows that the principal part $B:=(b_{ij})_{i,j\in I_{\ufv}}$
of $(b_{ij})$ is skew-symmetrizable as well. 

Conversely, suppose that we are given an $I\times I_{\ufv}$-integer
matrix $\tB=(b_{ij})_{i\in I,j\in I_{\ufv}}$ with principal part
$B$, such that $B$ is skew-symmetrizable by some diagonal matrix
$D=\mathrm{diag}(\diag_{k})_{k\in I_{\ufv}}$, $\diag_{k}\in\Z_{>0}$.
We can do the following extension.

\begin{Lem}\label{lem:matrix_extension}

We can find strictly positive integers $\diag_{f}$, $f\in I_{\fv}$,
and extend the matrix $\tB(t)$ to an $I\times I$ integer matrix
$(b_{ij}(t))$, such that $\diag_{i}b_{ij}(t)=-\diag_{j}b_{ji}(t)$. 

\end{Lem}

\begin{proof}

Let $\diag$ denote the least common multiple of $(\diag_{k})_{k\in I_{\ufv}}$.
We might choose $\diag_{f}=\diag,$ $b_{kf}(t)=-\frac{\diag}{\diag_{k}}b_{fk}(t)$,
$b_{ff'}=0$, $\forall f,f'\in I_{\fv},k\in I_{\ufv}$.

\end{proof}

Recall that, a seed by Fomin-Zelevinsky \cite{FominZelevinsky02}
takes the form $(\tB,(x_{i}))$ with a skew-symmetrizable principal
part $B$. By Lemma \ref{lem:matrix_extension}, their seed could
be extended to our seed by choose a matrix extension. The extra data
in our definition arise from the construction in \cite{FockGoncharov09} \cite{gross2013birational} \cite{gross2018canonical}.

We say the seed $t$ is skew-symmetrizable (resp. skew-symmetric)
if the matrix $(b_{ij}(t))$ is.

\subsubsection*{Lattices and $\epsilon$-matrices}

Following \cite{gross2013birational}\cite{gross2018canonical}, let
$\Mc(t)$ denote a lattice with a $\Z$-basis $\{f_{i}(t)|i\in I\}$
and $N(t)$ a lattice with a $\Z$-basis $\{e_{i}(t)|i\in I_{\ufv}\}$.
Define the pairing $\langle\ ,\ \rangle$ between $\Mc(t)$ and
$N(t)$ such that $\langle f_{i}(t),e_{j}(t)\rangle=\frac{1}{d_{i}}\delta_{ij}$.
Let $\yCone(t)$ denote the sublattice of $N(t)$ generated by $\{e_{k}(t)|k\in I_{\ufv}\}$.

Define the $\Q$-valued bilinear form $\{\ ,\ \}$ on $N(t)$ such
that $b_{ij}=\{e_{j}(t),e_{i}(t)\}d_{i}$. It turns out that $\{\ ,\ \}$
is skew-symmetric. 

\begin{Def}

The $\epsilon$-matrix is defined to be $$(\epsilon_{ij})_{i,j\in I}=(\{e_{i}(t),e_{j}(t)\}d_{j})_{i,j\in I}.$$

\end{Def}

Let $p^{*}$ denote the linear map from $N(t)$ to $\Mc(t)\otimes \Q$
such that 
\begin{align*}
p^{*}(n)=\{n,\ \}.
\end{align*}
Denote $v_{k}(t)=p^{*}(e_{k}(t))=\{e_{k}(t),\ \}$ for $k\in I_{\ufv}$.
It turns out that $v_{k}(t)=\sum_{i\in I}b_{ik}f_{i}(t)\in \Mc(t)$. 

We always assume that $p^{*}|_{\yCone(t)}$ is injective throughout this paper, or, equivalently, $\tB(t)$ satisfies the full rank assumption.

Let us consider the group ring (of characters) $\LP(t)=\kk[\Mc(t)]=\kk[\chi^{m}]_{m\in\Mc(t)}$
and the group ring (of cocharacters) $\kk[N(t)]=\kk[\lambda^{n}]_{n\in N(t)}$. We denote the $x$-variables $x_{i}(t)=\chi^{f_{i}(t)}$, Laurent monomials $x(t)^m=\chi ^m$, and the
$y$-variables $y_{i}(t)=\lambda^{e_{i}(t)}$. Similarly,
we can define $\bLP(t)=\kk[x_{f}(t)]_{f\in I_{\fv}}[x_{i}(t)^{\pm}]_{i\in I_{\ufv}}$
and call it the (partially) compactified Laurent polynomial ring.

The commutative product in $\LP(t)$ will be denoted by $\cdot$ or omitted for simplicity. For the quantum case ($\kk=\Z[q^{\pm\Hf}]$), we also define the twisted product $*$, see Section \ref{sec:quantization}.

Note that, for $\kk=\Z$, $\LP(t)\otimes\C$
is the ring of regular functions on the split algebraic torus $(\C^{*})^{I}$. And $\bLP(t)\otimes\C$ is the ring of regular functions
on the partial compactification $(\C^{*})^{I_{\ufv}}\times(\C)^{I_{\fv}}$
of $(\C^{*})^{I}$.

\subsubsection*{Mutations}

Let $[\ ]_{+}$ denote $\max(\ ,0)$ and define $[(g_{i})_{i\in I}]_{+}=([g_{i}]_{+})_{i\in I}$
for any vector $(g_{i})_{i\in I}$. For any $k\in I_{\ufv}$, we can
define a seed $t'=\mu_{k}t$ by the following procedure. 

We start by choosing a sign $\varepsilon\in\{+,-\}$, define the $I\times I$
matrix $\tE_{\varepsilon}$ and $I\times I$-matrix $\tF_{\varepsilon}$
such that
\begin{align*}
(\tF_{\varepsilon})_{ij} & =\begin{cases}
\delta_{ij} & k\notin\{i,j\}\\
-1 & i=j=k\\{}
[\varepsilon b_{kj}]_{+} & i=k,\ j\neq k
\end{cases}
\end{align*}

\begin{align*}
(\tE_{\varepsilon})_{ij} & =\begin{cases}
\delta_{ij} & k\notin\{i,j\}\\
-1 & i=j=k\\{}
[-\varepsilon b_{ik}]_{+} & i\neq k,\ j=k
\end{cases}.
\end{align*}
Notice that $\tF_{\varepsilon}^{2}=\Id_{I_{\ufv}}$ and $\tE_{\varepsilon}^{2}=\Id_{I}$.
The $I_{\ufv}\times I_{\ufv}$-submatrix of $\tE_{\varepsilon}$ (principal
part) is denoted by $E_{\varepsilon}$ and the $I_{\ufv}\times I_{\ufv}$-submatrix
of $\tF_{\varepsilon}$ denoted by $F_{\varepsilon}$.

Next, define a lattice $\Mc(t')$ with a basis $\{f_{i}'=f_{i}(t')\}_{i\in I}$
and a lattice $N(t')$ with a basis $\{e_{i}'=e_{i}(t')\}_{i\in I}$,
where we omit the symbol $t$ from now on. We define linear isomorphisms
$\tau_{k,\varepsilon}:\Mc(t')\rightarrow\Mc(t)$ and $\tau_{k,\varepsilon}:N(t')\rightarrow N(t)$
such that $\tau_{k,\varepsilon}(e_{i}')=\sum_{j\in I}e_{j}\cdot(\tF_{\varepsilon})_{ji}$
and $\tau_{k,\varepsilon}(f_{i}')=\sum_{j\in I}f_{j}\cdot(\tE_{\varepsilon})_{ji}$,
namely,

\begin{align}
\tau_{k,\varepsilon}(e_{i}') & =\begin{cases}
e_{i}+[\varepsilon b_{ki}]_{+}e_{k} & i\neq k\\
-e_{k} & i=k
\end{cases}\label{eq:tropical_e}
\end{align}

\begin{align}
\tau_{k,\varepsilon}(f_{i}') & =\begin{cases}
f_{i} & i\neq k\\
-f_{k}+\sum_{j}[-\varepsilon b_{jk}]_{+}f_{j} & i=k
\end{cases}.\label{eq:tropical_f}
\end{align}
Clearly, $\tau_{k,\epsilon}$ preserves the pairing $\langle\ ,\ \rangle$.
Further define the bilinear form $\{\ ,\ \}$ on $N(t')$ as induced
by that on $N(t)$ via $\tau_{k,\varepsilon}$. It is straightforward
to check that the corresponding matrix $(b_{ij}')_{i,j\in I}=(\{e_{j}',e_{i}'\}d_{i})_{i,j\in I}$
satisfies

\begin{align*}
b_{ij}' & =\begin{cases}
-b_{ij} & k\in\{i,j\}\\
b_{ij}+b_{ik}[\varepsilon b_{kj}]_{+}+[-\varepsilon b_{ik}]_{+}b_{kj} & k\neq i,j
\end{cases}.
\end{align*}
 Notice that $b_{ij}'$ are independent of the choice of the sign
$\varepsilon$.

We define the mutated seed $t'=\mu_{k}t$ as $((b_{ij}')_{i,j\in I},(x_{i}')_{i\in I})$.
Let us relate the cluster variables $x_{i}$ and $x_{i}'$ now.

First, the maps $\tau_{k,\varepsilon}$ induce isomorphisms between
Laurent polynomials rings, which are still denoted by $\tau_{k,\varepsilon}$,
such that

\begin{align*}
\tau_{k,\varepsilon}(y_{i}') & =\begin{cases}
y_{i}y_{k}^{[\varepsilon b_{ki}]_{+}} & i\neq k\\
y_{k}^{-1} & i=k
\end{cases}
\end{align*}

\begin{align*}
\tau_{k,\varepsilon}(x_{i}') & =\begin{cases}
x_{i} & i\neq k\\
x_{k}^{-1}\prod_{j}x_{j}^{[-\varepsilon b_{jk}]_{+}} & i=k
\end{cases}
\end{align*}

Now consider the classical case $\kk=\Z$ for simplicity (see Section \ref{sec:quantization} for the quantum case). Define the automorphisms $\rho_{k,\varepsilon}$ on the fraction
fields $\cF(t)=\cF(\LP(t))$ and $\cF(\kk[N(t)])$ respectively, such
that
\begin{align}\label{eq:rho}
\begin{split}
\rho_{k,\varepsilon}(x_{i}) & =\begin{cases}
x_{i} & i\neq k\\
x_{k}(1+x^{\varepsilon v_{k}})^{-1} & i=k
\end{cases}\\
\rho_{k,\varepsilon}(y_{i}) & =\begin{cases}
y_{i}(1+y_{k}^{\varepsilon})^{-b_{ki}} & i\neq k\\
y_{k} & i=k
\end{cases}
\end{split}
\end{align}
Then it turns out that
\begin{align}\label{eq:mutation}
\begin{split}
\rho_{k,\varepsilon}\circ\tau_{k,\varepsilon}(x_{i}') & =\begin{cases}
x_{i} & i\neq k\\
x_{k}^{-1}\prod_{j}x_{j}^{[-\varepsilon b_{jk}]_{+}}(1+\chi^{\varepsilon v_{k}}) & i=k
\end{cases}\\
\rho_{k,\varepsilon}\circ\tau_{k,\varepsilon}(y_{i}') & =\begin{cases}
y_{i}y_{k}^{[\varepsilon b_{ki}]_{+}}(1+y_{k}^{\varepsilon})^{-b_{ki}} & i\neq k\\
y_{k}^{-1} & i=k
\end{cases}
\end{split}
\end{align}
We observe that the compositions $\rho_{k,\varepsilon}\circ\tau_{k,\varepsilon}$
are independent of the choice of $\varepsilon$. Let us call them the
\emph{mutation birational maps}, which are denoted by $\mu_{k}^{*}$.
The maps $\tau_{k,\varepsilon}$ is called their \emph{monomial parts
}and $\rho_{k,\varepsilon}$ their \emph{Hamiltonian parts}. One can
show that the $\mu_{k}^{*}$ give isomorphisms between the fraction
fields $\cF(t')\simeq\cF(t)$ and between the fraction fields $\cF(\kk[N(t')])\simeq\cF(\kk[N(t)])$
respectively.

Given any two seeds $t,t'$ such that $t'=\overleftarrow{\mu}t$ for
some mutation sequence $\overleftarrow{\mu}$. Let $\overleftarrow{\mu}^{*}$
denote the mutation maps from the fraction field $\cF(t')$ to $\cF(t)$
defined by composing the corresponding mutation maps. Then we can
denote $\LP(t)\cap\LP(t')=\LP(t)\cap(\overleftarrow{\mu}^{*}\LP(t'))$
and also $\LP(t)\cap\LP(t')=(\overleftarrow{\mu}^{-1})^{*}\LP(t)\cap\LP(t')$.
Correspondingly, for any $z\in(\overleftarrow{\mu}^{-1})^{*}\LP(t)\cap\LP(t')$,
the Laurent polynomial $\overleftarrow{\mu}^{*}z\in\LP(t)$ is sometimes
also denoted by $z$ for simplicity.

\subsubsection*{y-variables}

Because $p^{*}$ is linear and $\tau_{k,\epsilon}$ preserves $\{\ ,\ \}$
and $\langle\ ,\ \rangle$, we have $\tau_{k,\varepsilon}(v_{i}')=\begin{cases}
v_{i}+[\varepsilon b_{ki}]_{+}v_{k} & i\neq k\\
-v_{k} & i=k
\end{cases}$. One can check that $\mu_{k,\varepsilon}^{*}(\chi^{v_{i}})=\begin{cases}
\chi^{v_{i}}\chi^{[\varepsilon b_{ki}]_{+}v_{k}}(1+\chi^{\varepsilon v_{k}})^{-b_{ki}} & i\neq k\\
\chi^{-v_{k}} & i=k
\end{cases}$, i.e. subject to the law given by \eqref{eq:mutation}. By abuse of notation, we define the Laurent monomial $y_{k}=\chi^{v_{k}}$, which equals $\prod_{i}x_{i}^{b_{ik}}$
in $\LP(t)$ under the commutative product. $y_{k}$ are still called the $y$-variables.

\subsubsection*{Tropicalization}

We refer the reader to \cite{FockGoncharov09}\cite{gross2013birational}\cite{gross2018canonical}
for more details. Recall that $\langle f_{i},e_{j}\rangle=\frac{\delta_{ij}}{d_{i}}$,
$b_{ij}=\{e_{j},e_{i}\}d_{i}$ and $b_{ji}\cdot d_{j}{}^{-1}=-b_{ij}\cdot d_{i}^{-1}$,
$i,j\in I$. Let $M(t)$ denote the sublattice of $\Mc(t)$ with
the basis $\{e_{i}^{*}=d_{i}f_{i}\}$. Let $N^{\circ}(t)$ denote
the sublattice of $N(t)$ with the basis $\{d_{i}e_{i}\}$. Then $M(t)$
is dual to $N(t)$ and $N^{\circ}(t)$ is dual to $\Mc(t)$ under
the pairing $\langle\ ,\ \rangle$ respectively. 

For any lattice $L$ and its dual $L^{*}$, we denote the split algebraic
torus $T_{L}=\Spec\Z[L^{*}]=\Spec\Z[\lambda^{n}]_{n\in L^{*}}$. Let
$P$ be a given semifield $(P,\oplus,\otimes)$ and $P^{\times}$
the multiplicative group. Let $Q_{sf}(L)$ denote the semifield of
subtraction free rational functions on $T_{L}$. A tropical point
in $T_{L}$ is defined to be a semifield homomorphism from $Q_{sf}(L)$
to $P$. The set of tropical points in $T_{L}$ is denoted by $T_{L}(P)$.
One can show that $T_{L}(P)\simeq\Hom_{\mathrm{groups}}(L^{*},P^{\times})\simeq L\otimes_{\Z}P^{\times}$
such that any point $m\otimes_{\Z}p$ sends a subtraction-free Laurent
polynomial $f=\sum_{n}\lambda^{n}\in Q_{sf}(L)$ to $\oplus_{n}p^{\otimes\langle m,n\rangle}\in P$,
see \cite{gross2018canonical}.

We usually work with $P=\Z^{T}=(\Z,\max(\ ,\ ),+)$ or $P=\Z^{t}=(\Z,\min(\ ,\ ),+)$,
in which case $P^{\times}=\Z\backslash\{0\}$ and $T_{L}(P)\simeq L$.
We have $-\max(a,b)=\min(-a,-b)$, $\forall a,b\in\Z$. It follows
that the map $i:\Z^{T}\rightarrow\Z^{t}$ such that $i(a)=-a$ is
an isomorphism between the semifields $\Z^{T}$ and $\Z^{t}$.

We will soon define the Langlands dual seed $t^{\vee}$. By taking
the tropicalization of the corresponding mutation maps on $T_{M(t^{\vee})}\simeq T_{\Mc(t)}$
with the tropical semifield $P=\Z^{T}$ \cite{FockGoncharov09}, we
obtain the following definition.

\begin{Def}[Tropical transformation]\label{def:tropical_transformation}

Let there be given seeds $t'=\mu_{k}t$. The tropical transformation
$\phi_{t',t}:\Mc(t)\rightarrow\Mc(t')$ is the piecewise linear
map such that, for any $g=\sum g_{i}f_{i}\in\Mc(t)$, its image
$g'=\sum g_{i}'f_{i}'=\phi_{t',t}(g)$ is given by

\begin{align*}
g'_{k} & =-g_{k}\\
g'_{i} & =g_{i}+[b_{ik}(t)]_{+}[g_{k}]_{+}-[-b_{ik}(t)]_{+}[-g_{k}]_{+}\ \forall i\neq k
\end{align*}

\end{Def}

For any two seeds $t',t$ related by a mutation sequence $\overleftarrow{\mu}=\mu_{k_{r}}\cdots\mu_{k_{1}}$
such that $t'=\overleftarrow{\mu}t$. Define $\phi_{t',t}$ to be
the composition of the corresponding tropical transformations. Then
it is independent of the choice of $\overleftarrow{\mu}$ because
it is the tropicalization of the mutation maps.

\subsubsection*{Langlands dual}

Let us sketch the construction of the Langlands dual, although we
will not investigate the duality in depth.

Let us define the Langlands dual seed $t^{\vee}=(b_{ij}(t^{\vee}),(x_{i}(t^{\vee}))_{i\in I})$
with strictly positive integers $d_{i}(t^{\vee})=d_{i}^{\vee}=\frac{d}{d_{i}}$.
We define $N(t^{\vee})$ to be the lattice $N^{\circ}(t)$ with basis
$\{e_{i}^{\vee}:=d_{i}e_{i}\}$ endowed with the bilinear form $\{\ ,\ \}^{\vee}$
such that $\{\ ,\ \}^{\vee}=\frac{1}{d}\{\ ,\ \}$, which implies
the definition $b_{ji}(t^{\vee}):=-b_{ij}$. Its dual lattice $M(t^{\vee})$
is then defined to be $M^{\circ}(t)$ spanned by the basis $\{(e_{i}^{\vee})^{*}=(d_{i})^{-1}e_{i}^{*}=f_{i}\}$.
Define $\Mc(t^{\vee})$ to be the lattice spanned by the basis $\{f_{i}^{\vee}:=\frac{1}{d_{i}^{\vee}}(e_{i}^{\vee})^{*}=\frac{1}{d}e_{i}^{*}\}$,
and $N^{\circ}(t^{\vee})$ the lattice spanned by the basis $\{d_{i}^{\vee}e_{i}^{\vee}=de_{i}\}$.

By construction, we have $T_{\Mc(t)}=T_{M(t^{\vee})}$. Moreover,
such identification commutes with the mutations, see \cite[Lemma 1.11]{FockGoncharov09}.

\subsubsection*{Cluster algebras and cluster varieties}

Choose an initial seed $t_{0}$. For any sequence of unfrozen vertices $(k_1,k_2,\ldots,k_r)$, we have a sequence of sign-coherent vectors called $c$-vectors, whose construction is technical and will be postponed to Section \ref{sec:exp-g-c-vector}. Correspondingly, we have a sequence of signs $(\varepsilon_1,\ldots,\varepsilon_r)$ and the corresponding sequence of mutation $\overleftarrow{\mu}=\mu_{k_r,\varepsilon_r}\cdots \mu_{k_1,\varepsilon_1}$ starting from $t_0$ (reading from right to left), see Theorem \ref{thm:cg_vector}. Unless otherwise specified, we always make this canonical choice of signs for mutations, and omit the sign symbols $\varepsilon_1, \ldots,\varepsilon_r$ for simplicity.

Let $\Delta^{+}=\Delta^{+}_{t_0}$ denote
the set of all seeds obtained from the initial seed by iterated mutations (with the canonical choice of signs).
For any $t\in\Delta^{+}$, view its cluster variables $x_{i}(t)$
as element in the (skew-)field of fractions $\LP(t_{0})$ via the mutation
maps. 

In the following, we construct the classical cluster algebras using the commutative product, and the quantum cluster algebras using the twisted product (see Section \ref{sec:quantization}).
\begin{Def}[Cluster algebras]\label{def:cluster_algebra}

We define the (partially) compactified cluster algebra as $\bClAlg=\kk[x_{i}(t)]_{\forall i\in I,t\in\Delta^{+}}$,
and the (localized) cluster algebra as $\clAlg=\bClAlg(t_{0})[x_{f}^{-1}]_{f\in I_{\fv}}$.
We define the (localized) upper cluster algebra as $\upClAlg=\cap_{t\in\Delta^{+}}\LP(t)$,
where Laurent polynomials at different seeds are identified via mutation
maps.

\end{Def}

In this paper, we shall focus our attention on the cluster algebras
$\clAlg$ and upper cluster algebras $\upClAlg$. Let us explain geometric objects associated to $\upClAlg$ with the choice $\kk=\Z$.

\begin{Def}

We define the cluster varieties to be $\AVar=\cup_{t\in\Delta^{+}}T_{N^{\circ}(t)}$
and $\XVar=\cup_{t\in\Delta^{+}}T_{M(t)}$, where the tori are glued
via mutation maps.

\end{Def}

The Fock-Goncharov dual of a variety $V=\cup T_{L}$ is defined as
$V^{\vee}=\cup T_{L^{*}}$. Therefore, the dual of $\AVar$ is given
by $\AVar^{\vee}=\cup_{t\in\Delta^{+}}T_{\Mc(t)}$ where the tori
are glued by mutation maps. Then $\AVar^{\vee}$ agrees with the variety
$\XVar(t_{0}^{\vee})$ associated to the Langlands dual initial seed
$t_{0}^{\vee}$. We observe that the ring of the regular functions
on $\AVar$ is just the upper cluster algebra $\upClAlg$ (with $\kk=\Z$). 

Recall that the gluing map between $T_{\Mc(t)}$ and $T_{\Mc(t')}$
tropicalizes to $\phi_{t,t'}:\Mc(t)\simeq\Mc(t')$. We define
the set of the tropical points $\AVar(\Z^{T})$ to be the set of equivalent
classes in $\sqcup_{t\in\Delta^{+}}\Mc(t)$ under the identifications
$\phi_{t,t'}$, which we also denote by $\tropSet$. The elements in
$\tropSet$ are denoted by $[g]$ for the representatives $g\in\Mc(t)$.

\subsection{Cluster expansions, $c$-vectors and $g$-vectors}
\label{sec:exp-g-c-vector}
Cluster variables have been shown to enjoy the Laurent phenomenon
\cite{FominZelevinsky02}. They can be calculated by the Caldero-Chapoton type expansion formula for the classical case 
\cite{CalderoChapoton06}\cite{DerksenWeymanZelevinsky09} and for the quantum case \cite{tran2011f}\cite{gross2018canonical}. We summarize these
properties as the following using the commutative product.

\begin{Thm}\label{thm:CC-formula}

For any seeds $t=\overleftarrow{\mu}t_{0}\in \Delta^+_{t_0}$ and $i\in I$, we have
$\overleftarrow{\mu}^{*}(x_{i}(t))\in\LP(t_{0})$. Moreover, we have

\begin{align*}
\overleftarrow{\mu}^{*}(x_{i}(t)) & =x(t_{0})^{g_{i}(t)}\cdot(\sum_{n\in\yCone^{\geq0}(t_{0})}c_{n}x(t_{0})^{\tB(t_{0})n})
\end{align*}
where $g_{i}(t)\in\Mc(t_{0})$, coefficient $c_{0}=1$, and all
coefficients $c_{n}\in\kk$.

\end{Thm}

The vector $g_{i}(t)$ is called the $i$-th (extended) $g$-vector
of the seed $t$ with respect to the initial seed $t_{0}$. Its principal
part is defined to be $\pr_{I_{\ufv}}g_{i}(t)$, where $\pr_{I_{\ufv}}$
denote the natural projection from $\Z^{I}$ to $\Z^{I_{\ufv}}$.
Let $\tG$ denote the $I\times I_{\ufv}$-matrix formed by the column
vectors $g_{k}^{t_{0}}(t)$, $k\in I_{\ufv}$, and $G(t)=G^{t_{0}}(t)$
its $I_{\ufv}\times I_{\ufv}$ submatrix called the $G$-matrix.

We extend the $I_{\ufv}\times I_{\ufv}$ matrix $B(t_{0})$ to the
$(I_{\ufv}\sqcup I_{\ufv}')\times I_{\ufv}$-matrix $\tB(t_{0})^{\prin}=\left(\begin{array}{c}
B(t_{0})\\
\mathrm{Id}_{I_{\ufv}}
\end{array}\right)$ with $I_{\ufv}'=I_{\ufv}$, called the matrix of principal coefficients.
For any seed $t=\overleftarrow{\mu}t_{0}$, we apply the mutation
sequence $\overleftarrow{\mu}$ to the initial matrix $\tB(t_{0})^{\prin}$
and the resulting matrix takes the form $\left(\begin{array}{c}
B(t)\\
C(t)
\end{array}\right)$. The $I_{\ufv}'\times I_{\ufv}$-matrix $C(t)=C^{t_{0}}(t)$ is called
the $C$-matrix. The $k$-th column vector of $C(t)$, denoted by
$c_{k}^{t_{0}}(t)$, is called the $k$-th $c$-vector.

Notice that the construction of the $c$-vectors and $g$-vectors
depend on the choice of the initial seed $t_{0}$. In addition, the
$c$-vectors and principal $g$-vectors only depend on the principal
part $B(t_{0})$. When the context is clear, we often omit symbol
$t_{0}$.

The following result is a consequence of \cite[Theorem 5.11]{gross2018canonical},
see also \cite{nakanishi2012tropical} \cite[Section 5.6]{Keller12}.

\begin{Thm}\label{thm:cg_vector}

(1) The $c$-vectors are sign coherent, i.e., for any seed $t$ and
$k\in I_{\ufv}$, we must have $c_{k}(t)\geq0$ or $c_{k}(t)\leq0$.

(2) For any given mutation sequence $\overleftarrow{\mu}=\mu_{i_{r}}\cdots\mu_{i_{0}}$,
denote $t_{s}=\mu_{i_{s-1}}\cdots\mu_{i_{0}}t_{0}$. Choose signs
$\varepsilon_{s}$ to be the sign of the $k$-th $c$-vector $c_{i}(t_{s})$.
Then we have $C(t)=C(t_{r+1})=F_{i_{0},\varepsilon_{0}}(t_{0})\cdots F_{i_{r},\varepsilon_{r}}(t_{r})$
and $G(t)=E_{i_{0},\varepsilon_{0}}(t_{0})\cdots E_{i_{r},\varepsilon_{r}}(t_{r})$.

\end{Thm}

Let $\pr_{I_{\ufv}}$ denote the natural projection from $\Z^I$ to $\Z^{I_{\ufv}}$.
\begin{Cor}\label{cor:c_g_basis}

Given seeds $t=\overleftarrow{\mu}t_{0}$ where $t_{0}$ is any chosen
initial seed. Then the $c$-vectors $c_{i}(t)$ of $t$ form a $\Z$-basis
of $\Z^{I_{\ufv}}$, and the principal $g$-vectors $\pr_{I_{\ufv}}g_{i}(t)$ form
a basis of $\Z^{I_{\ufv}}$.

\end{Cor}

We can view extended $g$-vectors as principal $g$-vectors in the
following way. View the vertices $I$ as unfrozen and add principal coefficients as in \cite{FominZelevinsky07}. Then the previous
extended $g$-vectors become principal $g$-vectors. Consequently,
we deduce that, by extending the matrix $\tG(t)$ with unit column vectors $f_j,j\in I_{\fv}$, the matrix $\begin{pmatrix}
\tG(t)\ |\ f_j,j\in I_{\fv}\\
\end{pmatrix}$ equals $\tE_{i_{0},\varepsilon_{0}}(t_{0})\cdots\tE_{i_{r},\varepsilon_{r}}(t_{r})$.

It is useful to collect some facts about the calculation involved
matrices $E_{k,\varepsilon}$ and $F_{k,\varepsilon}$, see \cite[Section 5.6]{Keller12}.

Let $t^{\vee}$ denote the Langlands dual of $t$ whose associated
matrix satisfy $b_{ij}(t^{\vee})=-b_{ji}(t)$. Let $t^{\op}$ denote
the seed opposite to $t$ such that $b_{ij}(t^{\op})=-b_{ij}(t)$.

\begin{Prop}\label{prop:calculate_EF}

Let $t'=\mu_{k}t$ for some $k\in I_{\ufv}$. Let $\varepsilon$ be
any sign.

(1) $\tB(t')=\tE_{k,\varepsilon}(t)\cdot\tB(t)\cdot F_{k,\varepsilon}(t)$
for any sign $\varepsilon$.

(2) $\tE_{k,\varepsilon}^{2}=\Id_{I}$ and $F_{k,\varepsilon}^{2}=\Id_{I_{\ufv}}$.

(3) We have
\begin{align*}
E_{k,-\varepsilon}(t') & =E_{k,\varepsilon}^{-1}(t)\\
F_{k,-\varepsilon}(t') & =F_{k,\varepsilon}^{-1}(t)\\
E_{k,\varepsilon}(t^{\op}) & =E_{k,-\varepsilon}(t)\\
F_{k,\varepsilon}(t^{\op}) & =F_{k,-\varepsilon}(t)
\end{align*}

(4) Let $D'$ denote the diagonal matrix $\mathrm{diag}(d_{k}')_{k\in I_{\ufv}}$,
then $E_{k,\varepsilon}^{T}D'F_{k,\varepsilon}=D'$.

(5) We have $E_{k,\varepsilon}(t^{\vee})^{T}=F_{k,\varepsilon}(t)$

(6) Given any initial seed $t_{0}$, we have $G(t')=G(t)\cdot E_{k,\sign(c_{k}(t))}(t)$.

\end{Prop}

\begin{proof}

The claim (6) is a consequence of Theorem \ref{thm:cg_vector}. The
other claims can be obtained from direct calculation.

\end{proof}

The following result shows that $\tB(\overleftarrow{\mu}(t^{\op}))=\tB((\overleftarrow{\mu}t)^{\op})$.

\begin{Lem}

Let $t=\overleftarrow{\mu}t_{0}$ where $\overleftarrow{\mu}=\mu_{i_{r}}\cdots\mu_{i_{0}}$.
Then we have $\overleftarrow{\mu}(-\tB(t_{0}))=-(\overleftarrow{\mu}\tB(t_{0}))$.

\end{Lem}

\begin{proof}

Denote $t_{s}=\mu_{i_{s-1}}\cdots\mu_{i_{0}}t_{0}$. Choose any signs
$\varepsilon_{s}$ for the seeds $t_{s}$.

We prove the claim by induction on the length of $\overleftarrow{\mu}$
which equals $r+1$. The case $r+1=0$ is trivial. Assume that we
have shown the result for length $r$.

We have

\begin{align*}
-(\overleftarrow{\mu}\tB(t_{0})) & =-\tE_{i_{r},\varepsilon_{r}}(t_{r})\tB(t_{r})F_{i_{r},\varepsilon_{r}}(t_{r})\\
 & =\tE_{i_{r},\varepsilon_{r}}(t_{r})\tB(t_{r}^{\op})F_{i_{r},\varepsilon_{r}}(t_{r})\\
 & =\tE_{i_{r},-\varepsilon_{r}}(t_{r}^{\op})\tB(t_{r}^{\op})F_{i_{r},-\varepsilon_{r}}(t_{r}^{\op})\\
 & =\mu_{i_{r}}\tB(t_{r}^{\op})
\end{align*}

By induction hypothesis, 
\begin{align*}
\tB(t_{r}^{\op}): & =-\tB(t_{r})\\
 & =-\mu_{i_{r-1}}\cdots\mu_{i_{0}}\tB(t_{0})\\
 & =\mu_{i_{r-1}}\cdots\mu_{i_{0}}(-\tB(t_{0}))
\end{align*}

Therefore, $-(\overleftarrow{\mu}\tB(t_{0}))=\mu_{i_{r}}\mu_{i_{r-1}}\cdots\mu_{i_{0}}(-\tB(t_{0}))=\overleftarrow{\mu}(-\tB(t_{0}))$.

\end{proof}

Finally, we have the following duality between $c$-vectors and $g$-vectors.

\begin{Thm}\cite[Theorem 1.2]{nakanishi2012tropical}\cite{gross2018canonical}\label{thm:cg_duality}

Given any seeds $t=\overleftarrow{\mu}t_{0}$. Then we have

\begin{align*}
G^{t_{0}}(t)^{T}\cdot C^{t_{0}^{\vee}}(\overleftarrow{\mu}t_{0}^{\vee}) & =\Id_{I_{\ufv}}\\
C^{t_{0}}(t)\cdot C^{t^{\op}}(\overleftarrow{\mu}^{-1}(t^{\op})) & =\Id_{I_{\ufv}}\\
G^{t_{0}}(t)^{T} & =C^{(t^{\vee})^{\op}}(\overleftarrow{\mu}^{-1}((t^{\vee})^{\op}))
\end{align*}

When $B(t_{0})^{T}=-B(t_{0})$, we have $B(t_{0}^{\vee})=B(t_{0})$,
and consequently, $G^{t_{0}}(t)^{T}\cdot C^{t_{0}}(t)=\Id_{I_{\ufv}}$.

\end{Thm}

The $g$-vectors of a seed $t'$ obey the tropical transformation
$\phi_{t,t_{0}}$ where $t,t_{0}$ are initial seeds. More precisely,
we have the following result.

\begin{Thm}\cite{DerksenWeymanZelevinsky09}\cite{gross2018canonical}

Given seeds $t_{0},t,t'$ related by mutations. Then we have $\tG^{t}(t')=\phi_{t,t_{0}}\tG^{t_{0}}(t')$.

\end{Thm}

\subsection{Injective-reachability and green to red sequences}

\begin{Def}\cite{qin2017triangular}

A seed $t$ is said to be injective-reachable if there exists a seed
$t'=\overleftarrow{\mu}t$ and a permutation $\sigma$ of $I_{\ufv}$
such that the principal $g$-vectors $\pr_{I_{\ufv}}g_{k}^{t}(t')$
equals $-\pr_{I_{\ufv}}f_{\sigma(k)}$ for any $k\in I_{\ufv}$, where
$f_{i}$ are the $i$-th unit vector of $\Mc(t)\simeq\Z^{I}$.

In this case, we denote $t'$ by $t[1]$, and $t$ by $t'[-1]$.

\end{Def}

Note that the mutation sequence $\overleftarrow{\mu}$ is not unique. We choose and fix one such sequence once and for all.

For any permutation $\sigma$, let $\Perm_{\sigma}$ denote the $I_{\ufv}\times I_{\ufv}$-matrix
such that $(\Perm_{\sigma})_{ik}=\delta_{i,\sigma(k)}$. Then $t$
is injective-reachable if and only if $G(t)=-\Perm_{\sigma}$ for
some $\sigma$. Notice that $\Perm_{\sigma^{-1}}=\Perm_{\sigma}^{T}$.

Notice that the seed $t[1]$, if it exists, is determined by $t$
up to a permutation of $I_{\ufv}$. Define $t[d+1]=t[d][1]$ and $t[d-1]=t[d][-1]$,
we obtain a chain of seeds $(t[d])_{d\in\Z}$. In addition, some
$t\in\Delta^{+}$ is injective-reachable implies all $t'\in\Delta^{+}$
are injective-reachable. See \cite{qin2017triangular} for more details.
We have the following notion following \cite{keller2011cluster}.

\begin{Def}

Given seeds $t'=\overleftarrow{\mu}t$. The mutation sequence $\overleftarrow{\mu}$
is said to be a green to red sequence starting from $t$ if $c_{k}^{t}(t')$
have negative signs for all $k\in I_{\ufv}$.

\end{Def}

\begin{Prop}\label{prop:negative_c_chamber}

The injective-reachable condition is satisfied if and only if $c_{k}^{t}(t')=-e_{\sigma(k)}$
for any $k\in I_{\ufv}$, where $e_{k}$ are the $k$-th unit vector
of $\yCone(t)\simeq\Z^{I_{\ufv}}$, or, equivalently, $C(t)=-\Perm_{\sigma}$.
In addition, when $C(t)=-\Perm_{\sigma}$, we must have $d'_{k}=d'_{\sigma(k)}$
for any $k\in I_{\ufv}$.

\end{Prop}

\begin{proof}

Denote $t'=\overleftarrow{\mu}t$ where $\overleftarrow{\mu}=\mu_{i_{r}}\cdots\mu_{i_{0}}$.
Define $t_{s}=\mu_{i_{s}}\cdots\mu_{i_{0}}t_{0}$, $\varepsilon_{s}=\sign(c_{i_{s}}(t_{s}))$,
$D'=\mathrm{diag}(d_{k}')_{k\in I_{\ufv}}$ as before. 

By Proposition \ref{prop:calculate_EF} and Theorem \ref{thm:cg_vector},
we have

\begin{align*}
D' & =E_{i_{r},\varepsilon_{r}}^{T}\cdots E_{i_{0},\varepsilon_{0}}^{T}D'F_{i_{0},\varepsilon_{0}}\cdots F_{i_{r},\varepsilon_{r}}\\
 & =(E_{i_{0},\varepsilon_{0}}\cdots E_{i_{r},\varepsilon_{r}})^{T}D'F_{i_{0},\varepsilon_{0}}\cdots F_{i_{r},\varepsilon_{r}}\\
 & =G^{t}(t')^{T}\cdot D'\cdot C^{t}(t')
\end{align*}

If the injective-reachable condition is satisfied, then we have $G^{t}(t')=-\Perm_{\sigma}$.
Therefore, $D'=-\Perm{}_{\sigma}^{T}D'C^{t}(t')$ and, consequently, we have 
$C^{t}(t')=-D'^{-1}\Perm_{\sigma}D'$, $c_{k}^{t}(t')=-\frac{d'_{k}}{d'_{\sigma(k)}}e_{\sigma(k)}$.
Because $c_{k}^{t}(t')$ are integer vectors, we must have $d'_{k}=d'_{\sigma(k)}$
and $c_{k}^{t}(t')=-e_{\sigma(k)}$ for any $k\in I_{\ufv}$.

Conversely, if $C^{t}(t')=-\Perm{}_{\sigma}$ $\forall k\in I_{\ufv}$,
we can show $d'_{k}=d'_{\sigma(k)}$ and $G^{t}(t')=-\Perm_{\sigma}$
in the same way.

\end{proof}

\begin{Cor}\label{cor:injective_red}

Given seeds $t'=\overleftarrow{\mu}t$. Then $t'=t[1]$ if and only
if $\overleftarrow{\mu}$ is a green to red sequence starting from
$t$.

\end{Cor}

\begin{proof}

The only if part is a consequence of Proposition \ref{prop:negative_c_chamber}. 

On the other hand, if $\overleftarrow{\mu}$ is a green to red sequence,
then $c_{k}^{t}(t')<0$ for all $k$. It defines a chamber $\cC^{t'}=\{m\in\R^{I}|c_{k}^{t}(t')\cdot\pr_{I_{\ufv}}m\geq0\}$
in the cluster scattering diagram associated to the initial seed $t$
(Section \ref{sec:Theta-functions}). But the chamber $\cC^{t'}$
contains the negative chamber $\cC^{-}=(\R_{\leq0}^{I_{\ufv}})\oplus\R^{I_{\fv}}$
of the scattering diagram. Therefore, one must have $\cC^{-}=\cC^{t'}$,
and consequently $C^{t}(t')=-\Perm_{\sigma}$ for some $\sigma$.
The claim follows from Proposition \ref{prop:negative_c_chamber}.

\end{proof}

\subsection{Cluster categories}\label{sec:cluster_category}

We refer the reader to \cite{Keller08Note}\cite{Plamondon10a} for
details of this section. A quiver $\tQ$ is a finite oriented graph,
which we assume to have no loops or $2$-cycles throughout this paper.
Denote its set of vertices by $I$ and arrows by $E$. An ice quiver
$\tQ$ is a quiver endowed with a partition of its vertices $I=I_{\ufv}\sqcup I_{\fv}$
(unfrozen and frozen respectively). The full subquiver of $\tQ$ supported
on the unfrozen vertices $I_{\ufv}$ is called the principal part
and denoted by $Q$.

To any ice quiver $\tQ$, we can associate an $I\times I$ skew-symmetric
matrix $(b_{ij})$ such that $b_{ij}$ is the difference between the
number of arrows from $j$ to $i$ and that from $i$ to $j$. Its
$I\times I_{\ufv}$-submatrix and $I_{\ufv}\times I_{\ufv}$-submatrix
are denoted by $\tB$ and $B$ as before. Conversely, to any $I\times I$
skew-symmetric integer matrix $(b_{ij})$, we can associate an ice
quiver $\tQ$.

The path algebra $\C\tQ$ is the $\C$-algebra
generated by paths of $\tQ$ whose multiplication is given by path
composition. $\C\tQ$ has the maximal ideal $\mm$ generated by the
arrows $a\in E$. Let $\widehat{\C\tQ}$ denote the completion. Choosing
a linear combination of oriented cycles $\tW\in\widehat{\C\tQ}$
called a potential, we can define its cyclic derivatives $\partial_{a}\tW$
for any $a\in E$, see \cite{DerksenWeymanZelevinsky08}. The ideal
$\langle\partial_{a}\tW\rangle_{a\in E}$ of $\widehat{\C\tQ}$ has
the closure $\overline{\langle\partial_{a}\tW\rangle_{a\in E}}=\cap_{n>0}(\langle\partial_{a}\tW\rangle_{a\in E}+\mm^{n})$.
We define the completed Jacobian algebra associated to the quiver
with potential to be $J_{(\tQ,\tW)}=\widehat{\C\tQ}/\overline{\langle\partial_{a}\tW\rangle_{a\in I}}$.
By restricting the potential $\tW$ to the full subquiver $Q$ (arrows
not contained in $Q$ are sent to $0$), we obtain the principal quiver
with potential $(Q,W)$ and the corresponding Jacobian algebra $J_{(Q,W)}$.

Let $\Gamma=\Gamma_{(\tQ,\tW)}$ denote the Ginzburg dg algebra (differential
graded algebra) associated to $(\tQ,\tW)$ \cite{Ginzburg06}. Then
its homology is concentrated at negative degrees such that $H^{>0}\Gamma=0$,
$H^{0}\Gamma=J_{(\tQ,\tW)}$. Let $\per\Gamma$ denote the perfect
derived category of $\Gamma$ (smallest triangulated category containing
$\Gamma$), and $D_{fd}\Gamma$ the full subcategory consisting of
objects with finite dimensional total homology. Let $\Sigma$ denote
the shift functor. 

The (generalized) cluster category $\cC=\cC_{(\tQ,\tW)}$ is defined
to be the quotient category $\per\Gamma/D_{fd}\Gamma$ \cite{Amiot09}.
Let $\pi$ denote the natural projection. We further assume that $J_{(\tQ,\tW)}=H^{0}\Gamma$
is finite dimensional. Then the category $\cC$ is a Hom-finite $2$-Calabi-Yau
triangulated category, which means $\Hom(X,\Sigma Y)\simeq D\Hom(Y,\Sigma X)$.
Furthermore, $\pi\Gamma$ is a cluster tilting object of $\cC$, i.e.,
$\Hom_{\cC}(\pi\Gamma,\Sigma(\pi\Gamma))=0$ and $\Hom_{\cC}(\pi\Gamma,\Sigma X)=0$
implies $X\in\add(\pi\Gamma)$. The subcategory of coefficient-free
objects is defined to be the full subcategory 
\begin{eqnarray*}
^{\bot}(\Sigma T_{\fv}) & = & \{X\in\cC|\Hom(X,\Sigma T_{\fv})=0\}
\end{eqnarray*}
where $T_{\fv}=\oplus_{i\in I_{\fv}}(\pi\Gamma_{i})$ and $\Gamma_{i}$
denote the $i$-th indecomposable projective of $\Gamma$.

From now on, we always assume that the potential $\tW$ is chosen to be
non-degenerate \cite{DerksenWeymanZelevinsky08}. Then we can mutate
cluster tilting objects. The cluster category $\cC$ associated to
$(\tQ,\tW)$ provides a categorification for the cluster algebra associated
to the initial seed $t_{0}=((b_{ij}),(x_{i}))$, such that we associate
cluster tilting objects $T(t)$ for $t\in\Delta^{+}$, with $T(t_{0})=\pi\Gamma$,
and quivers with potential $((\tQ(t),\tW(t))$ with $\tQ(t)$ corresponding
to $(b_{ij}(t))$. Notice that $T_{\fv}$ is a common summand for
all $T(t)$, $t\in\Delta^{+}$.

For any $M\in\cC$ and $T=T(t)$, we have an $\add T$-approximation
in $\cC$

\begin{eqnarray*}
T^{(1)} & \rightarrow & T^{(0)}\rightarrow M\rightarrow\Sigma T^{(1)}.
\end{eqnarray*}
Let us use identify the Grothendieck ring of $\add T$ with $\Mc(t)\simeq\Z^{I}$
such that the isoclass $[T_{i}]$ corresponds to the $i$-th unit
vector $f_{i}$. The index of $X$ is defined to be $\Ind^{T}M=[T^{(0)}]-[T^{(1)}]$.

For convenience, we consider right modules unless otherwise specified. We define the functor $F$ such that

\begin{eqnarray*}
F: & \cC & \rightarrow J_{(\tQ(t),\tW(t))}-\mod\\
 & X & \mapsto\Hom(T,\Sigma X).
\end{eqnarray*}
Its restriction on $^{\bot}(\Sigma T_{\fv})$ has image in $J_{(Q(t),W(t))}-\mod$.

\begin{Def}[Caldero-Chapoton formula]\label{def:CC_formula}
Consider the classical case $\kk=\Z$. For any given skew-symmetric seed $t$, the corresponding cluster
tilting object $T=T(t)$, and any coefficient-free object $M\in^{\bot}(\Sigma T_{\fv})$,
the cluster character of $M$ is defined to be the Laurent polynomial
in $\LP(t)$:

\begin{align*}
CC^{t}(M) & =x(t)^{\Ind^{T}M}(\sum_{n\in\yCone^{\geq0}(t)}\chi(\Gr_{n}FM)\cdot x(t)^{\tB\cdot n})
\end{align*}
where $\Gr_{n}FM$ is the submodule Grassmannian of the $J_{(Q(t),W(t))}$-module
$FM$ consisting of $n$-dimensional submodules, and $\chi$ denote
the topological Euler characteristic.

\end{Def}

We also define $CC^{t}(FM)=CC^{t}(M)$.

Let us recall the Calabi-Yau reduction in the sense of \cite{IyamaYoshino08},
see \cite[Section 3.3]{plamondon2013generic} for a brief introduction.

Let $(T_{\fv})$ denote the ideal of all morphisms of the cluster
category $\cC_{(\tQ,\tW)}$ which factor through $T_{\fv}$, then
the quotient $^{\bot}(\Sigma T_{\fv})/(T_{\fv})$ is naturally endowed
with a structure of triangulated category. Furthermore, $^{\bot}(\Sigma T_{\fv})/(T_{\fv})$
is equivalent to the cluster category $\cC_{(Q,W)}$ associated to
$(Q,W)$.

Let us use $\underline{\Gamma}$ to denote the Ginzburg algebra $\Gamma_{Q,W}$
and let $\underline{T}$ denote the corresponding cluster tilting
object in $\cC_{(Q,W)}$. Then, under the above quotient and equivalence,
any $T_{k}$ with $k\in I_{\ufv}$ is sent to $\underline{T}_{k}$.

Any object $M\in^{\bot}(\Sigma T_{\fv})$ is sent to an object $\underline{M}$
in $\cC_{(Q,W)}$. By \cite{plamondon2013generic}, the index of $\underline{M}$
is given by projection

\begin{align*}
\Ind^{\underline{T}}\underline{M} & =\pr_{I_{\ufv}}(\Ind^{T}M).
\end{align*}
In particular, if we let $I_{k}$ denote the indecomposable object
in $^{\bot}(\Sigma T_{\fv})$ which corresponds to $\Sigma(\uT_{k})$
in $\cC_{(Q,W)}$, then $\pr_{I_{\ufv}}\Ind^{T}I_{k}=-f_{k}$. Notice
that $F\Sigma(\uT_{k})$ is the $k$-th injective module of $J_{(Q,W)}$,
which we also denote by $I_{k}$.

By \cite{plamondon2013generic}, for any $g\in\Z^{I_{\ufv}}$, there
exists some $m\in\N^{I_{\fv}}$ depending on $g$ such that, for a
generic morphism $f\in\Hom(T^{[-g]_{+}},T^{[g]_{+}+m})$ (see \cite{Palu08a}),
$\mathrm{cone}f$ belongs to $\in^{\bot}(\Sigma T_{\fv})$ and has
no direct summand in $\add T_{\fv}$. We define the generic cluster
character associated to $g+m$ to be $\gen_{g+m}=CC(\mathrm{cone}f)$.

\begin{Thm}\cite[Theorem 1.3]{plamondon2013generic}

Given skew-symmetric seeds $t'=\overleftarrow{\mu}t$. Then the generic
cluster characters satisfy

\begin{align*}
\overleftarrow{\mu}^{*}\gen_{g+m}^{t} & =\gen_{\phi_{t',t}g+m}^{t'}.
\end{align*}

\end{Thm}

\subsection{Quantization}\label{sec:quantization}

We briefly recall the necessary modification needed for the quantum case $\kk=\Z[q^{\pm\Hf}]$. Assume that a seed $t$ satisfies the full rank assumption as before.

First, we endow the seed $t$ with a quantum seed structure by choosing a compatible $\Z$-valued skew-symmetric bilinear form $\lambda$ on $\Mc(t)$ and strictly positive integers $d'_k$, $k\in I_{\ufv}$. By compatibility, we mean
\begin{align*}
\lambda(f_i,p^* e_k)=-\delta_{i,k}d'_k,\ \forall i\in I,k\in I_{\ufv}.
\end{align*}
For any seed $t'=\mu_k t$, $k\in I_{\ufv}$, the linear isomorphism $\Mc(t')\simeq \Mc(t)$ via \eqref{eq:tropical_f} induces a bilinear form on $\Mc(t)$, which we still denote by $\lambda$. It follows from \cite{BerensteinZelevinsky05} that $\lambda$ is compatible with $t'$ as well. Repeatedly, we endow quantum seed structures to all seeds obtained from $t$ by iterated mutations.

For any quantum seed $t$, we endow the Laurent polynomial ring $\LP(t)$ with an extra multiplication called twisted product $*$, such that:
\begin{align*}
x^m * x^{m'}=q^{\Hf\lambda(m,m')}x^{m+m'},\ \forall m,m'\in\Mc(t).
\end{align*}
Note that $*$ becomes the commutative product $\cdot$ when we specialize $q^{\Hf}$ to $1$.

Unless otherwise specified, we will choose this twisted product $*$ as the multiplication for the $\kk$-algebra $\LP(t)$ instead of the commutative product $\cdot$. 

Similarly, we endow $\kk[N(t)]$ with the twisted product $*$
\begin{align*}
y^n* y^{n'}=q^{\Hf\lambda(p^*n,p^* n')}y^{n+n'},\ \forall n,n'\in N_{\ufv}(t).
\end{align*}
Then $p^*$ induces a $\kk$-algebra homomorphism from $\kk[N(t)]$ to $\LP(t)$ commuting with the twisted products.

Using the twisted product, we construct the skew-fields of fractions of $\LP(t)$ and $\kk[N(t)]$ and denote them by $\cF(t)=\cF(\LP(t))$ and $\cF(\kk[N(t)])$ respectively. The classical automorphisms in \eqref{eq:rho} are quantized to the automorphisms $\rho_{k,\varepsilon}$, such that, for $i\neq k$,
\begin{align}\label{eq:q_rho}
\begin{split}
\rho_{k,\varepsilon}(x_{i})  &= x_{i},\\
\rho_{k,\varepsilon}(x_{k}^{-1}) &= x_{k}^{-1}+x^{-f_k+\varepsilon v_{k}}\\
\rho_{k,\varepsilon}(y_{i})  &=y_{i}\cdot \sum_{s=0}^{|b_{ki}|} \begin{pmatrix}
|b_{ki}|\\s 
\end{pmatrix} _{q_k} y_{k}^{\varepsilon s},\ b_{ki}\leq 0 \\
\rho_{k,\varepsilon}(y_{i}^{-1})  &=y_{i}^{-1}\cdot \sum_{s=0}^{|b_{ki}|} \begin{pmatrix}
|b_{ki}|\\s 
\end{pmatrix} _{q_k} y_{k}^{\varepsilon s},\ b_{ki}> 0 \\
\rho_{k,\varepsilon}(y_{k})  &=y_{k}.
\end{split}
\end{align}
where we denote $q_k=q^{\Hf d'_k}$, the quantum number $[a]_{q}=\frac{q^a-q^{-a}}{q-q^{-1}}$ for $0\neq a\in \N$, $[0]_q!=1$, $[a]_q!=[a]_q[a-1]_q\cdots [1]_q$, and $\begin{pmatrix}
a\\b 
\end{pmatrix} _{q}=\frac {[a]_q!}{[b]_q! [a-b]_q!}$.

As before, define quantum mutations $\mu_k^*$ as the compositions $\rho_{k,\varepsilon}\circ\tau_{k,\varepsilon}$. Then they are independent of the choice of the sign $\varepsilon$, such that, for $i\neq k$,
\begin{align}\label{eq:q_mutation}
\begin{split}
\mu_k^* (x_{i}') & = x_{i}\\
\mu_k^* (x_{k}') & = x^{-f_k+\sum_j[-b_{jk}]_{+}f_j}+x^{-f_k+\sum_i[b_{ik}]_{+}f_i}\\
\mu_k^* (y_{i}') &=y_{i}\cdot \sum_{s=0}^{|b_{ki}|} \begin{pmatrix}
|b_{ki}|\\s 
\end{pmatrix} _{q_k} y_{k}^{ s},\ b_{ki}\leq 0 \\
\mu_k^* ((y_{i}')^{-1})  &=y_{i}^{-1}y_{k}^{-b_{ki}}\cdot \sum_{s=0}^{|b_{ki}|} \begin{pmatrix}
|b_{ki}|\\s 
\end{pmatrix} _{q_k} y_{k}^{ s},\ b_{ki}> 0 \\
\mu_k^* (y_{k}') &=y_{k}^{-1}.
\end{split}
\end{align}

We define the quantum (upper) cluster algebras as in Definition \ref{def:cluster_algebra}, using the quantum mutations and twisted products in the construction.

\section{Bidegrees and support of Laurent polynomials\label{sec:bidegrees_and_support}}

Given a seed $t=((b_{ij}(t))_{i,j\in I},(x_{i}(t))_{i\in I})$ such
that the $I\times I_{\ufv}$-matrix $\tB(t)$ is of full rank. Recall
that we have
\begin{align*}
\Mc(t) & \simeq\Z^{I}\\
N(t) & \simeq\Z^{I}\\
\yCone(t) & \simeq\Z^{I_{\ufv}}\\
\yCone^{\geq0}(t) & \simeq\N^{I_{\ufv}}\\
\yCone^{>0}(t) & \simeq\N^{I_{\ufv}}-\{0\},
\end{align*}
where the natural basis of $\Mc(t)$, $N(t)$ and $\yCone(t)$ are
denoted by $\{f_{i}|i\in I\}$, $\{e_{i}|i\in I\}$ and $\{e_{k}|k\in I_{\ufv}\}$
respectively. The pairing $\langle\ ,\ \rangle$ between $\Mc(t)$
and $N(t)$ are defined such that $\langle f_{i},e_{j}\rangle=\frac{1}{d_{i}}\delta_{ij}$.
In addition, $N(t)$ is endowed with the skew-symmetric bilinear form
$\{\ ,\ \}$ such that $\{e_{i},e_{j}\}=d_{j}^{-1}b_{ji}$. We also
have the linear map $p^{*}:\yCone(t)\rightarrow\Mc(t)$ such that
$p^{*}(n)=\{n,\ \}$, which turns out to be $p^{*}(n)=\tB(t)\cdot n$
under the identification $\Mc(t)\simeq\Z^{I}$ and $\yCone(t)\simeq\Z^{I_{\ufv}}$.
Denote $v_{k}=p^{*}(e_{k})\ \forall k\in I_{\ufv}$. The vectors $\{v_{k}\}_{k\in I_{\ufv}}$
are linearly independent by the full rank assumption on $\tB(t)$.

\subsection{Dominance order}

The dominance order is the following partial order defined on $\Mc(t)$.

\begin{Def}[Dominance order {\cite[Definition 3.1.1]{qin2017triangular}}]\label{def:dominance_order}

For any given seed $t$ and $g,g'\in\Mc(t)$, we say $g'$ is dominated
by $g$, denoted by $g'\preceq_{t}g$, if and only if we have $g'=g+p^{*}(n)$
for some $n\in\yCone^{\geq0}(t)$. We write $g'\prec_{t}g$ if $g\neq g'$.

\end{Def}

For any given $g,\eta\in\Mc(t)$, we define the following subsets
of $\Mc(t)$:

\begin{align*}
\Mc(t)_{\preceq_{t}g} & =\{g'\in\Mc(t)|g'\preceq_{t}g\}\\
 & =g+p^{*}N_{\ufv}^{\geq0}(t)\\
_{\eta\preceq_{t}}\Mc(t) & =\{g'\in\Mc(t)|\eta\preceq_{t}g'\}\\
 & =\eta-p^{*}N_{\ufv}^{\geq0}(t)\\
_{\eta\preceq_{t}}\Mc(t)_{\preceq_{t}g} & =\{g'\in\Mc(t)|\eta\preceq_{t}g'\preceq_{t}g\}\\
 & =\ _{\eta\preceq_{t}}\Mc(t)\cap\Mc(t)_{\preceq_{t}g}.
\end{align*}

\begin{Lem}[Finite Interval Lemma, {\cite[Lemma 3.1.2]{qin2017triangular}}]\label{lem:finite_interval}

For any $\eta,g\in\Mc(t)$, $_{\eta\preceq_{t}}\Mc(t)_{\preceq_{t}g}$
is a finite set. In particular, if $\eta\preceq_{t}g$ and $g\preceq_{t}\eta$,
we must have $\eta=g$ as elements in $\Mc(t)$.

\end{Lem}

\begin{proof}

The claim follows from the assumption that $\tB(t)$ is of full rank.

\end{proof}

Recall that, for any two seeds $t,\ t'\in\Delta^{+}$, we have tropical
transformation $\phi_{t',t}:\Mc(t)\rightarrow\Mc(t')$. By viewing $\phi_{t',t}$ as an identification, the set of tropical points $\tropSet$ is the set
of equivalent classes. Moreover, the dominance order $\prec_{t'}$
is transported to $\Mc(t)$ and $\tropSet$ such that, for any $g,h\in \Mc(t)$, whenever we have $\phi_{t',t}h\prec_{t'} \phi_{t',t'}g$, we define $h\prec_{t'}g$ in $\Mc(t)$ and $[g]\prec_{t'}[h]$ in $\tropSet$.

In general, for any given sets of seeds $S,S'$, we define

\begin{align*}
\tropSet_{\preceq_{S}[g]} & =\{[g']\in\tropSet|[g']\preceq_{t}[g],\ \forall t\in S\}\\
_{[\eta]\preceq_{S'}}\tropSet & =\{[g']\in\tropSet|[\eta]\preceq_{t}[g'],\ \forall t\in S'\}\\
_{[\eta]\preceq_{S'}}\tropSet_{\preceq_{S}[g]} & =\ _{[\eta]\preceq_{S'}}\tropSet\cap\tropSet_{\preceq_{S}[g]}.
\end{align*}
We have similar definitions for $\Mc(t)_{\preceq_{S}g}$, $_{\eta\preceq_{S'}}\Mc(t)$,
$_{\eta\preceq_{S'}}\Mc(t)_{\preceq_{S}g}$. 

From now on, we use
the symbols $\Mc(t)$ and $g\in\Mc(t)$ if we want to specify
a special seed $t$, and $\tropSet$ and $[g]\in\tropSet$ otherwise.

\subsection{Formal Laurent series and bidegrees}

The monoid algebra $\kk[\yCone^{\geq0}(t)]=\kk[\lambda^{n}]_{n\in\yCone^{\geq0}(t)}$
has a maximal ideal $\mm=\kk[\yCone^{>0}(t)]$. The corresponding completion
is denoted by $\widehat{\kk[\yCone^{\geq0}(t)]}$. The injective linear
map $p^{*}:\yCone(t)\rightarrow\Mc(t)$ induces an embedding $p^{*}$
from $\kk[\yCone^{\geq}(t)]$ to $\LP(t)=\kk[\Mc(t)]=\kk[\chi^{m}]_{m\in\Mc(t)}$
such that $p^{*}(\lambda^{n})=\chi^{p^{*}(n)}\ \forall n\in\yCone(t)$.
We define the set of formal Laurent series to be
\begin{align*}
\widehat{\LP(t)} & =\LP(t)\otimes_{\kk[\yCone^{\geq0}(t)]}\widehat{\kk[\yCone^{\geq0}(t)]}
\end{align*}
where $\kk[\yCone^{\geq0}(t)]$ is viewed as a subalgebra of $\kk[\Mc(t)]$
via the embedding $p^{*}$.

Then a formal Laurent series is a finite sum of the elements of the
following type 
\begin{align*}
a\cdot x(t)^{g}\cdot\sum_{n\in\yCone^{\geq0}(t)}b_{n}y(t)^{n}
\end{align*}
where $a,b_{n}\in\kk,g\in\Mc(t)$, $x_{i}(t)=\chi^{f_{i}}$ and $y_{k}=\chi^{p^{*}(e_{k})}=\prod_{i}x_{i}^{b_{ik}}$
by the embedding $p^{*}$.

Similarly, let $\widehat{\kk[-\yCone^{\geq0}(t)]}$ denote the completion
of $\kk[-\yCone^{\geq0}(t)]$ with respect to its maximal ideal $\kk[-\yCone^{>0}(t)]$,
we can define 
\begin{align*}
\widetilde{\LP(t)} & =\LP(t)\otimes_{\kk[-\yCone^{\geq0}(t)]}\widehat{\kk[-\yCone^{\geq0}(t)]}
\end{align*}

Then any formal series $z\in\widetilde{\LP(t)}$ is a finite sum of
the elements of the following type 
\begin{align*}
a\cdot x(t)^{g}\cdot\sum_{n\in-\yCone^{\geq0}(t)}b_{n}y(t)^{n}
\end{align*}
where $a,b_{n}\in\kk,g\in\Mc(t)$.

Let us postpone the discussion of the ring structure for the moment
and give an intuitive definition of (co)degrees arising from dominance
order.

\begin{Def}[Degree, pointed \cite{qin2017triangular}]\label{def:pointed}

Given any formal sum $z=\sum_{g\in\Mc(t)}c_{g}x(t)^{g}$ where $c_{g}\in\kk$.
If the set of the Laurent degrees $\{g|c_{g}\neq0\}$ has a unique
$\prec_{t}$-maximal element $g$, we say $z$ has degree $g$ with
respect to $t$, and denote $\deg^{t}z=g$.

If we have $\deg^{t}z=g$ and $c_{g}=1$, then $z$ is said to be
pointed at $g$.

\end{Def}
A set is said to be pointed if it consists of elements pointed at different degrees.

We also need the following notion dual to Definition \ref{def:pointed}.

\begin{Def}[Codegree, copointed]\label{def:copointed}

Given any formal sum $z=\sum_{g\in\Mc(t)}c_{g}x(t)^{g}$ where $c_{g}\in\kk$.
If the set of the Laurent degrees $\{g|c_{g}\neq0\}$ has a unique
$\prec_{t}$-minimal element $\eta$, we say $z$ has codegree $\eta$
with respect to $t$, and denote $\codeg^{t}z=\eta$.

If we have $\codeg^{t}z=\eta$ and $c_{\eta}=1$, then $z$ is said
to be copointed at $\eta$.

\end{Def}

\begin{Def}[Bidegree, bipointed]\label{def:bipointed}

Given any formal sum $z=\sum_{g\in\Mc(t)}c_{g}x(t)^{g}$. If it
has $\deg^{t}z=g$ and $\codeg^{t}z=\eta$ for some $g,\eta\in\Mc(t)$,
we say $z$ has bidegree $(\eta,g)$, denoted by $\bideg^{t}z=(\eta,g)$.

If $z$ is further pointed at $g$ and copointed at $\eta$, we say
it is bipointed at $(\eta,g)$.

\end{Def}

We have the following easy observation.

\begin{Lem}

Given any formal sum $z=\sum_{g\in\Mc(t)}c_{g}x(t)^{g}$. If it
has bidegree $(\eta,g)$, then the following claims are true:

\begin{enumerate}

\item $\eta\preceq_{t}g$ .

\item $z$ is a Laurent polynomial. 

\item $z$ is a Laurent monomial if and only if $\eta=g$.

\end{enumerate}

\end{Lem}

\begin{proof}

The claim follows from definition and the finiteness of $_{\eta\preceq_{t}}\Mc(t)_{\preceq_{t}g}$
(Lemma \ref{lem:finite_interval}).

\end{proof}

We will mainly be interested in Laurent polynomials. But sometimes
our calculation will be carried out as formal series. Let us look
at these series in more details. Recall that we have identified $\kk[\yCone(t)]$
as a subalgebra of $\kk[\Mc(t)]$ via the embedding $p^{*}$.

Given any $g\in\Mc(t)$. The $\kk$-submodule $x^{g}\cdot\kk[\yCone^{\geq0}(t)]\subset\kk[\Mc(t)]$
is a rank one free module of the algebra $\kk[\yCone^{\geq0}(t)]$.
We define its completion to be the rank one free $\widehat{\kk[\yCone^{\geq0}(t)]}$-module
$x^{g}\cdot\widehat{\kk[\yCone^{\geq0}(t)]}$.

The subset $\ptSet^{t}(g):=x^{g}\cdot(1+\kk[\yCone^{>0}(t)])$ of $x^{g}\cdot\kk[\yCone^{\geq0}(t)]$
is the set of Laurent polynomials pointed at degree $g$. Let $\widehat{\kk[\yCone^{>0}(t)]}$
denote the subset of series in $\widehat{\kk[\yCone^{\geq0}(t)]}$
with vanishing constant terms. Then the subset $\hPtSet^{t}(g):=x^{g}(1+\widehat{\kk[\yCone^{>0}(t)]})$
of $x^{g}\cdot\widehat{\kk[\yCone^{\geq0}(t)]}$ is the set of formal
Laurent series pointed at degree $g$. Notice that we have $\ptSet^{t}(g)\subset\hPtSet^{t}(g)\subset\widehat{\LP(t)}$.

Similarly, the subset $\cPtSet^{t}(\eta):=x^{\eta}\cdot(1+\kk[-\yCone^{>0}(t)])$
of $x^{\eta}\cdot\kk[-\yCone^{\geq0}(t)]$ is the set of Laurent polynomials
copointed at degree $g$. In addition, we have the subset of copointed
formal Laurent series $\tCPtSet^{t}(\eta)=x^{\eta}\cdot(1+\widehat{\kk[-\yCone^{>0}(t)]})$
of $x^{\eta}\cdot\widehat{\kk[-\yCone^{\geq0}(t)]}$. Notice that we
have $\cPtSet^{t}(g)\subset\tCPtSet^{t}(g)\subset\widetilde{\LP(t)}$.

Finally, the subset $\bPtSet^{t}(\eta,g):=\ptSet^{t}(g)\cap\cPtSet^{t}(\eta)$
of $\kk[\Mc(t)]$ is the set of Laurent polynomials bipointed at
bidegree $(\eta,g)$.

\begin{Lem}[inverse]\label{lem:inverse_series}

(1) For any given pointed formal Laurent series $u\in\hPtSet^{t}(g)$,
where $g\in\Mc(t)$, $u$ has a multiplicative inverse $v$ in the
ring of formal Laurent series $\widehat{\LP(t)}$. In addition, $v$
belongs to $\hPtSet^{t}(-g)$.

(2) For any given copointed element $u'\in\tCPtSet^{t}(\eta)$, where
$\eta\in\Mc(t)$, $u'$ has a multiplicative inverse $v'$ in $\widetilde{\LP(t)}$.
In addition, $v'$ belongs to $\tCPtSet^{t}(-\eta)$.

\end{Lem}

\begin{proof}

(1) $u$ takes the form $u=x(t)^{g}* F$, where $F\in1+\widehat{\kk[\yCone^{>0}(t)]}$, and $*$ denote the twisted product.
Notice that $F$ has a unique inverse $F'\in1+\widehat{\kk[\yCone^{>0}(t)]}$
in $\widehat{\kk[\yCone^{\geq0}(t)]}$. Then $u$ has the inverse $v=F'*x(t)^{-g}$. 

(2) The proof is similar to (1).

\end{proof}

\begin{Lem}[product]\label{lem:product_series}

(1) For any given pointed series $z_{g},z_{\eta}$ with degree $g$
and $\eta$ respectively, their product is a well-defined pointed
series with degree $g+\eta$.

(2) For any given copointed series $z_{g},z_{\eta}$ with codegree
$g$ and $\eta$ respectively, their product is a well-defined copointed
series with codegree $g+\eta$.

\end{Lem}

\begin{proof}

(1) Notice that, for each Laurent degree $g'$ in the product, only
finitely many Laurent monomials of the pointed series $z_{g}$ and
$z_{\eta}$ will have contribution, because $_{g'\preceq}\tropSet_{\preceq g}$
and $_{g'\preceq}\tropSet_{\preceq\eta}$ are finite by Lemma \ref{lem:finite_interval}.
Therefore, the product is well-defined. In addition, it is pointed
at degree $g+\eta$ by direct computation. 

(2) The proof is similar to that of (1).

\end{proof}

\subsection{Degrees and codegrees under mutation}\label{sec:mutation_degrees}

Given two seeds $t,t'$ connected by a mutation sequence $t'=\overleftarrow{\mu}t$.
Recall that the lattice $\Mc(t)\simeq\Z^{I}$ has a natural basis
$\{f_{i}=f_{i}(t)|i\in I\}$.

\begin{Def}[Degree transformation]\label{def:deg_transform}

We define the linear map $\psi_{t',t}:\Mc(t)\ra\Mc(t')$ such
that 
\begin{align*}
\psi_{t',t}(\sum_{i\in I}g_{i}f_{i}) & =\sum_{i\in I}g_{i}\phi_{t',t}(f_{i})
\end{align*}
 for any $(g_{i})_{i\in I}\in\Z^{I}$.

\end{Def}

We have the following result, see Example \ref{eg:diff_psi_phi}.
\begin{Lem}\label{lem:diff_psi_phi}
Let there be given seeds $t'=\mu_k t\in \Delta^+$ for some $k\in I_{\ufv}$. Denote $\phi=\phi_{t',t}$ and $\psi=\psi_{t',t}$. Let $e_k'$ denote the $k$-th unit vector in $N_{\ufv}(t')$. For any $i\neq k\in I$ and $g\in \Mc(t)$, we have $\psi g-\phi g=[-g_k]_+ \tB' e_k'$.
\end{Lem}
\begin{proof}
Note that, in the lattice $\Mc(t')$, we have $\psi(f_k)=\deg^{t'}x_k(t)=\phi(f_k)=-f_k'+[b_{ik}]_+ f_i'$, see Definition \ref{def:tropical_transformation}. Direct calculation shows that, for any $i\neq k\in I$,
\begin{align*}
(\phi g-\psi g)_{i} & =(g_{i}+[b_{ik}]_{+}[g_{k}]_{+}-[-b_{ik}]_{+}[-g_{k}]_{+})-(g_{i}+[b_{ik}]_{+}g_{k})\\
 & =([b_{ik}]_{+}[g_{k}]_{+}-[-b_{ik}]_{+}[-g_{k}]_{+})-[b_{ik}]_{+}g_k\\
 & =([b_{ik}]_{+}[g_{k}]_{+}-[-b_{ik}]_{+}[-g_{k}]_{+})-[b_{ik}]_{+}([g_{k}]_{+}-[-g_{k}]_{+})\\ 
 & =-[-b_{ik}]_{+}[-g_{k}]_{+}+[b_{ik}]_{+}[-g_{k}]_{+}\\
 & =b_{ik}[-g_{k}]_{+}\\
 & =-b_{ik}'[-g_{k}]_{+}.
\end{align*}
Moreover, $(\phi g)_k=-g_k=(\psi g)_k$. We deduce that $\phi g-\psi g=-(\tB')\cdot[-g_{k}]_{+}e'_{k}$.
\end{proof}

\begin{Eg}\label{eg:diff_psi_phi}
Choose a seed $t$ such that $I=I_{\ufv}=\{1,2\}$, $\tB=(b_{ij})=\begin{pmatrix}
0&-1\\
1&0
\end{pmatrix}$. Take any $g=g_1 f_1+g_2 f_2\in \Mc(t)$, $g_1,g_2\in \Z$.

First, take $t'=\mu_1 t$. Then $\tB'=-\tB$. We have $\phi(g)=(-g_1)f_1'+(g_2+[g_1]_+)f_2'$, see Definition \ref{def:tropical_transformation}. In particular, $\psi(f_1)= \phi(f_1)=-f_1'+f_2'$ and $\psi(f_2)=\phi(f_2)=f_2'$. It follows that $\psi(g)=g_1\psi (f_1)+g_2\psi(f_2)=(-g_1)f_1'+(g_1+g_2)f_2'$. Therefore, $\psi g-\phi g=-[-g_1]_+ f_2'=[-g_1]_+\tB' e_1'$.

Second, take $t'=\mu_2 t$. Then $\tB'=-\tB$. We have $\phi(g)=(g_1-[-g_2]_+)f_1'+(-g_2)f_2'$, see Definition \ref{def:tropical_transformation}. In particular, $\psi(f_1)= \phi(f_1)=f_1'$ and $\psi(f_2)=\phi(f_2)=-f_2'$. It follows that $\psi(g)=g_1\psi (f_1)+g_2\psi(f_2)=g_1 f_1'+(-g_2)f_2'$. Therefore, $\psi g-\phi g=[-g_2]_+ f_1'=[-g_2]_+\tB' e_2'$.
\end{Eg}

\begin{Rem}[Non-trivial monodromy]\label{rem:monodromy}

Recall that the maps $\phi_{t',t}$ are piecewise linear and $\phi_{t,t'}\phi_{t',t}=\phi_{t,t}=\Id_{\Mc(t)}$.
By contrast, the maps $\psi_{t',t}$ are linear, but at the cost that
$\psi_{t,t'}\psi_{t',t}\neq\Id_{\Mc(t)}$ in general. 

It would be interesting to understand such non-trivial monodromy. We observe that this  monodromy for adjacent seeds agrees with the monodromy of signed mutations. 

More precisely, take $t$ as the initial seed and assume that $t'=\mu_{k,+} t$ for some unfrozen vertex $k$. Note that $b_{ik}'=-b_{ik}$ for any $i\in I$. Direct computations shows that, for $i\neq k$,
\begin{align}
\psi_{t,t'}\psi_{t',t}(f_{i}) =f_i
\end{align}
\begin{align*}
\psi_{t,t'}\psi_{t',t}(f_{k}) &=\psi_{t,t'}(
-f_k'+\sum_{i}[-b_{ik}']_{+}f_{i}' )\\
& =f_k-\sum_{j}[- b_{jk}]_{+}f_{j} +\sum_{i}[-b_{ik}']_{+}f_{i}' \\
&=f_k+\sum_{i}b_{ik}f_{i}
\end{align*}

On the other hand, if we apply signed mutations $\mu_{k,+}$ twice on the initial seed $t$, we obtain a seed $t''=\mu_{k,+}\mu_{k,+}(t)$. Let us compute $\tau_{k,+}\tau_{k,+}:\Mc(t'')\simeq \Mc(t)$. For $i\neq k$, we have
\begin{align}
\tau_{k,+}\tau_{k,+}(f_{i}'') =f_i
\end{align}
\begin{align*}
\tau_{k,+}\tau_{k,+}(f_{k}'') &=\tau_{k,+}(
-f_k'+\sum_{i}[- b_{ik}']_{+}f_{i}' )\\
& =f_k-\sum_{j}[-  b_{jk}]_{+}f_{j} +\sum_{i}[- b_{ik}']_{+}f_{i}' \\
&=f_k+\sum_{i}b_{ik}f_{i}
\end{align*}
We deduce that $\psi_{t,t'}\psi_{t',t}=\tau_{k,+}\tau_{k,+}$ if we identify $f_i''=f_i$ for any $i$. 

Note that the signed mutation monodromy $\mu_{k,+}\mu_{k,+}$ was discussed in \cite[Remark 2.5]{gross2013birational}. 
\end{Rem}

\begin{Lem}\label{lem:bijective_linear_map}

$\psi_{t't}$ is bijective.

\end{Lem}

\begin{proof}
Identify $\Mc(t')\simeq \Z^I$ such that $f_i'$ is viewed as the $i$-th unit vector. Let $\pr_{I_{\ufv}}$ denote the natural projection from $\Z^I$ to $\Z^{I_{\ufv}}$.

Denote $g_k(t;t')=\phi_{t',t}f_k(t)=\psi_{t',t}f_k(t)\in \Mc(t')$. It follows from Corollary \ref{cor:c_g_basis} that the principal $g$-vectors $\pr_{I_{\ufv}}g_k(t;t')$ with respect to the initial seed $t'$, $k\in I_{\ufv}$, form a basis of $\Z^{I_{\ufv}}$. Note that $g_j(t;t')=f_j'$ for any frozen vertex $j$. It follows that $g_i(t;t')$, $i\in I$, is a basis for $\Z^I$. In particular, the linear map $\psi_{t't}$ is bijective.
\end{proof}

Notice that we have two inclusions $\LP(t)\subset\cF(t)$ and $\LP(t)\subset\widehat{\LP(t)}$.
On the one hand, the mutation map $\overleftarrow{\mu}^{*}$ is an
isomorphism from the rational function field $\cF(t')$ to $\cF(t)$.
On the other hand, we have $\overleftarrow{\mu}^{*}(\kk[x_{i}(t')]_{\forall i})\subset\LP(t)\subset\widehat{\LP(t)}$.
In addition, $\overleftarrow{\mu}^{*}(x_{i}(t'))$, $\forall i$,
are pointed Laurent polynomials in $\LP(t)$, which are invertible
in $\widehat{\LP(t)}$ by Lemma \ref{lem:inverse_series}. Consequently,
the mutation map $\overleftarrow{\mu}^{*}$ induces an algebraic homomorphism
$\iota:\LP(t')\rightarrow\widehat{\LP(t)}$.

Our next observation shows that the linear map $\psi_{t',t}$ tracks
the degree of a Laurent monomial under change of seeds.

\begin{Lem}\label{lem:linear_transform_monomial}
	
	Given $t'=\overleftarrow{\mu}t$, any $g'\in\Mc(t')$ and $z=x(t')^{g'}\in\LP(t')$.
	Then we have $\iota(z)\in\hPtSet^{t}(\psi_{t,t'}g')$.
	
\end{Lem}

\begin{proof}
	
	Notice that the map $\iota$ identifies
	$x_{i}(t')$ with a pointed Laurent polynomial in $\ptSet^{t}(\deg^{t}x_{i}(t'))$.
	Then Lemma \ref{lem:inverse_series} implies $\iota x_{i}(t')^{-1}\in\hPtSet^{t}(-\deg^{t}x_{i}(t'))$.
	We obtain the claim by taking the product of these pointed formal
	series (Lemma \ref{lem:product_series}).
	
\end{proof}

\begin{Lem}\label{lem:mutation_series}

(1) The map $\iota$ is an embedding. 

(2) If $z\in\LP(t')\cap(\overleftarrow{\mu}^{*})^{-1}\LP(t)$, then
$\iota(z)=\overleftarrow{\mu}^{*}(z)\in\LP(t)$.

\end{Lem}

\begin{proof}

(1) For any Laurent polynomial $0\neq z=\sum_{g'\in\Mc(t')}b_{g'}x(t')^{g'}\in\LP(t)$,
$b_{g'}\in\kk$, the image $\iota(x(t')^{g'})\in\widehat{\LP(t)}$
is pointed at degree $\psi_{t,t'}g'$. Since $\psi_{t,t'}$ is bijective,
the image $\iota(z)$ is a finite sum of pointed elements with distinct
leading degrees. In particular, $\iota(z)\neq0$.

(2) Take any 
$z=(x')^{-d}*F$ for some $F\in\kk[\Mc(t')]$, $d\in \N^{I}$. On the one hand, we have $\iota((x')^{d})*\iota(z)=\iota(F)$ in $\widehat{\LP(t)}$. On the other hand, we have $\overleftarrow{\mu}^{*}((x')^{d})*\overleftarrow{\mu}^{*}(z)=\overleftarrow{\mu}^{*}(F)$ in $\LP(t)$. By definition of $\iota$, we have $\iota((x')^{d})=\overleftarrow{\mu}^{*}((x')^{d})$ and $\iota(F)=\overleftarrow{\mu}^{*}(F)$ in $\LP(t)$. The claim follows.

\end{proof}

Using this embedding, we can identify any Laurent polynomial $z\in\LP(t')$
as a formal Laurent series $\overleftarrow{\mu}^{*}(z):=\iota(z)$
in $\widehat{\LP(t)}$, called the \emph{formal Laurent series expansion}
of $z$ with respect to the seed $t$, or (formal) Laurent expansion
for short.

\begin{Rem}[Different expansion using codegrees]

Notice that the Laurent polynomials $\overleftarrow{\mu}^{*}(x_{i}(t))$,
$i\in I$, are copointed (Proposition  \ref{prop:cluster_monomial_bipointed}).
Then we can construct a similar embedding $\iota'$ from $\LP(t')$
to $\widetilde{\LP(t)}$ as a different formal series expansion.

\end{Rem}

\begin{Def}[Tropical points as degrees]

Given a formal Laurent series $z\in\widehat{\LP(t_{0})}$ with degree
$g\in\Mc(t_{0})$ such that, for any seeds $t_{0}=\overleftarrow{\mu}t$,
$\overleftarrow{\mu}^{*}z$ is a well-defined formal Laurent series
in $\widehat{\LP(t)}$ with degree $\deg^{t}\overleftarrow{\mu}^{*}z=\phi_{t,t_{0}}g\in\Mc(t)$.
Then we we say $z$ has degree $[g]\in\tropSet$.

\end{Def}

As before, denote $y_{k}(t)=y(t)^{e_{k}}=x(t)^{\sum_i b_{ik}(t)f_i}$,
$k\in I_{\ufv}$, where $e_{k}$ is the $k$-th unit vector in $\yCone(t)\simeq\Z^{I_{\ufv}}$ and $f_i$ the $i$-th unit vector in $\Mc(t)\simeq \Z^I$.
Apparently, $y_{k}(t)$ is a pointed Laurent polynomial in $\LP(t)$
and we have $\deg^{t}y_{k}(t)=\tB(t)\cdot e_{k}=\sum_{i\in I}b_{ik}(t)f_{i}$.
It follows that for any $n\in\yCone(t)$, we have $\deg^{t}(y(t)^{n})=\tB(t)\cdot n$.

The next result shows how $c$-vectors appear when one calculate the
degree of $y$-variables. This result is known for skew-symmetric
seeds via the cluster category approach, see \cite{KellerYang09}\cite{Nagao10}\cite{Keller12}.

\begin{Prop}[{\cite[Proposition 3.13]{FominZelevinsky07}}]\label{prop:c_vector_y_deg}

Given any seeds $t'=\overleftarrow{\mu}t$. For any $k\in I_{\ufv}$,
we have $\deg^{t}\overleftarrow{\mu}^{*}y_{k}(t')=\deg^{t}(y(t)^{c_{k}^{t}(t')})=\tB(t)\cdot c_{k}^{t}(t')$,
where $c_{k}^{t}(t')$ is the $k$-th $c$-vector of the seed $t'$
with respect to the initial seed $t$.

\end{Prop}

\begin{proof}

We use the description of $c$-vectors and $g$-vectors by Theorem
\ref{thm:cg_vector}. Denote the mutation sequence $\overleftarrow{\mu}=\mu_{i_{r}}\cdots\mu_{i_{0}}$,
seeds $t_{s}=\mu_{i_{s-1}}\cdots\mu_{i_{0}}t_{0}$ where $t_{0}=t$
and $t_{r+1}=t'$. Choose signs $\varepsilon_{s}$ to be the sign
of the $k$-th $c$-vector $c_{i}(t_{s})$.

Recall that $\tB(t')=E_{i_{r},\varepsilon_{r}}(t_{r})\cdots E_{i_{0},\varepsilon_{0}}(t_{0})\tB(t_{0})F_{i_{0},\varepsilon_{0}}(t_{0})\cdots F_{i_{r},\varepsilon_{r}}(t_{r})$.
Starting with the product $\overleftarrow{\mu}^{*}y_{k}(t')$,
we have

\begin{align*}
&\deg^{t}\overleftarrow{\mu}^{*}y_{k}(t')\\
 & =\sum_{i}\deg^{t}\overleftarrow{\mu}^{*}x_{i}(t')\cdot b_{ik}(t')\\
 & =\tG(t')\cdot\tB(t')\cdot e_{k}\\
 & =E_{i_{0},\varepsilon_{0}}(t_{0})\cdots E_{i_{r},\varepsilon_{r}}(t_{r})\cdot \\
&\qquad \left(E_{i_{r},\varepsilon_{r}}(t_{r})\cdots E_{i_{0},\varepsilon_{0}}(t_{0})\tB(t_{0})F_{i_{0},\varepsilon_{0}}(t_{0})\cdots F_{i_{r},\varepsilon_{r}}(t_{r})\right)\cdot e_{k}\\
 & =\tB(t_{0})\cdot F_{i_{0},\varepsilon_{0}}(t_{0})\cdots F_{i_{r},\varepsilon_{r}}(t_{r})\cdot e_{k}\\
 & =\tB(t_{0})\cdot C^{t_{0}}(t_{r+1})\cdot e_{k}\\
 & =\tB(t)\cdot c_{k}^{t}(t')
\end{align*}

\end{proof}

Assume the cluster algebra is injective-reachable. Then for any seed
$t$ we have seeds $t[1]$ and $t[-1]$ constructed from $t$ by mutation
sequences. The following crucial result tells us that the linear map
$\psi_{t[-1],t}$ reverses the dominance order in $t$ and $t[-1]$.

\begin{Prop}[order reverse]\label{prop:order_reverse}

Given an injective-reachable seed $t=\overleftarrow{\mu}t[-1]$ such
that $C^{t[-1]}(t)=-\Perm_{\sigma}$ for some permutation $\sigma$
of $I_{\ufv}$. Given any $\eta,g\in\Mc(t)$. Then we have $\eta=g+\tB(t)\cdot n$
for some $n\in\yCone(t)$ if and only if $\eta'=g'+\tB(t[-1])\cdot(-\Perm_{\sigma}\cdot n)$
where $\eta'=\psi_{t[-1],t}\eta$ and $g'=\psi_{t[-1],t}g$. In particular,
we have $\eta\preceq_{t}g$ if and only if $\psi_{t[-1],t}\eta\succeq_{t[-1]}\psi_{t[-1],t}g$.

\end{Prop}

\begin{proof}

Notice that $\psi_{t[-1],t}$ is  a bijective linear map from $\Mc(t)$
to $\Mc(t[-1])$ by Lemma \ref{lem:bijective_linear_map}. The claim
is equivalent to $\psi_{t[-1],t}(\tB(t)\cdot n)=\tB(t[-1])\cdot(-\Perm_{\sigma}\cdot n)$. Also, recall that $\deg^{t}(y(t)^{n})=\tB(t)\cdot n$.

Applying
the linear map $\psi_{t[-1],t}:\Mc(t)\rightarrow\Mc(t[-1])$ and
using Lemma \ref{lem:linear_transform_monomial} Proposition \ref{prop:c_vector_y_deg}, we obtain

\begin{align*}
\psi_{t[-1],t}(\tB(t)\cdot n)= & \psi_{t[-1],t}\deg^{t}(y(t)^{n})\\
= & \sum_{k}\psi_{t[-1],t}\deg^{t}(y_{k}(t))\cdot n_{k}\\
(\mathrm{Lemma\ }\ref{lem:linear_transform_monomial})= & \sum_{k}\deg^{t[-1]}\overleftarrow{\mu}^{*}(y_{k}(t))\cdot n_{k}\\
(\mathrm{Proposition\ }\ref{prop:c_vector_y_deg})= & \sum_{k}\deg^{t[-1]}y(t[-1])^{c_{k}^{t[-1]}(t)}\cdot n_{k}\\
= & \deg^{t[-1]}y(t[-1])^{C^{t[-1]}(t)\cdot n}\\
= & \tB(t[-1])\cdot(-\Perm_{\sigma}\cdot n).
\end{align*}

\end{proof}

We have the following consequence which tells us that the degree and
codegree in $t$ and $t[-1]$ swap.

\begin{Prop}[degree/codegree swap]\label{prop:cpt_pt_swap}

Let there be given an injective-reachable seed $t=\overleftarrow{\mu}t[-1]$
and any $z\in\LP(t)$ such that $\overleftarrow{\mu}^{*}z\in\LP(t[-1])$.
Then $z$ is copointed in $\LP(t)$ at the codegree $\codeg^{t}z=\eta$
if and only if $\overleftarrow{\mu}^{*}z$ is pointed at $\LP(t[-1])$
with the degree $\deg^{t[-1]}(\overleftarrow{\mu}^{*}z)=\psi_{t[-1],t}\eta$.

\end{Prop}

\begin{proof}

Let us denote the Laurent expansion of $z$ in $\LP(t)$ by $z=\sum_{m\in\Mc(t)}b_{m}x(t)^{m}$,
where only finitely many coefficients $b_{m}$ are non-zero. Taking
the formal Laurent expansion in $\widehat{\LP(t[-1])}$, we obtain
$\overleftarrow{\mu}^{*}z=\sum_{m\in\Mc(t)}\overleftarrow{\mu}^{*}(b_{m}x(t)^{m})$.

Each formal Laurent series $\overleftarrow{\mu}^{*}(x(t)^{m})$ in
$\widehat{\LP(t[-1])}$ has the degree $\psi_{t[-1],t}m$ by Lemma
\ref{lem:linear_transform_monomial}. On the one hand, $z$ is copointed
at $\eta$ if and only if $\{m|b_{m}\neq0\}$ has a unique $\prec_{t}$-minimal
element $\eta$ and $b_{\eta}=1$. On the other hand, $\overleftarrow{\mu}^{*}z$
is pointed at some degree $g$ if and only if $\{\psi_{t[-1],t}m|b_{m}\neq0\}$
has a unique $\prec_{t[-1]}$-maximal element $g=\psi_{t[-1],t}\eta$
and $b_{\eta}=1$. Because $\psi_{t[-1],t}$ reverses the order $\preceq_{t}$
and $\preceq_{t[-1]}$ by Proposition \ref{prop:order_reverse}, these
two conditions are equivalent. 

\end{proof}

\subsection{Support of bipointed Laurent polynomials}

\begin{Def}[Support]\label{def:f_dim}

The support of any $n=\sum n_{k}e_{k}\in\yCone(t)$ is defined to
be the set of vertices $\supp n=\{i\in I_{\ufv}|n_{i}\neq0\}$. 

Given any Laurent polynomial $z\in\LP(t)$ with bidegree $(\eta,g)$.
Its support dimension $\suppDim^{t}z$ is defined to be the unique
element $n\in\yCone^{\geq0}(t)$ such that $\eta=g+p^{*}n$. We define
its support to be $\supp^{t}z=\supp(n)$.

\end{Def}

Recall that, for any seeds $t'=\overleftarrow{\mu}t$, the mutation
map $\overleftarrow{\mu}^{*}$ identifies $\cF(t')$ and $\cF(t)$,
and $\LP(t')\cap\LP(t)$ denote $\LP(t')\cap(\overleftarrow{\mu}^{*})^{-1}\LP(t)$.

\begin{Def}\label{def:compatably_pointed}

Let $S$ be any given set of seeds connected by mutations. A Laurent
polynomial $z\in\cap_{t_{i}\in S}\LP(t_{i})$ is said to be compatibly
pointed at the seeds in $S$, if we have $z\in\cap_{t\in S}\ptSet^{t}(g(t))$
for some degrees $g(t)\in\Mc(t)$ such that $g(t')=\phi_{t',t}g(t)$
for all $t,t'\in S$.

\end{Def}

Similarly, given any formal Laurent series $z\in\widehat{\LP(t_{0})}$,
$t_{0}\in S$, such that its formal Laurent expansion in $\widehat{\LP(t)}$
are well defined for all $t\in S$ (NOT always true). We can say $z$
is compatibly pointed at the seeds in $S$, if $z$ is pointed at
degrees $g(t)\in\Mc(t)$ in $\widehat{\LP(t)}$ such that $g(t')=\phi_{t',t}g(t)$
for all $t,t'\in S$.

\begin{Eg}
Let us give an example of an element $z$ in the upper cluster algebra which is NOT compatibly pointed at all seeds. 

Consider the classical case $\kk=\Z$. Take a type $A_2$ cluster algebra, whose initial seed $t$ consists of the initial cluster variables $x_1,x_2$ and the initial $B$-matrix $B=\begin{pmatrix}
0&-1\\
1&0
\end{pmatrix}$.
Denote $y_1=x_2$ and $y_2=x_1^{-1}$.

Applying the mutation $\mu_1$ to $t$, we obtain a new seed $t'=\mu_1 t$ with new variables $x_1'=x_1^{-1}(1+y_1)=x_1^{-1}(1+x_2)$ and $x_2'=x_2$, $B'=\begin{pmatrix}
0&1\\
-1&0
\end{pmatrix}$, $y_1'=(x_2')^{-1}=y_1^{-1}$, $y_2'=x_1'$.

Define $z:=x_1\cdot  x_1'=1+x_2=1+y_1$. Then $z$ lies in the upper cluster algebra. It is $0$-pointed in the seed $t_0$, but its leading term comes from the contribution of $x_2=x_2'$ in the seed $t'$. In particular, it is not $0$-pointed in the seed $t'$, i.e. not compatibly pointed at the seeds $\{t,t'\}$.
\end{Eg}

Next, we define the ``correct'' support dimension for bipointed
Laurent polynomials, as we shall show in Proposition \ref{prop:bipointed_support}.

\begin{Def}\label{def:support_dimension}

Given any injective-reachable seed $t$ and $g\in\Mc(t)$. If there exists $n\in\yCone^{\geq0}(t)$
such that
\begin{align*}
\eta & =g+\tB(t)\cdot n
\end{align*}
 where $\eta=\psi_{t[-1],t}^{-1}\phi_{t[-1],t}g$, we
define the support dimension associated to $g$ to be 
\begin{align*}
\suppDim g & =n
\end{align*}
 and the bidegree interval associated to $g$ to be the following
subset of $\Mc(t)$:

\begin{align*}
\intv_{g} & =_{\eta\preceq_{t}}\Mc(t)_{\preceq_{t}g}.
\end{align*}

Given any tropical point $[g]\in\tropSet$. If for all $t\in\Delta^{+}$,
$g\in\Mc(t)$ has a support dimension, where $[g]=g$ under the
identification $\tropSet\simeq\Mc(t)$, then we say $[g]$ has the support
dimensions.

\end{Def}

Notice that the support dimension $\suppDim g $ is well-defined if and only if
$\psi_{t[-1],t}^{-1}\phi_{t[-1],t}g\preceq_{t}g$. It will turn out that it is always well-defined by Proposition \ref{prop:all_support_defined} and the existence of generic cluster characters (for skew-symmetric cases) or the existence of theta functions (for skew-symmetrizable cases).

\begin{Rem}

We claim that the support dimension $\suppDim g$ and $\intv_{g}$
do not depend on the choice of $t[-1]$ up to permutations $\sigma$
of $I_{\ufv}$. To see this, for any permutation $\sigma$, we introduce
the index relabelling operation $\sigma$ on the seed $t$ which generates
a new seed $\sigma t=((b_{\sigma i,\sigma j})_{i,j\in I},(x_{\sigma i}(t))$.
Then $\sigma$ commutes with $\phi_{t,t'}$, $\psi_{t,t'}$, and induces
automorphisms on fraction fields which commute with mutations. The
claim follows from direct comparison between different choices of
$t[-1]$ via the relabelling $\sigma$.

\end{Rem}

The following result tells us that the subset $\tropSet_{\preceq_{\{t,t[-1]\}}[g]}$
of tropical points could be described by the inclusion of the bidegree
intervals. Notice that the inclusion gives a natural partial order
bounded from below, and it will be crucial when we construct bases
later.

\begin{Prop}[Inclusion property]\label{prop:inclusion_description}

Let there be given an injective-reachable seed $t=\overleftarrow{\mu}t[-1]$
and $g,g'\in\Mc(t)$ with support dimension. 

(1) We have $g'\prec_{t}g,\ \phi_{t[-1],t}g'\prec_{t[-1]}\phi_{t[-1],t}g$
if and only if $\intv_{g'}\subsetneq\intv_{g}$.

(2) Under the assumption in (1), we have $\Mc(t)_{\prec_{\{t,t[-1]\}}g}=\{g'\in\Mc(t)|\intv_{g'}\subsetneq\intv_{g}\}$
for any $g\in\Mc(t)$. In addition, $\Mc(t)_{\prec_{\{t,t[-1]\}}g}$
is finite.

\end{Prop}

\begin{proof}

(1) By Proposition\ref{prop:order_reverse}, $\phi_{t[-1],t}g'\prec_{t[-1]}\phi_{t[-1],t}g$
is equivalent to $\psi_{t[-1],t}^{-1}\phi_{t[-1],t}g'\succ_{t}\psi_{t[-1],t}^{-1}\phi_{t[-1],t}g$.
Because $g,g'$ have support dimensions, we have $g\succ\psi_{t[-1],t}^{-1}\phi_{t[-1],t}g$
and $g'\succ_t\psi_{t[-1],t}^{-1}\phi_{t[-1],t}g'$. The claim follows
from definition of the bidegree intervals $\intv_{g'},\ \intv_{g}$.

(2) The first claim follows from (1). Notice that $\intv_{g}$ is
finite by Lemma \ref{lem:finite_interval} and $g'\in\intv_{g}$ for
any $g'\in\Mc(t)_{\prec_{\{t,t[-1]\}}g}$, the second claim follows.

\end{proof}

\begin{Rem}

It might be possible to generalize the notion of support dimensions
by removing the restriction $n\in\yCone^{\geq0}(t)$. It is also an
interesting question to write down the mutation rule of these dimensions,
see \cite{fujiwara2018duality} for a formula for the support dimensions
for cluster variables (called $f$-vectors).

\end{Rem}

The following result gives an equivalence between being bipointed with
the ``correct'' support dimension and being compatibly pointed at
$t,t[-1]$.

\begin{Prop}[Compatibility and support dimensions]\label{prop:bipointed_support}

Given seeds $t=\overleftarrow{\mu}t[-1]$ and a pointed Laurent polynomial
$z\in\ptSet^{t}(g)$, $g\in\Mc(t)$.

(1) If $z$ is compatibly pointed at seeds $t,t[-1]$, then $g$ has
a support dimension. Moreover, $z$ is bipointed with $\suppDim z=\suppDim g$
in this case.

(2) If $g$ has a support dimension and $z\in\LP(t)$ is bipointed
with $\suppDim z=\suppDim g$, then $z$ is compatibly pointed at
seeds $t,t[-1]$.

\end{Prop}

\begin{proof}

(1) By Proposition \ref{prop:cpt_pt_swap}, we know that $z\in\LP(t)$
is copointed with codegree $\psi_{t[-1],t}^{-1}\deg^{t[-1]}\overleftarrow{\mu}^{*}z$,
which equals $\psi_{t[-1],t}^{-1}\phi_{t[-1],t}g$ because $z$ is
compatibly pointed at seeds $t,t[-1]$. The claims follow.

(2) By definition, $z$ is bipointed at bidegree $(g,\psi_{t[-1],t}^{-1}\phi_{t[-1],t}g)$.
By Proposition \ref{prop:cpt_pt_swap}, we know that $\overleftarrow{\mu}^{*}z$
is pointed with degree $\phi_{t[-1],t}g$.

\end{proof}

Recall that we have the following result which tells us that a finite
decomposition of pointed Laurent series is unitriangular.

\begin{Lem}\cite[Lemma 3.1.10(iii)]{qin2017triangular}\label{lem:triangular_decomposition}

Given any finite linear decomposition of pointed formal Laurent series
$u,z_{j}$ in $\widehat{\LP(t)}$, where $z_{j}$ have distinct degrees:

\begin{align*}
u & =\sum_{0\leq j\leq r}b_{j}z_{j},
\end{align*}
with $r\in\N$ and the coefficients $b_{j}\in\kk$. Then the decomposition
must be $\prec_{t}$-unitriangular, i.e., we can reindex $z_{j}$
such that $u=z_{0}+\sum_{1\leq j\leq r}b_{j}z_{j}$, with $b_{0}=1$,
$\deg^{t}z_{0}=\deg^{t}u$ and $\deg^{t}z_{j}\prec_{t}\deg^{t}u$
for all $j\geq1$.

\end{Lem}

We have a better control of a finite decomposition of Laurent polynomials
compatibly pointed at $t,t[-1]$ (or, equivalently, bipointed with
correct support dimensions by Proposition \ref{prop:bipointed_support}).

\begin{Prop}[decomposition]\label{prop:decompose_bipointed}

Given seeds $t=\overleftarrow{\mu}t[-1]$ and any finite decomposition
of pointed Laurent polynomials $u,z_{j}$ in $\LP(t)$, where $z_{j}$ have distinct degrees:

\begin{align*}
u & =\sum_{0\leq j\leq r}b_{j}z_{j},
\end{align*}
with $\deg^{t}z_{0}=\deg^{t}u$ and all coefficients $b_{j}\neq0$.
We further assume that all $u,z_{j}$ are compatibly pointed at $t$
and $t[-1]$. Then the following claims are true:

(1) All $u,z_{j}$ are bipointed.

(2) We have $\deg^{t}u=\deg^{t}z_{0}$ and $\deg^{t}z_{0}\succ_{t}\deg^{t}z_{j}$
for all $j>0$. 

(3) We have $\codeg^{t}u=\codeg^{t}z_{0}$ and $\codeg^{t}z_{j}\succ_{t}\codeg^{t}z_{0}$
for all $j>0$.

(4) We have inclusion between bidegree intervals $\intv_{\deg^{t}z_{j}}\subsetneq\intv_{\deg^{t}z_{0}}$
for all $j>0$.

(5) We have $\suppDim\deg^{t}z_{j}<\suppDim\deg^{t}z_{0}$ for all
$j>0$ in $\yCone^{\geq0}(t)$.

\end{Prop}

\begin{proof}

(1) Because $u,z_{j}$ are compatibly bipointed at $t,t[-1]$, we
can apply Proposition \ref{prop:bipointed_support}. As consequences,
$\deg^{t}u$ has a support dimension $\suppDim u=\suppDim\deg^{t}u$,
$u$ is bipointed at the bidegree $(\deg^{t}u,\psi_{t[-1],t}^{-1}\phi_{t[-1],t}\deg^{t}u)$,
all $\deg^{t}z_{j}$ have support dimensions $\suppDim z_{j}=\suppDim\deg^{t}z_{j}$,
and all $z_{j}$ are bipointed at bidegree $(\deg^{t}z_{j},\psi_{t[-1],t}^{-1}\phi_{t[-1],t}\deg^{t}z_{j})$. 

(2) This claim follows from Lemma \ref{lem:triangular_decomposition}.

(3) Because $\deg^{t}u=\deg^{t}z_{0}$, $u$ and $z_{0}$ must have
the same codegree $\psi_{t[-1],t}^{-1}\phi_{t[-1],t}\deg^{t}u=\psi_{t[-1],t}^{-1}\phi_{t[-1],t}\deg^{t}z_{0}$.
Because $u=\sum b_{j}z_{j}$ is a finite decomposition, the $\prec_{t}$-minimal
Laurent degree $\codeg^{t}u$ of $u$ must be the $\prec_{t}$-minimal
element of $\{\codeg^{t}z_{j},\forall j\}$. Therefore, $\codeg^{t}z_{j}\succ_{t}\codeg^{t}z_{0}$
$\forall j>0$.

(4) The claim follows from (2)(3).

(5) By (4), for any $j>0$, we have 
\begin{align*}
\deg^{t}z_{0}\succ_{t}\deg^{t}z_{j}\succeq_{t}\codeg^{t}z_{j}\succ_{t}\codeg^{t}z_{0}.
\end{align*}
Therefore, there exists $n_{1},n_{2},n_{3}\in\yCone^{\geq0}(t)$,
with $n_{1},n_{3}\neq0$, such that 
\begin{align*}
\deg^{t}z_{j} & =\deg^{t}z_{0}+\tB(t)n_{1}\\
\codeg^{t}z_{j} & =\deg^{t}z_{j}+\tB(t)n_{2}\\
\codeg^{t}z_{0} & =\codeg^{t}z_{j}+\tB(t)n_{3}.
\end{align*}
We obtain $\suppDim z_{j}=n_{2}<n_{1}+n_{2}+n_{3}=\suppDim z_{0}$.

\end{proof}

Conversely, by slightly changing the statement in Proposition \ref{prop:decompose_bipointed},
we describe a finite sum of pointed Laurent polynomials with well
controlled bidegrees.

\begin{Prop}[combination]\label{prop:sum_bipointed}

Given seeds $t=\overleftarrow{\mu}t[-1]$ and any finite decomposition
of Laurent polynomials $u,z_{j}$ in $\LP(t)$

\begin{align*}
u & =\sum_{0\leq j\leq r}b_{j}z_{j}
\end{align*}
with coefficients $b_{j}\neq0$. We further assume that all $z_{j}$
are compatibly pointed at $t,t[-1]$ and their bidegrees satisfy $\intv_{\deg^{t}z_{j}}\subsetneq\intv_{\deg^{t}z_{0}}$
for all $j>0$. 

Then $u$ is compatibly pointed at $t$, $t[-1]$, bipointed at $\LP(t)$
with bidegree $(\codeg^{t}z_{0},\deg^{t}z_{0})$, and has a support
dimension $\suppDim u=\suppDim z_{0}$.

\end{Prop}

\begin{proof}

By the inclusion assumption on bidegrees of $z_{j}$, $u$ must be
bipointed at bidegree $(\codeg^{t}z_{0},\deg^{t}z_{0})$ with support
$\suppDim u=\suppDim z_{0}$. Because $z_{0}$ is compatibly pointed
at $t,t[-1]$, $\deg^{t}z_{0}$ has a support dimension $\suppDim\deg^{t}z_{0}=\suppDim z_{0}$
by Proposition \ref{prop:bipointed_support}. Consequently, $u$ is
compatibly pointed at seeds $t,t[-1]$ by Proposition \ref{prop:bipointed_support}(2).

\end{proof}

Finally, we discuss properties of localized cluster monomials. Given
seeds $t'=\overleftarrow{\mu}t$ and a localized cluster monomial
$x(t')^{d}$ where $d\in\N^{I_{\ufv}}\oplus\Z^{I_{\fv}}$. Recall
that its Laurent expansion in $\LP(t)$ is computed as $\overleftarrow{\mu}^{*}x(t')^{d}$.

\begin{Lem}\label{lem:compatible_at_g_vector}

If any $z\in\LP(t)$ has degree $\deg^{t}z=\deg^{t}\overleftarrow{\mu}^{*}x(t')^{d}$
and is compatibly pointed at $\{t,t',t'[-1]\}$, then $z=\overleftarrow{\mu}^{*}x(t')^{d}$.

\end{Lem}

\begin{proof}

We have $\deg^{t'}(\overleftarrow{\mu}^{-1})^{*}z=\phi_{t',t}\deg^{t}z=\phi_{t',t}\deg^{t}\overleftarrow{\mu}^{*}x(t')^{d}=d$.
Therefore, $(\overleftarrow{\mu}^{-1})^{*}z$ and $x(t')^{d}$ have
the same degree in $\LP(t')$. Because they are compatibly pointed
in $\{t',t'[-1]\}$, by Proposition \ref{prop:bipointed_support},
they have the same support dimension, which is given by $\suppDim x(t')^{d}=0$.
Consequently, we have $(\overleftarrow{\mu}^{-1})^{*}z=x(t')^{d}$.

\end{proof}

It is natural to ask if we can extend the above property without the
injective-reachability assumption. 

The following property is known without this assumption.

\begin{Prop}\cite[Proposition 5.3]{FominZelevinsky07}\label{prop:cluster_monomial_bipointed}

For general initial seed $t_{0}$, the Laurent expansion $\overleftarrow{\mu}^{*}x_{i}(t')^{d}\in\LP(t)$
is bipointed.

\end{Prop}

\section{Properties of $\prec_{t}$-decompositions\label{sec:Properties-of--decompositions}}

\subsection{$\prec_{t}$-decompositions}

Given a seed $t=((b_{ij})_{i,j\in I},(x_{i})_{i\in I})$ and a collection
$\cS=\{s_{g}|g\in\Mc(t)\}\subset\widehat{\LP(t)}$ such that $s_{g}$
is pointed at $g$. By definition, any $z=\sum_{g\in\Mc(t)}b_{g}x^{g}\in\widehat{\LP(t)}$
has finitely many $\prec_{t}$-maximal Laurent degrees. Similar to
\cite[Lemma 3.1.10(i) Remark 3.1.8]{qin2017triangular}, we can decompose
$z$ in terms of the pointed elements in $S$ inductively via the
partial order $\prec_{t}$.

\begin{DefLem}[Dominance order decomposition]\label{def:dominance_order_decomposition}

There exists a unique decomposition

\begin{eqnarray}
z & = & \sum_{g\in\Mc(t)}\alpha_{t}(z)(g)\cdot s_{g},\quad\alpha_{t}(z)\in\Hom_{\mathrm{set}}(\Mc(t),\ \kk),\label{eq:dominance_localized_t_decompositon}
\end{eqnarray}
in $\widehat{\LP(t)}$ for some coefficient function $\alpha_{t}(z)$ such
that the support $\supp(\alpha_{t}(z)):=\{g|\alpha_{t}(z)(g)\neq0\}$ has finitely
many $\prec_{t}$-maximal elements. We call it the $\prec_{t}$-decomposition
of $z$ into elements of $\cS$.

\end{DefLem}

\begin{proof}

Let $g^{(j)}$, $1\leq j\leq l$, $0\neq j\in\N$ denote the $\prec_{t}$-maximal
Laurent degrees of $z$. If \eqref{eq:dominance_localized_t_decompositon}
holds, by comparing the Laurent monomials with $\prec_{t}$-maximal
degrees on both sides, we deduce that the $\prec_{t}$-maximal elements
of $\supp(\alpha_{t}(z))$ are exactly $g^{(j)}$, $1\leq j\leq l$,
and their coefficients must be $\alpha_{t}(z)(g^{(j)})=b_{g^{(j)}}$.

Let us draw a directed graph $G$ such that its vertices are $\cup_{1\leq j\leq l}\Mc(t)_{\preceq_{t}g^{(j)}}$
and, whenever $g'=g+\tB\cdot e_{k}$ for some $k\in I_{\ufv}$, we
draw an arrow from $g$ to $g'$. Then there is a (probably length
$0$) path from $g$ to $g'$ if and only if $g'\preceq_{t}g$. 

Notice that the source points of $G$ are the leading degrees $g^{(j)}$.
Moreover, for any vertex $g'$, there exists finitely many vertices
$g$ in $G$ such that $g'\preceq_{t}g$ by the Finite Interval Lemma
\ref{lem:finite_interval}. Then the decomposition coefficients for
general vertices $g\in G$ are inductively determined by travelling
further away from the source points, see \cite[Remark 3.1.8]{qin2017triangular}.

\end{proof}

\subsection{Change of seeds.}\label{sec:change_seed}

We want to show the desired property that the $\prec_{t}$-decomposition
is independent of the seed $t$, provided $S$ satisfy some tropical
properties. We learn from the inspirational paper \cite[Section 6]{gross2018canonical}
to give a proof based on the nilpotent Nakayama Lemma.

The idea of the proof is straightforward for the principal coefficient cases in the sense of \cite{FominZelevinsky07}. Endow such (partially compactified) cluster algebras with natural adic-topologies. Then the nilpotent Nakayama Lemma provides a method to verify that a given collection of elements is a basis. Our proof looks more technical because it treats general cases, and we need to modify calculation for the principal coefficient cases in the spirit of the correction technique (\cite[Section 9]{Qin12} or \cite[Section 4]{qin2017triangular}).

Given $k\in I_{\ufv}$. We denote the mutated seed $t'=\mu_{k}t=((b_{ij}'),(x_{i}'))$.
Recall that we have the tropical transformation $\phi=\phi_{t',t}:\Mc(t)\simeq\Mc(t')$.
For any $g\in\Mc(t)$, denote $g'=\phi_{t',t}g$ for simplicity.

For simplicity, let us assume $z\in\LP(t)\cap\LP(t')$ and $\cS\subset\LP(t)\cap\LP(t')$,
which is sufficient for this paper. Further assume that the collection
$\cS=\{s_{g}|g\in\Mc(t)\}$ is compatibly pointed at the seeds $t$,
$t'$, i.e., $s_{g}$ is pointed at $g'$ in $\LP(t')$. Then we have
a (possibly infinite) $\prec_{t'}$-decomposition in $\LP(t')$:

\begin{eqnarray}
z & = & \sum_{g'\in\Mc(t')}\alpha_{t'}(z)(g')\cdot s_{g},\quad\alpha_{t'}(z)\in\Hom_{\mathrm{set}}(\Mc(t'),\kk).\label{eq:dominance_localized_t_prime_decompositon}
\end{eqnarray}

The aim of this section is to prove the following result.

\begin{Prop}\label{prop:invariant_decomposition}

We have $\alpha_{t}(z)(g)=\alpha_{t'}(z)(g')$ $\forall g\in\Mc(t)$.
In particular, we have $\phi\supp(\alpha_{t}(z))=\supp(\alpha_{t'}(z))$.

\end{Prop}

Our strategy is to use the nilpotent Nakayama Lemma \cite[Theorem 8.4]{Matsumura86}
as in \cite{gross2018canonical}, and compare the collection $\cS$
with the natural basis of the type $A_{1}$ cluster algebra $\LP(t)\cap\LP(t')$
using the tropical properties (Lemma \ref{lem:compatible_at_g_vector}).

\begin{Lem}[Nilpotent Nakayama Lemma]\label{lem:Nakayama}

Let $A$ denote a ring, $\mm$ its nilpotent $2$-sided ideal such that $\mm^{r}=0$,
and $U$ its left module. For any subset $S$ of $U$, if its image in
$U/\mm U$ generates $U/\mm U$ as an $A/\mm$-module, then $S$ generates
$U$ as an $A$-module.

\end{Lem}

\begin{proof}

We learn the following proof from Matthew Emerton. By assumption,
$U=AS+\mm U$. Repeat the substitution, we get

\begin{align*}
U & =AS+\mm(AS+\mm U)\\
 & =AS+\mm(AS+\mm(AS+\mm U))\\
 & =\cdots\\
 & =AS+\mm S+\mm^{2}S+\cdots+\mm^{r-1}S+\mm^{r}U\\
 & =AS+\mm S+\mm^{2}S+\cdots+\mm^{r-1}S.
\end{align*}

The claim follows.

\end{proof}

For applying the Nakayama Lemma, we want to work with an $\mm$-adic
topology where the ideal $\mm$ is generated by the $y$-variables.
Correspondingly, it is convenient to add extra principal framing frozen
vertices $I'=\{i'|i\neq k,i\in I_{\ufv}\}$. Extending the vertex
set $I$ to $\tI=I\sqcup I'$. Extend the matrix $(b_{ij})_{i,j\in I}$
to $(b_{ij})_{i,j\in\tI}$ such that, for $i\neq k$, $i\in I_{\ufv}$,
\begin{eqnarray*}
b_{i',i} & =&1 \\
b_{i,i'} & =&-1
\end{eqnarray*}
and other entries are extended by zero. We obtain the principal framing seed $t^{\prin}=((b_{ij})_{i,j\in\tI},(x_{i})_{i\in\tI})$, which is said to have (a modified version of) the \emph{principal coefficients} in the sense of \cite{FominZelevinsky07}. 
Then its mutated seed $(t^{\prin})':=\mu_{k}(t^{\prin})$ agrees with
the principal framing $(t')^{\prin}$ of $t'$. 

When working with the quantum case $\kk=\Z[q^{\pm\Hf}]$, we extend the compatible bilinear form $\lambda$ on $\Mc(t)$ to $\Mc(t^{\prin})$ by zero. The resulting bilinear form on $\Mc(t^{\prin})$, still denoted by $\lambda$, is compatible with $t^{\prin}$.

We have the natural
embedding $\Mc(t)\simeq\Mc(t)\oplus0\subset\Mc(t^{\prin})$.
Conversely, for any $\tg$ from the extended degree lattice $\Mc(t^{\prin})$, denote its projection
to $\Mc(t)$ by $g$. Denote $\phi_{(t')^{\prin},t^{\prin}}\tg=\tg'$.

Notice that the $y$-variables in $t^{\prin}$ and $t$ satisfy $$y_{i}(t^{\prin})=\begin{cases}
x_{i'}\cdot y_{i} & i\neq k\in I_{\ufv}\\
y_{k} & i=k
\end{cases},$$ and the same formula holds for $(t')^{\prin}$ and $t'$. Define
the grading $\gr(\ )$ on $\Mc(t^{\prin})$ such that $\gr(f_{i})=\begin{cases}
1 & i\in I'\\
0 & i\notin I'
\end{cases}$, and similarly $\gr'(f_{i}')=\begin{cases}
1 & i\in I'\\
0 & i\notin I'
\end{cases}$ on $\Mc((t')^{\prin})$. Then $\phi:\Mc(t^{\prin})\simeq\Mc((t')^{\prin})$
is homogeneous, i.e., $\gr(\tg)=\gr'(\tg')$. We have the following
observation.

\begin{Lem}\label{lem:dominance_refine_grading}

If $\tilde{\eta}=\tg+\tB\cdot n$ in $\Mc(t^{\prin})$ for some
$n\in N_{\ufv}^{\geq0}(t^{\prin})$, then $\gr(\tilde{\eta})\geq\gr(\tg)$.
Moreover, we have $\gr(\tilde{\eta})>\gr(\tg)$ if and only if $n_{i}>0$
for some $i\neq k,i\in I_{\ufv}$.

\end{Lem}

We have an induced grading $\gr$ on $\LP(t^{\prin})$ such that $\gr(x_{i}):=\gr(f_{i})$
and similarly $\gr'$ on $\LP((t')^{\prin})$.

The intersection $\mathcal{\upClAlg}_{k}:=\LP(t^{\prin})\cap\LP((t')^{\prin})$
is the (type $A_{1}$) upper cluster algebra obtained from the initial
seed $t^{\prin}$ such that $k$ is the only unfrozen vertex. It is well-known that it has the basis
$\{m_{\tg}|\tg\in\Mc(t^{\prin})\}$ where $m_{\tg}$ are its localized
cluster monomials with degree $\tg$. Recall that, for the classical case $\kk=\Z$, $m_{\tg}=x^{\tg}(1+y_{k})^{[-g_{k}]_{+}}$
for this type $A_{1}$ upper cluster algebra (see Section \ref{sec:quantization} for the quantum case $\kk=\Z[q^{\pm\Hf}]$). In particular, $m_{\tg}$ has
homogeneous grading $\gr(\tg)$ in $\LP(t^{\prin})$. Similarly, $m_{\tg}$
has homogeneous grading $\gr'(\tg')=\gr(\tg)$ in $\LP((t')^{\prin})$.
Therefore, the two gradings in $\LP(t^{\prin})$ and $\LP((t')^{\prin})$
give the same grading on the algebra $\upClAlg_{k}$.

\begin{Lem}\label{lem:homogenous_Laurent}

Given any element $z\in\LP((t')^{\prin})$ and decompose $z=\sum z_{i}$
into homogeneous parts $z_{i}\in\LP((t')^{\prin})$ of different gradings.
Then $\mu_{k}^{*}z\in\LP(t^{\prin})$ if and only if all $\mu_{k}^{*}z_{i}\in\LP(t^{\prin})$.

\end{Lem}

\begin{proof}

If $z\in\LP(t^{\prin})\cap\LP((t')^{\prin})$, then we can decompose
it into a finite sum $z=\sum\alpha_{\tg}m_{\tg}$. Since $m_{\tg}$
are homogeneous, we obtain that $z_{i}=\sum_{\gr(\tg)=\gr(z_{i})}\alpha_{\tg}m_{\tg}$.
In particular, $z_{i}\in\LP(t^{\prin})\cap\LP((t')^{\prin})$. The
converse statement is trivial.

\end{proof}

Take any $\tg\in \Mc(t^{\prin})$. Since $s_{g}\in\cS$ is pointed at $g$, it takes the form $s_{g}=x^{g}\cdot F_{g}((y_{i})_{i\in I_{\ufv}})$
where $F_{g}(\ )$ is a multivariate polynomial with constant $1$ and we use the commutative product.
Correspondingly, define $s_{\tg}:=x^{\tg}\cdot F_{g}((y_{i}(t^{\prin}))_{i\in I_{\ufv}})$
and $\tcS:=\{s_{\tg}|\tg\in\Mc(t^{\prin})\}$. Note that $s_{\tg}$ belongs to $\LP(t^{\prin})$ and $\LP((t')^{\prin})$, see \cite[Theorem 9.2]{Qin12}.

\begin{Lem}\label{lem:dominance_under_projection}

If $\tg'=\tg+\tB(t^{\prin})\cdot n$ in $\Mc(t^{\prin})$ for some
$0\neq n\in\N^{I_{\ufv}}$, then $g'=g+\tB(t)\cdot n$.

\end{Lem}

\begin{proof}

The claim follows by taking the projection from $\Mc(t^{\prin})\rightarrow\Mc(t)$.

\end{proof}

\begin{Lem}\label{lem:tg_compatible_pointed}

$s_{\tg}$ is compatibly pointed at $\tg$ and $\tg'$ at seeds $t^{\prin}$and
$(t')^{\prin}$ respectively.

\end{Lem}
\begin{proof}

(i) Denote $\psi=\psi_{t',t}$, $\tilde{\psi}=\psi_{(t')^{\prin},t^{\prin}}$
for simplicity. By Lemma \ref{lem:diff_psi_phi}, we have $\psi g-g'=(\tB')\cdot[-g_{k}]_{+}e'_{k}$, where $e'_{k}$ denote the $k$-th unit vector in $N_{\ufv}(t')$. Similarly, we have $\tilde{\psi}\tg-\tg'=\tB((t')^{\prin}))\cdot[-g_{k}]_{+}e'_{k}$.

(ii)
Let $\mu_k^*$ denote the mutation map from $\LP(t)$ to $\LP(t')$. For any $Z=\sum c_n y^n\in \widehat{\kk[N_{\ufv}^{\geq 0}(t)]}$, we denote its evaluation $Z|_{y^n=x^{\tB n}}$ by $Z(x^{\tB n})$. Similarly, we denote $Z((x')^{\tB' n})=Z|_{(y')^n=(x')^{\tB' n}}$ for $Z\in \widehat{\kk[N_{\ufv}^{\geq 0}(t')]}$. Note that we have $\widehat{\kk[N_{\ufv}^{\geq 0}(t)]}=\widehat{\kk[N_{\ufv}^{\geq 0}(t^{\prin})]}$ and $\widehat{\kk[N_{\ufv}^{\geq 0}(t')]}=\widehat{\kk[N_{\ufv}^{\geq 0}((t')^{\prin})]}$.

By assumption, $s_g$ is compatibly pointed at $t$ and $t'$. Then there exists $F\in \kk[N_{\ufv}^{\geq 0}(t)]$ and $G\in \kk[N_{\ufv}^{\geq 0}(t')]$ with constant term $1$, such that $s_g=x^g * F( x^{\tB n})$ and $\mu_k^* s_g=(x')^{g'}*G((x')^{\tB' n})$.

Note that $\mu^* F$ and $G^{-1}$ are well defined in $\widehat{\kk[N_{\ufv}^{\geq 0}(t')]}$. By \eqref{eq:mutation} and \eqref{eq:q_mutation}, we can write $\mu_k^*(x^g)$ as $(x')^{\psi g}* Q(x^{\tB n})$, where $Q\in \widehat{\kk[N_{\ufv}^{\geq 0}(t')]}$ is a formal series in $y'_k$. Moreover, the mutation rule for $x^{\tB n}$ and $y^n$ are the same. We deduce that $\mu_k^*(F(x^{\tB n}))=(\mu_k^* F)((x')^{\tB' n})$. Then we have 
\begin{align*}
\mu_k^*(s_g)&=(x')^{\psi g}* Q((x')^{\tB' n})* \mu_k^*(F)((x')^{\tB' n})\\
&=(x')^{g'}*G(x^{\tB' n})
\end{align*}

(iii)
It follows that
\begin{align*}
(Q*\mu_k^*(F)*G^{-1})((x')^{\tB' n})&=(x')^{-\psi g}*(x')^{g'}\\
&=q^{\alpha}(x')^{-\psi g+g'}\\
&=q^\alpha (x')^{-\tB' [-g_k]_+ e'_k}
\end{align*}
Here, $q=1$ for the classical case $\kk=\Z$. For the quantum case $\kk=\Z[q^{\pm\Hf}]$, we have $\alpha=\Hf\lambda(-\psi g,g')$.

We explicitly compute that
\begin{align*}
2\alpha&=\lambda(g'-\psi g,g')\\
&=\lambda(g',\psi g-g')\\
&=-g'_k[-g_k]_+d'_k\\
&= g_k[-g_k]_+ d'_k.
\end{align*}
Similarly, we have $\lambda(\tg'-\tilde{\psi} \tg,\tg')= \tg_k [-\tg_k]_+d'_k$. Note that $\tg_k=g_k$. We deduce that
\begin{align}\label{eq:mutation_s_tg}
\begin{split}
(Q*\mu_k^*(F)*G^{-1})((x')^{\tB((t')^{\prin} n})&=q^\alpha (x')^{-\tB((t')^{\prin} )[-g_k]_+ e'_k}\\
&=(x')^{-\tilde{\psi}\tg}*(x')^{\tg'}
\end{split}
\end{align}

(iv) Let us apply mutation $\mu_k^*$ to $s_{\tg}=x^{\tg}*F(x^{\tB(t^{\prin}) n})$. Since $Q$ only depends on $g_k=\tg_k$, we deduce that $\mu_k^*(x^{\tg})=(x')^{\tilde{\psi} \tg}*Q((x')^{\tB((t')^{\prin}) n})$. Then \eqref{eq:mutation_s_tg} implies
\begin{align*}
\mu_k^*(s_{\tg})&=(x')^{\psi \tg}* Q((x')^{\tB((t')^{\prin}) n})* \mu_k^*(F)((x')^{\tB((t')^{\prin}) n})\\
&=(x')^{\tg'} *G((x')^{\tB((t')^{\prin})  n})
\end{align*}
In particular, $\mu_k^*(s_{\tg})$ is $\tg'$-pointed.

\end{proof}

Consider the following subalgebra of $\upClAlg_{k}$:
\begin{eqnarray*}
U_{k}: & = & \{z\in\upClAlg_{k}|z\mathrm{\ has\ no\ pole\ at\ }x_{i'}=0,\forall i\in I_{\ufv},i\neq k\}
\end{eqnarray*}
In fact, $U_{k}$ is a locally compactified version of the cluster algebra where the frozen variables $x_{i'}$, $i'\in I_{\ufv}$, are not invertible, and thus allows us to use the nilpotent Nakayama Lemma. Define $$C:=\{\tg\in\Mc(t^{\prin})|(\tg)_{i'}\geq0\ \forall i\in I_{\ufv},i\neq k\}.$$

\begin{Lem}\label{lem:optimized_seed_pole}

If $\tg\in C,$ then any $\tilde{\eta}\preceq_{t}\tg$ is contained
in $C$.

\end{Lem}

\begin{proof}

Notice that we have $\tilde{\eta}=\tg+\tB(t^{\prin})\cdot n$ for
$n\in\N^{I_{\ufv}}$ and column vectors of $\tB(t^{\prin})$ have
non-negative coordinates at $I'$. The claim follows.

\end{proof}

As a consequence, we have $s_{\tg}=x^{\tg}\cdot F_{\tg}((y_{i}(t^{\prin})_{i\in I_{\ufv}})\in U_{k}$
if and only if $\tg\in C$.

\begin{Prop}

The set $\{m_{\tg}|\tg\in C\}$ is a basis of $U_{k}$.

\end{Prop}

\begin{proof}
We have $m_{\tg}=x^{\tg}\cdot(1+y_{k})^{[-g_{k}]_{+}}$ for the classical case $\kk=\Z$. See Section \ref{sec:quantization} for the quantum case $\kk=\Z[q^{\pm\Hf}]$. We deduce that $m_{\tg}$ has a pole
at some $x_{i'}=0$ if and only if $\tg\notin C$.

For any $z\in U_{k}\subset\upClAlg_{k}$, we have a finite decomposition
$z=\sum b_{\tg}m_{\tg}$ in terms of the basis $\{m_{\tg}|\tg\in\Mc(t)\}$. Define the support $G=\{\tg | b_{\tg}\neq0 \}$.

Assume that $G\backslash C\neq\emptyset$. Let $\eta$ denote a $\prec_{t^{\prin}}$-maximal element in $G\backslash C$. Then $m_\eta$ contributes a Laurent monomial $b_{\tilde{\eta}}x^{\tilde{\eta}}$
with a pole at some $x_{i'}=0$. Since all $\tg$ from $C$ do not have a pole here, they do not contribute to the Laurent degree $\eta$. Since $\eta$ is maximal, other $m_\tg$ appearing with $\tg\notin C$ do not contribute to this degree neither. Therefore, $z$ has a pole here and does not belong to $U_{k}$. This
contradiction shows that every $z\in U_{k}$ is a finite combination
into $m_{\tg}$, $\tg\in C$. The claim follows.

\end{proof}

Define the graded polynomial ring $A=\kk[x_{i'}]_{i'\in I'}$ with
the grading $\gr(x_{i'})=1$ (endowed with the twisted product in the quantum case). Take its homogeneous decomposition $A=\oplus_{r\in\N}A^{r}$.
It has the maximal ideal $\mm:=\oplus_{r>0}A^{r}$. Then $\mm$ gives
a nilpotent ideal $\mm$ in the quotient ring $A^{\leq r}:=A/\oplus_{d\geq r+1}A^{d}$.

We take the homogeneous decomposition $U_{k}=\oplus_{r\in\N}U_{k}^{r}$.
It is an $A$-module such that the action is given by the multiplication.
The quotient algebra $U_{k}^{\leq r}=U_{k}/\oplus_{d\geq r+1}U_{k}^{d}$
is an $A^{\leq r}$-module, and it equals $\oplus_{0\leq d\leq r}U_{k}^{d}$
as a $\kk$-module. We have the natural projections $\pi^{r}:U_{k}\rightarrow U_{k}^{r}$
as $\kk$-modules and $\pi^{\leq r}:U_{k}\rightarrow U_{k}^{\leq r}$
as algebras.

\begin{Lem}

For any $\tg\in C$, we have $\pi^{\leq\gr(\tg)}s_{\tg}=\pi^{\gr(\tg)}s_{\tg}=m_{\tg}$.

\end{Lem}

\begin{proof}

By Lemma \ref{lem:dominance_refine_grading}, the homogeneous part
of $s_{\tg}$ in $\LP(t^{\prin})$ with the \textbf{minimal} grading has the
grading $\gr(\tg)$ and contains the leading term $x^{\tg}$. Similarly,
the homogeneous part of $\mu_{k}^{*}s_{\tg}$ in $\LP((t')^{\prin})$
with the \textbf{minimal} grading has the grading $\gr'(\tg')=\gr(\tg)$ and
contains the leading term $(x')^{\tg'}$. By Lemma \ref{lem:homogenous_Laurent},
these homogeneous parts of $s_{\tg}$ in $\LP(t^{\prin})$ and $\LP((t')^{\prin})$
respectively are related by mutation. We obtain that $\pi^{\gr(\tg)}s_{\tg}$
is pointed at $\tg$, $\tg'$ in $\LP(t^{\prin})$ and $\LP((t')^{\prin})$
respectively.

Because $\pi^{\gr(\tg)}s_{\tg}\in\LP(t^{\prin})$ is pointed at $\tg$
and has homogeneous grading, we have $\pi^{\gr(\tg)}s_{\tg}=x^{\tg}* F(y_{k}(t^{\prin}))$
for some polynomial $F$ with constant term $1$. Similarly, in $\LP(t')$
we have $\mu_{k}^{*}(\pi^{\gr(\tg)}s_{\tg})=\pi^{\gr'(\tg')}(\mu_{k}^{*}s_{\tg})=(x')^{\tg'}* G((y_{k}((t')^{\prin}))$
for some polynomial $G$ with constant term $1$. Therefore, it is
pointed at $\tg$ and $\tg'$ for the dominance orders associated
to the seeds of the (type $A_{1}$) upper cluster algebra $\upClAlg_{k}$
respectively, where $k$ is the only unfrozen vertex. By using Lemma
\ref{lem:compatible_at_g_vector}, we deduce that $\pi^{\gr(\tg)}s_{\tg}$
agrees with the localized cluster monomial $m_{\tg}$ of $\upClAlg_{k}$.

\end{proof}

\begin{Lem}\label{lem:finite_order_basis}

For any $r\in\N$, $\{\pi^{\leq r}s_{\tg}|\tg\in C,\gr(\tg)\leq r\}$
is a $\kk$-basis of $U_k^{\leq r}$.

\end{Lem}

\begin{proof}

First consider the case $r=0$. For any $\tg\in C$, we have $\pi^{\leq0}s_{\tg}=\pi^{\leq0}(\pi^{\gr(\tg)}s_{\tg})=\pi^{\leq0}m_{\tg}$.
The claim follows from the fact that $\{m_{\tg}|\gr(\tg)=0,\tg\in C\}$
is a $\kk$-basis of the homogeneous component $U_k^{0}$ of $U_k$.

By the nilpotent Nakayama Lemma \ref{lem:Nakayama}, $\{\pi^{\leq r}s_{\tg}|\tg\in C\}$
generates $U_k^{\leq r}$ over $A^{\leq r}$. Notice that $A^{\leq r}$
acts on $s_{\tg}$ by multiplication. We observe that $\{\pi^{\leq r}s_{\tg}|\tg\in C\}$
in fact generates $U_k^{\leq r}$ over $\kk$. Because its non-zero elements
have different leading terms, they are linearly independent and form
a $\kk$-basis.

\end{proof}

\begin{proof}[Proof of Proposition \ref{prop:invariant_decomposition}]

Denote $C^{\leq r}=\{\tg\in C|\gr(\tg)\leq r\}$. Given any $z\in\upClAlg_{k}$,
there exists some $c\in\N^{I'}$ such that $z\cdot x^{c}\in U_{k}$.
Then, up to any order $r\in\N$, we have a finite decomposition inside
the $\kk$-module $U^{\leq r}$ by Lemma \ref{lem:finite_order_basis}:
\begin{align}
\pi^{\leq r}(z\cdot x^{c}) & =\sum_{\tg\in C^{\leq r}}\alpha^{\leq r}(z\cdot x^{c})(\tg)\cdot\pi^{\leq r}s_{\tg}.\label{eq:finite_grading_decomposition}
\end{align}
By letting $r$ tends to $+\infty$, the decomposition \eqref{eq:finite_grading_decomposition}
becomes a possibly infinite decomposition (which converges under the
$\mm$-adic topology on the $A$-module $U_{k}$):

\begin{align}
z\cdot x^{c} & =\sum_{\tg\in C}\alpha(z\cdot x^{c})(\tg)\cdot s_{\tg}\label{eq:converge_gradig_decomposition}
\end{align}

Meanwhile, we have a $\prec_{t^{\prin}}$ decomposition with finitely
many $\prec_{t^{\prin}}$-leading terms in $\LP(t^{\prin})$:

\begin{align}
z\cdot x^{c} & =\sum_{\tg\in C}\alpha_{(t^{\prin})}(z\cdot x^{c})(\tg)\cdot s_{\tg}\label{eq:dominance_t_decomposition}
\end{align}
and a $\prec_{(t')^{\prin}}$-decomposition with finitely many $\prec_{(t')^{\prin}}$-leading
terms in $\LP((t')^{\prin})$: 

\begin{align}
z\cdot x^{c} & =\sum_{\tg\in C}\alpha_{((t')^{\prin})}(z\cdot x^{c})(\tg')\cdot s_{\tg}\label{eq:dominance_t_prime_decomposition}
\end{align}

Recall that $\preceq_{t^{\prin}}$ and $\preceq_{(t')^{\prin}}$ implies
the grading order by Lemma \ref{lem:dominance_refine_grading}. It
follows that both decompositions \eqref{eq:dominance_t_decomposition}\eqref{eq:dominance_t_prime_decomposition}
agree with the decomposition \eqref{eq:converge_gradig_decomposition}.
To be more precise, we can compare the decompositions as follows:
taking the restrictions of both decompositions \eqref{eq:dominance_t_decomposition}\eqref{eq:dominance_t_prime_decomposition}
in grading $\leq r$, then they agree with the finite decomposition
\eqref{eq:finite_grading_decomposition} by Lemma \ref{lem:finite_order_basis}.
Let $r$ tends to $+\infty$, then the restrictions grow to the triangular
decompositions \eqref{eq:dominance_t_decomposition}\eqref{eq:dominance_t_prime_decomposition}
by Lemma \ref{lem:dominance_refine_grading}, while \eqref{eq:finite_grading_decomposition}
grows to \eqref{eq:converge_gradig_decomposition}. 

Notice that $s_{\tg-c}=s_{\tg}\cdot x^{-c}$ by construction. Dividing
both sides of the decomposition \eqref{eq:converge_gradig_decomposition}
by $x^{c}$, we obtain

\begin{align}
\alpha(z)(\tg-c) & :=\alpha(z\cdot x^{c})(\tg)\nonumber \\
z & =\sum_{\tg-c\in\Mc(t^{\prin})}\alpha(z)(\tg-c)\cdot s_{\tg-c}\nonumber \\
 & =\sum_{\tg\in\Mc(t)}\alpha(z)(\tg)\cdot s_{\tg},\label{eq:converge_localized_decomposition}
\end{align}
which gives simultaneously the $\prec_{t^{\prin}}$-decomposition in $\LP(t^{\prin})$
and the $\prec_{(t')^{\prin}}$-decomposition in $\LP((t')^{\prin})$.
We obtain that $\alpha_{(t^{\prin})}(z)(\tg)=\alpha_{((t')^{\prin})}(z)(\tg')=\alpha(z)(\tg)$
for any $\tg\in\Mc(t^{\prin})$.

Finally, let us return to the seeds $t$, $t'$. Let $\proj$ denote the natural projection from $\Z^{I\sqcup I'}$ to $\Z^I$. It induces the $\kk$-linear map $\proj$ from $\LP(t^{\prin})$ to $\LP(t)$ such that $\pr (x^\tg)=x^g$, and similarly the $\kk$-linear map $\proj$ from $\LP((t')^{\prin})$ to $\LP(t')$ such that $\pr (x^{\tg'})=x^{g'}$. We deduce the claim follows
by applying the linear maps $\proj$ to the decomposition \eqref{eq:converge_localized_decomposition}
and by Lemma \ref{lem:dominance_under_projection}.

\end{proof}

\subsection{Bases with tropical properties}

We show that tropical properties of a collection $\cS$ implies that
it is a basis. Assume that $t$ is injective-reachable and denote
$t=\overleftarrow{\mu}t[-1]$.

As in Section \ref{sec:change_seed}, we restrict ourselves to consider elements in the upper cluster algebras to avoid the difficulty of defining mutations for formal Laurent series.

\begin{Thm}\label{thm:tropical_finite_decomposition}

Assume that the full rank assumption holds. If a subset $\cS$ of the upper cluster algebra $\upClAlg(t)$ is compatibly pointed at the seeds appearing along
the mutation sequence $\overleftarrow{\mu}$ from $t[-1]$ to $t$,
then $\cS$ is a basis of $\upClAlg(t)$.

\end{Thm}

\begin{proof}
Let there be given any $z\in\upClAlg$. Working with the seed $t$, we have a $\prec_{t}$-decomposition in $\widehat{\LP(t)}$:
\begin{eqnarray*}
z & = & \sum\alpha_{t}(z)(g)\cdot s_{g}
\end{eqnarray*}
Notice that $\cS$ remains pointed in the seed $t[-1]$ by our assumption. Similarly, working with the seed $t[-1]$, we have a $\prec_{t[-1]}$-decomposition in $\widehat{\LP(t[-1])}$:
\begin{eqnarray*}
z & = & \sum\alpha_{t[-1]}(z)(\phi_{t[-1],t}g)\cdot s_{g}
\end{eqnarray*}
Since $z$ and $\cS$ are contained in the upper cluster algebra $\upClAlg(t)$, the above decompositions take place in $\LP(t)$ and  $\LP(t[-1])$ respectively.

By applying Proposition \ref{prop:invariant_decomposition} for adjacent
seeds along the sequence $\overleftarrow{\mu}$ from $t[-1]$ to $t$,
we obtain that $\alpha_{t}(z)(g)=\alpha_{t[-1]}(z)(\phi_{t[-1],t}g)$,
and $\phi_{t,t[-1]}\supp(\alpha_{t[-1]}(z))=\supp(\alpha_{t}(z))=\{g|\alpha_{t}(z)(g)\neq0\}$.

Notice that $\supp(\alpha_{t}(z))$ has finitely many $\prec_{t}$-maximal
elements which we denote by $g^{(i)}$, $1\leq i\leq l$, $0\neq l\in\N$.
Then any $s_{g}$ appearing satisfies $\deg^{t}s_{g}=g\preceq_{t}g^{(i)}$
for some $i$. Similarly, $\supp(\alpha_{t[-1]}(z))$ has finitely
many $\prec_{t[-1]}$-maximal elements which we denote by $\phi_{t[-1],t}h^{(j)}$,
$1\leq j\leq r$, $0\neq r\in\N$, for some $h^{(j)}\in \Mc(t)$. Then any $s_{g}$ appearing satisfies
$\deg^{t[-1]}s_{g}=\phi_{t[-1],t}g\preceq_{t[-1]}\phi_{t[-1],t}h^{(j)}=\deg^{t[-1]}s_{h^{(j)}}$
for some $j$. By Proposition \ref{prop:order_reverse}, this is equivalent
to $\psi_{t[-1],t}^{-1}\deg^{t[-1]}s_{g}\succeq_{t}\psi_{t[-1],t}^{-1}\deg^{t[-1]}s_{h^{(j)}}$,
i.e. $\codeg^{t}s_{g}\succeq_{t}\codeg^{t}s_{h^{(j)}}$ by Definition
\ref{def:support_dimension} and Proposition \ref{prop:bipointed_support}(1).
It follows that $g\succeq_{t}\codeg^{t}s_{g}\succeq_{t}\eta^{(j)}:=\codeg^{t}s_{h^{(j)}}$.

Therefore, $\supp(\alpha_{t}(z))$ is contained in $\cup_{i,j}({}_{\eta^{(j)}\preceq_{t}}\tropSet_{\preceq_{t}g^{(i)}})$.
In particular, it is a finite set by the Finite Interval Lemma \ref{lem:finite_interval}.

\end{proof}

Theorem \ref{thm:tropical_finite_decomposition} immediately implies Theorem \ref{thm:all_pointed_bases}(1) and the existence of the generic basis for
injective-reachable skew-symmetric seed $t$ (Theorem \ref{thm:generic_basis_existence}),
see Section \ref{sec:Generic-bases-and} for more details.

\begin{Rem}
When we take $\cS$ to be the collection of theta functions, this
result recovers Theorem \ref{thm:theta_basis} originally proved by
\cite{gross2018canonical}. Their proof is based on a thorough study
of the global monomials, tropical functions, convexity, boundedness
of polytopes and EGM arguments, see \cite[Section 7 8]{gross2018canonical}.
Our proof is specific for the injective-reachable case, but more direct
and elementary.

Note that for both works need the full rank assumption to obtain bases for the (upper) cluster algebra.
\end{Rem}

\section{Main results\label{sec:results_proofs}}

As before, we assume that the seeds satisfy the full rank assumption throughout this section.

\subsection{Bases parametrized by tropical points}

\begin{Lem}\label{lem:general_bipointed_basis}

Let there be given an injective-reachable seed $t=\overleftarrow{\mu}t[-1]$ subject to the full rank assumption,
a subset $\Theta\subset\Mc(t)$, and a collection of Laurent polynomials
$\mathcal{Z}=\{z_{g}\in\LP(t)|g\in\Theta\}$ such that the $z_{g}$
are compatibly pointed at seeds $t,t[-1]$ with $\deg^{t}z_{g}=g$.
Let $\clAlg^{\Theta}$ denote the free $\kk$-module $\oplus_{g\in\Theta}\kk\cdot z_{g}$.
Then the following claims are true.

(1) Let $\mathcal{S}$ be any collection $\mathcal{S}=\{s_{g}\in\clAlg^{\Theta}|g\in\Theta\}$
such that the $s_{g}$ satisfy $\deg^{t}s_{g}=g$ and are compatibly
pointed at seeds $t,t[-1]$. Then it is a $\kk$-basis of $\clAlg^{\Theta}$.

(2) Given any $g\in\Theta$ and $s_{g}\in\clAlg^{\Theta}$ such that
it satisfies $\deg^{t}s_{g}=g$ and is compatibly pointed at seeds
$t,t[-1]$. Then $s_{g}$ has the following decomposition in $\{z_{g}|g\in\Theta\}$:

\begin{align}\label{eq:transition_sz}
s_{g} & =z_{g}+\sum_{g'\in\Theta\cap\Mc(t)_{\prec_{\{t,t[-1]\}}g}}b_{g,g'}z_{g'}
\end{align}
with coefficients $b_{g,g'}\in\kk$. In addition, $\Theta\cap\Mc(t)_{\prec_{\{t,t[-1]\}}g}$
are finite sets for all $g\in\Theta$.

(3) Given any set $\mathcal{S}=\{s_{g}\in\clAlg^{\Theta}|g\in\Theta\}$
such that the $s_{g}$ have decomposition into $\{z_{g}\}$ as in
\eqref{eq:transition_sz}. Then the $s_{g}$ satisfy $\deg^{t}s_{g}=g$ and are compatibly
pointed at seeds $t,t[-1]$. In particular, $\{s_{g}|g\in\Theta\}$
is a $\kk$-basis of $\clAlg^{\Theta}$ by (1).

\end{Lem}

Notice that Lemma \ref{lem:general_bipointed_basis} gives a complete
description of the bases $\mathcal{S}$ in (1) using the special chosen
basis $\mathcal{Z}$ and the transition rule in claim (2).

\begin{proof}

Claim (2):

For any $g\in\Theta$, because $s_{g}\in\clAlg^{\Theta}$ and $\{z_{g}|g\in\Theta\}$
is a basis of $\clAlg^{\Theta}$, $s_{g}$ has a finite decomposition
into $z_{g}$:

\begin{align*}
s_{g} & =\sum_{0\leq i\leq r}b_{i}z_{g_{i}}
\end{align*}
with coefficients $b_{i}\neq0$. By assumption, $s_{g},z_{g_{i}}$ are
compatibly pointed at $t,t[-1]$. Then we can apply
Proposition \ref{prop:decompose_bipointed} and deduce that, by reindexing
$z_{g_{i}}$, we have $g_{0}=g$, $b_{0}=1$, $g_{i}\in\Theta$, $\intv_{g_{i}}\subsetneq\intv_{g}$
for any $i>0$. Notice that the last condition is equivalent to $g_{i}\in\Mc(t)_{\prec_{t,t[-1]}g}$
by Proposition \ref{prop:inclusion_description}. Therefore, we obtain
the claim on the decomposition of $s_{g}$. Finally, $\Theta\cap\Mc(t)_{\prec_{t,t[-1]}g}$
are finite by Proposition \ref{prop:inclusion_description}.

Claim (1):

Because $s_{g}$ are pointed at different degrees, they are linearly
independent by Lemma \ref{lem:triangular_decomposition}. It suffices
to verify the claim that any $z_{g}$, $g\in\Theta$, is a finite
sum of elements from $\{s_{g}|g\in\Theta\}$.

Let us do an induction on the cardinality of the finite set $\Theta\cap\Mc(t)_{\prec_{\{t,t[-1]\}}g}$.
If it is an empty set, we have $z_{g}=s_{g}$ by (2). 

Assume that the claim has been verified for all cardinalities no larger
than $d\in\N$. Let us check the case $|\Theta\cap\Mc(t)_{\prec_{\{t,t[-1]\}}g}|=d+1$.
Take any $g'\in\Theta\cap\Mc(t)_{\prec_{\{t,t[-1]\}}g}$. By Proposition
\ref{prop:inclusion_description}, we have 
\begin{align*}
\Mc(t)_{\prec_{t,t[-1]}g'} & =\{g''\in\Mc(t)|\intv_{g''}\subsetneq\intv_{g'}\}\\
 & \subset\{g''\in\Mc(t)|\intv_{g''}\subsetneq\intv_{g}\}\\
 & =\Mc(t)_{\prec_{\{t,t[-1]\}}g}
\end{align*}
and, in addition, $\Theta\cap\Mc(t)_{\prec_{t,t[-1]}g'}\neq\Theta\cap\Mc(t)_{\prec_{t,t[-1]}g}$
because only the right hand side contains $g'$. Therefore, $|\Theta\cap\Mc(t)_{\prec_{t,t[-1]}g'}|\leq d$
and $z_{g'}$ is a finite sum of elements of $\{s_{g}|g\in\Theta\}$
by our induction hypothesis. By (2), $z_{g}$ is a finite linear composition
of $s_{g}$ and $z_{g'}$, $g'\in\Theta\cap\Mc(t)_{\prec_{\{t,t[-1]\}}g}$,
the claim follows.

Claim (3):

The claim follows from Proposition \ref{prop:sum_bipointed}.

\end{proof}

By applying Lemma \ref{lem:general_bipointed_basis} to injective-reachable
upper cluster algebras, we obtain the following consequences.

\begin{Thm}\label{thm:bipointed_bases}

Let there be given an injective-reachable seed $t=\overleftarrow{\mu}t[-1]$ subject to the full rank assumption. Consider the classical case $\kk=\Z$.

(1) For any collection $\mathcal{S}=\{s_{g}\in\upClAlg|g\in\Mc(t)\}$
such that the $s_{g}$ satisfy $\deg^{t}s_{g}=g$ and are compatibly
pointed at $t$ and $t[-1]$, $\mathcal{S}$ must be a $\kk$-basis of $\upClAlg$.

(2) There exists at least one such basis, which we choose and denote by $\mathcal{Z}=\{z_{[g]}\}$.

(3) The set of all such bases $\mathcal{S}$ is parametrized as follows:

\begin{eqnarray*}
\prod_{g\in\Mc(t)}\kk^{\Mc(t)_{\prec_{\{t,t[-1]\}}g}} & \simeq & \{\mathcal{S}\}\\
((b_{g,g'})_{g'\in\Mc(t)_{\prec_{\{t,t[-1]\}}g}})_{g\in\Mc(t)} & \mapsto & \mathcal{S}=\{s_{g}|g\in\Mc(t)\}
\end{eqnarray*}
such that $s_{g}=z_{g}+\sum_{g'\in\Mc(t)_{\prec_{\{t,t[-1]\}}g}}b_{g,g'}z_{g'}$. In addition, the $\Mc(t)_{\prec_{\{t,t[-1]\}}g}$
are finite sets. 

\end{Thm}

\begin{proof}

It suffices to show that there exists a collection $\mathcal{Z}=\{z_{g}|g\in\Mc(t)\}$
in $\{\mathcal{S}\}$ such that it is a basis of $\upClAlg$. Then
the claim follows from Lemma \ref{lem:general_bipointed_basis}
where we take $\Theta=\Mc(t)$.

If $t$ is skew-symmetric, we can choose $\mathcal{Z}$ to be the
collection of the localized generic cluster characters, which are
known to be compatibly pointed at $t'\in\Delta^{+}$ by \cite[Theorem 1.3]{plamondon2013generic}.
Then, by Theorem \ref{thm:tropical_finite_decomposition}, it is a
basis. See Section \ref{sec:Generic-bases-and} for more details.

For general $t$, we have the theta functions $\theta_{t,g}^{t}$
for any $g\in\Mc(t)$, which are compatibly pointed at $t\in\Delta^{+}$
by \cite{carl2010tropical}\cite{gross2018canonical} (see Theorem
\ref{thm:theta_compatable}). Therefore, the set $\{\theta_{t,g}^{t}|g\in\Mc(t)\}$
is a basis of $\upClAlg$ by Theorem \ref{thm:tropical_finite_decomposition}
(alternatively, see Theorem \ref{thm:theta_basis} by \cite{gross2018canonical}).

\end{proof}

Recall that $s_{[g]}\in\upClAlg$, $[g]\in\tropSet$, is said to be pointed
at $[g]$ if $s_{[g]}$ is pointed at the representative $g\in\Mc(t)$
of $[g]$ in $\LP(t)$ for all seeds $t\in\Delta^{+}$.

\begin{Thm}[Theorem \ref{thm:all_pointed_bases}]

Let there be given an injective-reachable seed $t=\overleftarrow{\mu}t[-1]$ subject to the full rank assumption. Consider the classical case $\kk=\Z$.

(1) For any collection $\mathcal{S}=\{s_{[g]}\in\upClAlg|[g]\in\tropSet\}$
such that $s_{[g]}$ are pointed at the tropical points $[g]$, $\mathcal{S}$
must be a $\kk$-basis of $\upClAlg$ containing all cluster monomials. 

(2) There exists at least one such basis, which we choose and denote by $\mathcal{Z}=\{z_{[g]}\}$.

(3) The set of all such bases $\mathcal{S}$ and is parametrized as follows:

\begin{eqnarray*}
\prod_{g\in\tropSet}\kk^{\tropSet_{\prec_{\Delta^{+}}[g]}} & \simeq & \{\mathcal{S}\}\\
((b_{[g],[g']})_{[g']\in\tropSet_{\prec_{\Delta^{+}}[g]}})_{[g]\in\tropSet} & \mapsto & \mathcal{S}=\{s_{[g]}|[g]\in\tropSet\}
\end{eqnarray*}
such that $s_{[g]}=z_{[g]}+\sum_{[g']\in\tropSet_{\prec_{\Delta^{+}}[g]}}b_{[g],[g']}z_{[g']}$. In addition, the $\tropSet_{\prec_{\Delta^{+}}[g]}$
are finite sets. 

\end{Thm}

\begin{proof}

Notice that being compatibly pointed at $\Delta^{+}$ is a stronger
property than being compatibly pointed at $t,t[-1]$. Theorem \ref{thm:bipointed_bases}
gives a complete description of the bases $\{s_{g}|g\in\Mc(t)\}$
such that $s_{g}$ are compatibly pointed at $t,t[-1]$. Let us choose
a basis $\mathcal{Z}$ such that it is compatibly pointed at $\Delta^{+}$,
where possible candidates include the theta basis or the generic basis
for skew-symmetric seeds (see the proof of Theorem \ref{thm:bipointed_bases}).

Then a basis $\{s_{g}|g\in\Mc(t)\}$, where $s_{g}=z_{g}+\sum_{g'\in\Mc(t)_{\prec_{\{t,t[-1]\}}g}}b_{g,g'}z_{g'}$
satisfy this stronger property if and only if $\deg^{t'}s_{g}=\phi_{t',t}g=\deg^{t'}z_{g}$
$\forall t'$, i.e. if and only if $\deg^{t'}z_{g}\succ_{t'}\deg^{t'}z_{g'}$
for any $t'$ and $g'\in\Mc(t)_{\prec_{\{t,t[-1]\}}g}$ with non-vanishing
coefficient $b_{g,g'}$. This condition is equivalent to require all
$z_{g'}$ appearing satisfy $g'\in\Mc(t)_{\prec_{\Delta^{+}}g}$.
The parametrization of $\{\mathcal{S}\}$ follows. 

Finally, $\mathcal{S}$ contains all cluster monomials by Lemma \ref{lem:compatible_at_g_vector}.

\end{proof}

We can understand the bijection in Theorem \ref{thm:all_pointed_bases}
as a statement that the set of bases with a choice of a special one
is parametrized by the transition matrices, which are all nilpotent
lower $\prec_{\Delta^{+}}$-triangular matrices with indices given
by the tropical points.

\begin{Rem}[Basis and frozen factors] \label{rem:frozen_shift_basis}

In cluster theory, it is often natural to ask for pointed bases that factor
through the frozen variables, i.e. $s_{g}\cdot x^{c}=s_{g+c}$ for
$c\in\Z^{I_{\fv}}$, see Definition \ref{def:factor_through_frozen}.
To adapt Theorem \ref{thm:all_pointed_bases} for this purpose, we
simply impose the restriction that the special basis $\mathcal{Z}$
factors through the frozen variables, and that the transition matrix
satisfy $b_{g+c,g'+c}=b_{g,g'}$. Possible candidates include the
theta basis or the generic basis, see the proof of Theorem \ref{thm:bipointed_bases}.

\end{Rem}

Finally, let us give a description of the bases in terms of ``correct''
support dimensions, which is more natural from the view of representation
theory.

\begin{Prop}\label{prop:all_support_defined}

Let there be given an injective-reachable seed $t=\overleftarrow{\mu}t[-1]$ and any $g\in\Mc(t)$.

(1) The support dimension $\suppDim g$
is well-defined in $\yCone^{\geq0}(t)$.

(2) The support dimension $\suppDim g$ only depends on the principal part $\pr_{I_{\ufv}} g$ and $B(t)$. 
\end{Prop}

\begin{proof}

(1) By Proposition \ref{prop:bipointed_support}, it suffices to find
a Laurent polynomial $z_{g}\in\LP(t)$ with degree $g$ and compatibly
pointed at seeds $t,t[-1]$. One can take $z_{g}$ to be the theta
function $\theta_{t,g}^{t}$ or the localized generic cluster character
$\gen_{g}$ in Section \ref{sec:Generic-bases-and} for skew-symmetric
$t$.

(2) We see in (1) that $\suppDim g$ can be realized as the support dimension of the corresponding theta function or the localized generic cluster character. (2) follows from the properties of such elements.
\end{proof}

\begin{Thm}\label{thm:correct_support_basis}

Consider the classical case $\kk=\Z$. Let there be given an injective-reachable seed $t$  subject to the full rank assumption and a collection
of bipointed elements $\mathcal{S}=\{s_{g}|g\in\Mc(t)\}$ of $\upClAlg$.
Then $\mathcal{S}$ is a basis of $\upClAlg$ whose elements $s_{g}$
are compatibly pointed at seeds $t,t[-1]$ if and only if $\suppDim s_{g}=\suppDim g$
for all $g$.

\end{Thm}

\begin{proof}

The claim follows from Theorem \ref{thm:bipointed_bases} and Proposition
\ref{prop:bipointed_support}.

\end{proof}

\subsection{Generic bases and its analog\label{sec:Generic-bases-and}}

Let us investigate the generic basis and analogous bases constructed
from cluster characters. At this moment, generic quantum cluster characters are not defined in general. Correspondingly, we have to restrict to the classical case $\kk=\Z$.

\begin{Def}\label{def:factor_through_frozen}

Given a seed $t$ and a subset $\Theta$ of $\Mc(t)$. A set of
pointed formal Laurent series $\mathcal{Z}=\{z_{g}|g\in\Theta\}$,
where $\deg^{t}z_{g}=g$, is said to factor through the frozen variables
$x_{j}$, $j\in I_{\fv}$, if for any $g,g'\in\Theta$ such that $g'=g+f_{j}$,
we have $z_{g'}=z_{g}\cdot x_{j}$. 

In this case, we define the localization of $\mathcal{Z}$ to be the
set $\mathcal{Z}[x_{j}^{-1}]_{j\in I_{\fv}}=\{z_{g}\cdot x^{m}|g\in\Theta,m\in\Z^{I_{\fv}}\}$.

\end{Def}

Let there be given an injective-reachable skew-symmetric seed $t$.
Take $T$ to be the corresponding cluster tilting object and identify
$K_{0}(\add T)\simeq\Mc(t)\simeq\Z^{I}$. For any $g\in\Z^{I}$,
there exists some $m\in\Z^{I_{\fv}}$ depending on $g$, such that $\gen_{g+m}$
is the \emph{generic cluster character} in \cite{plamondon2013generic} (see Section \ref{sec:cluster_category}).
Define the localized generic cluster character $\gen_{g}$ to be the
localization $\gen_{g+m}\cdot x(t)^{-m}$.

\begin{Thm}[Theorem \ref{thm:generic_basis_existence}]

Let there be given an injective-reachable skew-symmetric seed $t$.
Then the set of the localized generic cluster characters $\{\gen_{g}|g\in\Mc(t)\}$
is a basis of $\upClAlg$, called the generic basis.

\end{Thm}

\begin{proof}

Recall that the generic cluster characters are known to be compatibly
pointed in all seeds by by Plamondon \cite[Theorem 1.3]{plamondon2013generic}.
So do the localized generic cluster characters. Then, Theorem \ref{thm:tropical_finite_decomposition} provides a
direct proof for the statement.

Alternatively, as an indirect proof, we use the fact that the theta
basis exists (\cite{gross2018canonical}, Theorem \ref{thm:theta_basis})
and choose it to be the special basis in the main theorem (Theorem
\ref{thm:all_pointed_bases}). Then the collection of the generic
cluster characters is also a basis by the main theorem.

\end{proof}

Let us discuss analog of the generic basis, where the objects chosen
are not necessarily generic.

\begin{Lem}\label{lem:support_factor_through_frozen}

Given injective-reachable seeds $t=\overleftarrow{\mu}t[-1]$. Assume
that some $g\in\Mc(t)$ has a support dimension $\suppDim g$, then
for any $m\in\Z^{I_{\fv}}$, $g+m$ has the support dimension $\suppDim(g+m)=\suppDim g$.

\end{Lem}

\begin{proof}

For any $k\in I_{\ufv}$, we have $\phi_{\mu_{k}t,t}(g+m)=\phi_{\mu_{k}t,t}(g)+\phi_{\mu_{k}t,t}(m)=\phi_{\mu_{k}t,t}(g)+m$.
Repeatedly applying tropical transformations along $\overleftarrow{\mu}^{-1}$
from $t$ to $t[-1]$, we obtain that $\phi_{t[-1],t}(g+m)=\phi_{t[-1],t}(g)+m$.
Because the map $\psi_{t[-1],t}^{-1}$ is linear, we obtain that 
\begin{align*}
\psi_{t[-1],t}^{-1}\phi_{t[-1],t}(g+m)-(g+m) & =\psi_{t[-1],t}^{-1}\phi_{t[-1],t}g+\psi_{t[-1],t}^{-1}m-g-m\\
 & =\psi_{t[-1],t}^{-1}\phi_{t[-1],t}g-g\\
 & =\tB(t)\cdot\suppDim g
\end{align*}

The claim follows from definition of support dimension.

\end{proof}

\begin{Prop}\label{prop:support_generic_character}

Given an injective-reachable skew-symmetric seed $t$. For any $g\in\Mc(t)$,
it has the support dimension given by that of the localized generic
cluster character: $\suppDim g=\suppDim\gen_{g}$.

\end{Prop}

\begin{proof}

It follows from \cite[Theorem 1.3]{plamondon2013generic} that generic
cluster characters $\gen_{g+m}$, $g\in\Z^{I_{\ufv}}$, $m\in\Z^{I_{\fv}}$,
are compatibly pointed at $t$ and $t[-1]$. This implies the claim
for such $g+m$ by Proposition \ref{prop:bipointed_support}. Finally,
the claim holds for all $g\in\Mc(t)$ by Lemma \ref{lem:support_factor_through_frozen}. 

\end{proof}

\begin{Thm}\label{thm:CC_basis_existence}

Let there be given an injective-reachable skew-symmetric seed $t$.
Denote $\Theta=\{\Ind^{T}M|M\in^{\bot}(\Sigma T_{\fv})\}$ where $T$ is
the cluster tilting object corresponding to $t$. Let $\{M_{g}|g\in\Theta\}$
denote the set of any given objects in $^{\bot}(\Sigma T_{\fv})$ such that
$\Ind^{T}M_{g}=g$ and $\dim FM_{g}=\suppDim(g)$. Then, the set of
localized cluster characters $\{CC(M_{g})|g\in\Theta\}[x_{j}^{-1}]_{j\in I_{\fv}}$
is a basis of the upper cluster algebra $\upClAlg$.

\end{Thm}

\begin{proof}

By the following Lemma \ref{lem:localized_index}, for any $g\in\Mc(t)$,
there is a localized cluster character $CC(M_{g+m})\cdot x^{-m}$
pointed at $g$ such that $g+m\in\Theta$. The claim follows from
Proposition \ref{prop:support_generic_character} and Theorem \ref{thm:correct_support_basis}.

\end{proof}

\begin{Lem}\label{lem:localized_index}

For any $g\in\Mc(t)$, there exists some $m\in\N^{I_{\fv}}$ such
that $g+m=\Ind^{T}X$ for some $X\in^{\bot}(\Sigma T_{\fv})$

\end{Lem}

\begin{proof}

Consider the object $Y=(\oplus_{k\in I_{\ufv}}T_{k}^{[g_{k}]_{+}})\oplus(\oplus_{k\in I_{\ufv}}I_{k}^{[-g_{k}]_{+}})$.
It follows that $\Ind^{T}Y=\pr_{I_{\ufv}}g+m'$ for some $m'\in\Z^{I_{\fv}}$.
Then we can take $m=([m'_{j}]_{+})_{j\in I_{\fv}}$ and $X=Y\oplus(\oplus_{j\in I_{\fv}}T_{j}^{[-m'_{j}]_{+}})$.

\end{proof}

By \cite[Theorem 1.18]{BerensteinFominZelevinsky05}, the cluster
algebra $\clAlg$ agrees with the upper cluster algebra $\upClAlg$
when the initial quiver $Q(t_{0})$ is acyclic. The following result
shows that a basis consisting of cluster characters can be constructed
quite easy in this case.

\begin{Cor}

Let there be given a skew-symmetric seed $t$ and the corresponding
principal quiver $Q(t)$ is acyclic. Let $T$ denote the corresponding
cluster tilting object.

(1) Denote $\Theta=\{\Ind^{T}M|M\in^{\bot}(\Sigma T_{\fv})\}$. Then
for any choice of objects $M_{g}\in^{\bot}(\Sigma T_{\fv})$ with
$\Ind^{T}M_{g}=g$, the set of localized cluster characters $\{CC(M_{g})|g\in\Theta\}[x_{j}^{-1}]_{j\in I_{\fv}}$
is a basis of the cluster algebra $\clAlg=\upClAlg$.

(2) Choose a pair $(V_{d},m)$ for each dimension vector $d\in\N^{I_{\ufv}}$
and $m\in\N^{I}$ such that $V_{d}$ is a $d$-dimensional $\C Q(t)$-module
and $\supp m\cap\supp d=\emptyset$. Then the set of localized cluster
characters $\{x^{m}CC(V_{d})|\forall(V_{d,}m)\}[x_{j}^{-1}]_{j\in I_{\fv}}$
is a basis of the cluster algebra $\clAlg=\upClAlg$.

\end{Cor}

\begin{proof}

(1) Notice that $^{\bot}(\Sigma T_{\fv})$ is a full subcategory of
$\cC_{(\tQ,\tW)}$ and all morphisms from $M\in^{\bot}(\Sigma T_{\fv})$
to $\Sigma T_{k}$ do not factor through $T_{\fv}$. We obtain that,
for any $k\in I_{\ufv}$, 
\begin{align*}
\Hom_{\cC_{(\tQ,\tW)}}(M,\Sigma T_{k}) & =\Hom_{^{\bot}(\Sigma T_{\fv})/(T_{\fv})}(M,\Sigma T_{k})\\
 & =\Hom_{\cC_{(Q,W)}}(\underline{M},\Sigma\underline{T}_{k}).
\end{align*}
Therefore, the support dimension of $CC(M)$ equals that of $CC(\underline{M})$.

Let us work in $\cC_{(Q,W)}$. Any object $\underline{M}$ has an
$\add\uT$-approximation $\underline{T}^{(1)}\rightarrow\underline{T}^{(0)}\rightarrow\uM$.
By applying the functor $F=\Hom(\underline{T},\Sigma(\ ))$, we obtain
a long exact sequence

\begin{align*}
0 & \rightarrow F\underline{M}\rightarrow F\Sigma\underline{T}^{(1)}\rightarrow F\Sigma\underline{T}^{(0)}\rightarrow\cdots.
\end{align*}
Notice that $\Sigma\underline{T}^{(1)},\Sigma\underline{T}^{(0)}$
are injective modules of the Jacobian algebra $J_{(Q,W)}$. Because
$Q$ is acyclic, we have $W=0$ and $J_{(Q,W)}$ agrees with the hereditary
path algebra $\C Q$. As a consequence, we obtain a short exact sequence

\begin{align*}
0 & \rightarrow F\underline{M}\rightarrow F\Sigma\underline{T}^{(1)}\rightarrow F\Sigma\underline{T}^{(0)}\rightarrow0.
\end{align*}

It turns out that $\suppDim CC(\underline{M})=\dim F\underline{M}$
only depends on the index $\Ind^{\underline{T}}\underline{M}$.

Therefore, for any $M\in^{\bot}(\Sigma T_{\fv})$, $\dim FM=\dim FM_{g}=\suppDim g$
where $M_{g}$ is an generic object of index the $\Ind^{T}M$. The
claim follows from Theorem \ref{thm:CC_basis_existence}.

(2) In the proof for (1), set $V_{d}=F\underline{M}$ and $d=\dim V_{d}$.
Let $R$ denote the matrix whose column vectors are the dimension
vectors of the injectives $F(\Sigma T_{k})$, $k\in I_{\ufv}$. Then
$d=-R\cdot\pr_{I_{\ufv}}g$. Since $Q$ is acyclic, $R$ is a unitriangular
matrix after relabelling the vertices. In particular, $R$ is invertible.
We can then deduce (2) from (1).

\end{proof}

\section{Related topics and discussion\label{sec:Other-applications-and}}

As before, we assume that the seeds satisfy the full rank assumption in the following discussion.

\subsection{Deformation factors\label{sec:Deformation-factors}}

\begin{Def}

The subset $\tropSet_{\prec_{\Delta^{+}}[g]}$ is called the deformation
factor associated to $[g]$.

\end{Def}

We have seen in the main theorem (Theorem \ref{thm:all_pointed_bases})
that basis deformation are controlled by the deformation factors $\tropSet_{\prec_{\Delta^{+}}[g]}$
, $[g]\in\tropSet$. These factors are important for constructing the
bases. It is
therefore a natural question to understand them. One might want to
interpret these deformation factors in terms of homology in cluster
category, or representation theory (such as quiver representations
or Lie theory), or tropical geometry. 

As a first step, one might ask when the deformation factors are empty
set, i.e., one can not do a deformation. Recall that all bases in
construction share the localized cluster monomials by Lemma \ref{lem:compatible_at_g_vector}.
This immediately implies the following property.

\begin{Prop}

If $g\in\Mc(t)$ is the maximal $\prec_{t}$-degree of any localized
cluster monomial, then $\Mc(t)_{\prec_{\Delta^{+}}g}=\emptyset$.

\end{Prop}

This property is a supporting evidence for the following natural expectation.

\begin{Conj}\label{conj:rigid_deformation_factor}

Assume that $t$ is skew-symmetric. If a generic object $M_{g}$ for
some $g\in\Mc(t)$ in the cluster category is rigid, then $\Mc(t)_{\prec_{\Delta^{+}}g}=\emptyset$.

\end{Conj}

\begin{Rem}[Open obit conjecture]\label{rem:open_orbit_conjecture}
If Conjecture \ref{conj:rigid_deformation_factor} is true, then all bases parametrized by tropical points must share the same elements for the $g$-vectors corresponding to rigid modules. In particular, if we consider the cluster algebras arising from the coordinate rings of unipotent subgroups, then the generic bases (dual semi-canonical bases) and the dual canonical bases share such elements. Then we obtain the open orbit conjecture for these coordinate rings (see \cite{GeissLeclercSchroeer10}).
\end{Rem}

One might also study the cardinality $|\Mc(t)_{\prec_{\Delta^{+}}g}|$.

\begin{Eg}[Bases for Kronecker type]\label{eg:Kronecker}

Take $\kk=\Z$, $I=I_{\ufv}=\{1,2\}$, and the initial seed $t_{0}$ such that
$B(t_{0})=\left(\begin{array}{cc}
0 & -2\\
2 & 0
\end{array}\right)$. Then $y_{1}=x_{2}^{2}$ and $y_{2}=x_{1}^{-2}$, which in particular
have even degrees. Denote $\delta=(1,-1)$, $z=x^{\delta}(1+y_{2}+y_{1}y_{2})$.
It is well known that the corresponding upper cluster algebra $\upClAlg$ has the generic
basis which consists of the cluster monomials and $z^{d}$, $d\geq1$. 

Notice that $\delta$ is invariant under tropical transformations.
Then any pointed element $s_{d\delta}\in\upClAlg$ parametrized by
the tropical point $d\delta$ must always have the leading degree
$d\delta$ in all seeds. One can deduce that the deformation from
$z^{d}$ to $s_{d\delta}$ cannot involve any cluster monomials. Also
notice that $s_{d\delta}$ is pointed and $\eta-d\delta$ have even
degrees whenever $\eta\prec_{t}d\delta$. We obtain

\begin{eqnarray*}
s_{d\delta} & = & z^{d}+\sum_{k\geq0,d-2k\geq0}b_{d-2k}z^{d-2k},\ b_{d-2k}\in\Z.
\end{eqnarray*}
Therefore, the deformation factors has cardinality $|\Mc(t_{0})_{\prec_{t}d\delta}|=[\frac{d}{2}]$
where $[\ ]$ denote the integer part. 

The infinite families of bases in this Kronecker example is also found
in \cite{rupel2019affine} by using Lie theory.

\end{Eg}

Finally, still working with the Kronecker Example \ref{eg:Kronecker},
it is known that the the triangular basis (dual canonical basis) and
theta basis (greedy basis) differ by taking the usual quiver Grassmannians
or the transverse quiver Grassmannians \cite{dupont2010transverse}\cite{irelli2013homological}.
We expect that one might relate the deformation factor to such a difference.

\subsection{Quantum bases}\label{sec:quantum_bases}

Theorems \ref{thm:all_pointed_bases} \ref{thm:bipointed_bases} \ref{thm:correct_support_basis}  are stated for the classical case $\kk=\Z$. Let us consider their analogs for the the quantum case $\kk=\Z[q^{\pm\Hf}]$.

\begin{Thm}
Consider the quantum case $\kk=\Z[q^{\pm\Hf}]$. Assume the quantum seeds are injective-reachable and satisfy the full rank assumption. 

(1) The analog of Theorem \ref{thm:all_pointed_bases}(1) remains true.

(2) If the analog of Theorem \ref{thm:all_pointed_bases}(2) is true, then the analog of Theorem \ref{thm:all_pointed_bases}(3) is true.

(3) If the analog of Theorem \ref{thm:bipointed_bases}(2) is true, then the analogs of Theorem \ref{thm:bipointed_bases} \ref{thm:correct_support_basis} are true.
\end{Thm}
\begin{proof}
The analog of Theorem \ref{thm:all_pointed_bases}(1) is a direct consequence of Theorem \ref{thm:tropical_finite_decomposition}.

Assume that a basis has been given by the analog of Theorem \ref{thm:all_pointed_bases}(2) (resp. \ref{thm:bipointed_bases}(2)), the proof for the analog of Theorem \ref{thm:all_pointed_bases}(3) (resp. \ref{thm:bipointed_bases}) is the same as before. More precisely, we use Lemma \ref{lem:general_bipointed_basis} by setting $\Theta=\Mc(t)$ and $\cA^{\Theta}=\upClAlg$ the free $\kk$-module spanned by the given basis.

As before, the analog of Theorem \ref{thm:correct_support_basis} is a consequence of Proposition \ref{prop:bipointed_support} and the analog of Theorem \ref{thm:bipointed_bases}.
\end{proof}

The obstruction appears in the analogs of Theorems \ref{thm:bipointed_bases}(2) and \ref{thm:all_pointed_bases}(2), i.e. we do not know a quantum basis $\mathcal{Z}$ inside a quantum upper cluster algebra. Thanks to \cite{davison2019strong}, the quantum theta functions provide such a
basis for an injective-reachable skew-symmetric seed $t$ subject to the full rank assumption, see Remark  \ref{rem:update_DM}. By \cite{qin2020dual}, the dual canonical basis provides another such basis, when $t$ arises from a quantum unipotent cell with symmetrizable Cartan datum.

\subsection{Weak genteelness\label{sec:Weakly-genteelness}}

For skew-symmetric injective-reachable seeds, we have seen the existence
of the generic basis, which is constructed using the representation
theory. It is natural to ask if we can also interpret the theta basis
using the representation theory in this case.

For finite dimensional Jacobian algebra $J_{(Q,W)}$, Bridgeland has
defined a representation theoretic version of the scattering diagram
called the stability scattering diagram, for which some theta functions
have a representation theoretic formula \cite{bridgeland2017scattering}.
Then this formula is effective for theta functions appearing in upper
cluster algebras, if the stability scattering diagram is equivalent
to the cluster scattering diagram in \cite{gross2018canonical}. If
so, we say the quiver with potential is \emph{weakly genteel}.

We refer the reader to Section \ref{sec:Theta-functions} \ref{sec:proof_weak-genteelness}
for necessary definitions for the statements below.

\begin{Thm}[Theorem \ref{thm:genteel}]

Take $\kk=\Z$. Let there be given a skew-symmetric injective-reachable seed $t$.
Then Bridgeland's representation theoretic formula is effective for
theta functions in the cluster scattering diagram. Moreover, the stability
scattering diagram and the cluster scattering diagram are equivalent.

\end{Thm}

The proof is given in Section \ref{sec:proof_weak-genteelness}.

\begin{Conj}\label{conj:weakly_genteel}

Let $(Q,W)$ be any quiver with a generic potential such that the
Jacobian algebra $J_{(Q,W)}$ is finite dimensional, then it is weakly
genteel.

\end{Conj}

Here, by a generic potential, we mean a generic point in the space of all potentials in the sense of \cite{DerksenWeymanZelevinsky08}. In particular, it is assumed to be non-degenerate.

\begin{Conj}

The
Jacobian algebra $J_{(Q,W)}$ in Conjecture \ref{conj:weakly_genteel} is genteel.

\end{Conj}

Here, we take a generic potential from the space of all potentials
\cite{DerksenWeymanZelevinsky07}. It might be possible to only assume
the potential $W$ to be non-degenerate. We can also generalize the
conjectures to the case when $J_{(Q,W)}$ has infinite dimension,
for which we need to modify the stability scattering diagram by working
with nilpotent modules, see \cite{Nagao10}.

\subsection{Partial compactification\label{sec:partial_compactification}}

In representation theory, it is often natural to work with a partial
compactified upper cluster algebra $\bUpClAlg$, defined as the ring
of regular functions over some partial compactification $\bAVar$
of the cluster variety $\AVar$. Correspondingly, it is natural to
ask if the basis of $\upClAlg$ give rise to a basis of $\bUpClAlg$,
defined by choosing those basis elements without poles on the boundary
$\bAVar\backslash\AVar$. 

For example, for some important cluster algebras arising from representation
theory, $\bUpClAlg$ agrees with the compactified cluster algebra
$\bClAlg$, and the boundary condition demands the functions in $\upClAlg$
to have no pole at the frozen variables $x_{j}=0$, $j\in I_{\fv}$.
Moreover, in such examples, for any $j\in I_{\fv}$, there exists
a seed $t\in\Delta^{+}$ such that $b_{jk}(t)\geq0$ for any $k\in I_{\ufv}$,
called a seed optimized for $x_{j}$ following \cite{gross2018canonical}.

This is a difficult and largely open question in general. Consider the classical case $\kk=\Z$. \cite[Section 9]{gross2018canonical}
gives an affirmative answer when one has enough optimized seeds, for
which a subset of the theta functions form a basis of $\bUpClAlg$.
Let $\Theta$ denote the set of tropical points parametrizing this
subset.

Then, $\bUpClAlg$ is a $\Z$-module spanned by the basis $\{\theta_{g}|g\in\Theta\}$.
We can apply our Lemma \ref{lem:general_bipointed_basis} and obtain
many bases of $\bUpClAlg$. As in the proof of Theorem \ref{thm:all_pointed_bases},
we deduce that the set the bases of $\bUpClAlg$ compatibly pointed
at seeds in $\Delta^{+}$ is in bijection with $\prod_{[g]\in\Theta}\Z^{\Theta\cap(\tropSet_{\prec_{\Delta^{+}}[g]})}$.
Again, the restriction of the generic basis $\{\gen_{\tg}|\tg\in\Theta\}$
is a such basis.

\appendix

\section{Scattering diagrams} \label{sec:appendix_scattering_diagram}
For simplicity, we assume that the seeds satisfy the full rank assumption so that the scattering diagrams and theta functions can be easily constructed, except in the proof of Theorem \ref{thm:genteel}. The construction for an arbitrary seed can be obtained by taking a projection from the construction for the corresponding seed with principal coefficients, see \cite{gross2018canonical}.

\subsection{Basics of scattering diagrams and theta functions\label{sec:Theta-functions}}

We refer the reader to the original paper of Gross-Hacking-Keel-Kontsevich
\cite{gross2018canonical} for more details.

Let $t_{0}$ be any chosen initial seed. Recall that we have an isomorphism
$N(t_{0})\simeq\Z^{I}$ with the natural basis $\{e_{i}\}$ which
endows $\Z^{I}$ with the bilinear form $\{\ ,\ \}$, and an isomorphism
$\Mc(t_{0})\simeq\Z^{I}$ with the natural basis $\{f_{i}\}$. Define
the $\yCone^{\geq0}(t_{0})$-graded Poisson algebra $A=\Z[\yCone^{\geq0}(t_{0})]=\oplus_{n\in\yCone^{\geq0}(t_{0})}y(t_{0})^{n}$
such that $\{y(t_{0})^{n},y(t_{0})^{n'}\}=-\{n,n'\}y(t_{0})^{n+n'}$.
Let $|n|$ denote $\sum n_{i}$. Then $\frg=A_{>0}$ is naturally
a graded Lie algebra via its Poisson bracket. Its completion $\hat{\frg}$
is defined to be the inverse limit of $\mathfrak{g}/\oplus_{n:|n|>k}\frg_{n}$,
$k>0$. Let $G$ denote the group $\exp\mathfrak{\hat{g}}$ defined
via the Baker–Campbell–Hausdorff formula. 

Recall that the matrix $\tB(t_{0})$ gives us an embedding $p^{*}:\Z^{I_{\ufv}}\rightarrow\Z^{I}$
such that $p^{*}(e_{k})=\sum_{i\in I}b_{ik}f_{i}$. Let $A$ acts
linearly on $\widehat{\LP(t_{0})}$ via the derivation $\{A,\ \}$
such that $\{y(t_{0})^{n},x(t_{0})^{m}\}=\langle m,n\rangle x(t_{0})^{m+p^{*}(n)}$.
In particular, $\{y(t_{0})^{n},x(t_{0})^{p^{*}n'}\}=-\{n,n'\}x(t_{0})^{p^{*}(n'+n)}$,
which explains the minus sign in the definition of $A$. By the injectivity
of $p^{*}$, this induces a faithful action of $G$ on $\widehat{\LP(t_{0})}$.

A wall in $M(t_{0})_{\R}=M(t_{0})\otimes\R$ is a pair $(\mathfrak{d},\mathfrak{p}_{\mathfrak{d}})$
such that $\frd$ is a codimension $1$ rational polyhedral cone,
$\frd\subset n_{0}^{\bot}$ for some primitive normal direction $n_{0}\in\N^{I_{\ufv}}$,
and the wall crossing operator $\frp_{\frd}\in\exp(y(t_{0})^{n_{0}}\Z[[y(t_{0})^{n_{0}}]])\subset G$.
It is said to be non-trivial if $\frp_{\frd}$ is. A scattering diagram
$\frD$ is a collection of walls subject to some finiteness condition
in \cite{gross2018canonical}. $\frD$ cuts out many chambers in $M(t_{0})_{\R}$,
among which we have two special ones $\cC^{\pm}:=(\pm\R_{\geq0}^{I_{\ufv}})\oplus\R^{I_{\fv}}$. 

Given two chambers $\cC^{1},\cC^{2}$ and any smooth path $\gamma:[0,1]\rightarrow\R^{I}$
from the interior of $\cC^{1}$ to that of $\cC^{2}$. We first assume
that $\gamma$ intersects transversely the interior of finitely many
walls $\frd_{i}$ with normal direction $n_{i}\in\N^{I}$, $1\leq i\leq r$,
at time $t_{1}\leq t_{2}\leq\ldots\leq t_{r}$, and we define the
wall crossing operation along $\gamma$ to be $\frp_{\gamma}:=\frp_{\frd_{r}}^{\varepsilon_{r}}\cdots\frp_{\frd_{1}}^{\varepsilon_{1}}$
where $\varepsilon_{i}=-\sign\langle\gamma'(t_{i}),n_{i}\rangle$.
Let $\gamma^{-1}:[0,1]\rightarrow\R^{I}$ denote the inverse path
$\gamma^{-1}(t)=\gamma(1-t)$. Then $\frp_{\gamma^{-1}}=\frp_{\gamma}^{-1}$.
We further define $\frp_{\gamma}$ for the case of infinitely many
intersections as an inverse limit, see \cite{gross2018canonical}.

We say $\frD$ is consistent if $\frp_{\gamma}$ is always independent
of the choice of $\gamma$, which we can denote by $\frp_{\cC^{2},\cC^{1}}$.
Two scattering diagrams are equivalent if they give the same $\frp_{\gamma}$
for any $\gamma$. The equivalent class of a consistent $\mathfrak{D}$
is determined by $\frp_{\cC^{-},\cC^{+}}$ \cite[Theorem 1.17]{gross2018canonical}\cite[2.1.6]{kontsevich2014wall}.

A wall $(\frd,\frp_{\frd})$ with primitive normal direction $n_{0}\in\N^{I_{\ufv}}$
is said to be incoming if $p^{*}(n_{0})\in\frd$. Up to equivalence,
for any collection $\frD_{\mathrm{in}}$ of incoming walls, there
exists a unique consistent scattering diagram $\frD$ such that $\frD_{\mathrm{in}}\subset\frD$
and there is no incoming walls in $\frD\backslash\frD_{\mathrm{in}}$.

For any chosen base point $Q\in M(t_{0})_{\R}$ not contained in any non-trivial
wall, the theta functions $\theta_{Q,g}^{t_{0}}$, $\forall g\in\Z^{I}$,
are certain formal Laurent series in $\widehat{\LP(t_{0})}$ which
takes the form $x(t_{0})^{g}(1+\sum_{n>0}c_{n}y(t_{0})^{n})$ with
coefficients $c_{n}\in\Z$. It has the property $\theta_{Q',g}^{t_{0}}=\mathfrak{p}_{\gamma}\theta_{Q,g}^{t_{0}}$
for any path $\gamma$ from $Q$ to $Q'$. If $Q$ is a generic point
in some chamber $\cC$, then $\theta_{Q,g}^{t_{0}}$ only depends
on the chamber, and we write $\theta_{\cC,g}^{t_{0}}=\theta_{Q,g}^{t_{0}}$.
We write $\theta_{g}=\theta_{\cC^{+},g}^{t_{0}}$ for simplicity.

Notice that, to each seed $t\in\Delta^{+}$, one can associate a chamber
$\cC^{t}$. In particular, we have $\cC^{t_{0}}=\cC^{+}$ and, when
$t_{0}[1]$ exists, $\cC^{t_{0}[1]}=\cC^{-}$. So we can write $\theta_{t,g}^{t}=\theta_{\cC^{t},g}^{t}$.

Let $\mathrm{Li}_2(\ )$ denote the dilogarithm function, see \cite{gross2018canonical}.

\begin{Def}

Let there be given an initial seed $t_{0}$. The consistent scattering
diagram $\frD$ whose incoming walls are $(e_{k}^{\bot},\exp(-d_{k}\mathrm{Li}_{2}(-y(t_{0})_{k})))$,
$k\in I_{\ufv}$, is called the cluster scattering diagram associated
to $t_{0}$.

\end{Def}

Consider cluster scattering diagrams from now on. Let us compare our
tropical transformations with that of \cite{gross2018canonical}.
By \cite{gross2018canonical}, for any $k\in I_{\ufv}$, we have the
tropical transformation which preserves the theta functions
\begin{eqnarray*}
T_{k}: & \Z^{I} & \rightarrow\Z^{I}\\
 & m=\sum m_{i}f_{i} & \mapsto m+[m_{k}]_{+}\sum_{i}b_{ik}f_{i}
\end{eqnarray*}

Consider the seed $t'=\mu_{k}t_{0}$. We identify $\Mc(t')\simeq\Z^{I}\simeq\Mc(t_{0})$
such that the basis elements $f_{i}'=f_{i}(t')$ are given by \eqref{eq:tropical_f}
with the sign $\varepsilon=+$:

\begin{align*}
f_{i}' & =\begin{cases}
f_{i} & i\neq k\\
-f_{k}+\sum_{j}[-b_{jk}]_{+}f_{j} & i=k
\end{cases}
\end{align*}

\begin{Lem}

For any $m=\sum m_{i}f_{i}$, the coordinates of its image $m'=T_{k}m=\sum m_{i}'f_{i}'$
are given by the tropical transformation $\phi_{t',t_{0}}$ (Definition
\eqref{def:tropical_transformation}):

\begin{align*}
m_{i}' & =\begin{cases}
-m_{k} & i=k\\
m_{i}+m_{k}[b_{ik}]_{+} & i\neq k,m_{k}>0\\
m_{i}+m_{k}[-b_{ik}]_{+} & i=k,m_{k}<0
\end{cases}
\end{align*}

\end{Lem}

\begin{proof}

By the mutation rule of $f_{i}'$, we have 
\begin{align*}
m' & =\sum m_{i}'f_{i}'\\
 & =m_{k}'f_{k}'+\sum_{i:\ i\neq k}m_{i}'f_{i}'\\
 & =m_{k}'(-f_{k}+\sum_{i}[-b_{ik}]_{+}f_{i})+\sum_{i\neq k}m_{i}'f_{i}\\
 & =(-m_{k}')f_{k}+\sum_{i:\ i\neq k}(m_{i}'+[-b_{ik}]_{+}m_{k}')f_{i}
\end{align*}

First assume $m_{k}\geq0$, by the transformation $T_{k}$, we have
\begin{align*}
m' & =m+m_{k}\sum_{i}b_{ik}f_{i}\\
 & =m_{k}f_{k}+\sum_{i:\ i\neq k}(m_{i}+b_{ik}m_{k})f_{i}
\end{align*}

Therefore, we obtain 
\begin{align*}
m_{k}' & =-m_{k}\\
m_{i}' & =m_{i}+(b_{ik}+[-b_{ik}]_{+})m_{k}\\
 & =m_{i}+[b_{ik}]_{+}m_{k}
\end{align*}

Next, assume that $m_{k}<0$,by the transformation $T_{k}$, we have 

\begin{align*}
m' & =m\\
 & =m_{k}f_{k}+\sum_{i:\ i\neq k}m_{i}f_{i}
\end{align*}

Therefore, we obtain

\begin{align*}
m_{k}' & =-m_{k}\\
m_{i}' & =m_{i}+[-b_{ik}]_{+}m_{k}.
\end{align*}

\end{proof}

\begin{Thm}\cite{gross2018canonical}

For any $t\in\Delta^{+}$ and $g\in\cC^{t}$, the theta function $\theta_{g}$
is a localized cluster monomial in the seed $t$. In particular, the
cluster variables $x_{i}(t)$ equals $\theta_{g_{i}(t)}$ in $\widehat{\LP(t_{0})}$.

\end{Thm}

\begin{Thm}\label{thm:theta_compatable}\cite{gross2018canonical}

Given seeds $t=\overleftarrow{\mu}t_{0}$, then we have $\overleftarrow{\mu}^{*}\theta_{t,g}^{t}=\theta_{t_{0},\phi_{t_{0},t}g}^{t_{0}}$
for any $g\in\Mc(t)$.

\end{Thm}

\begin{proof}

It seems that \cite{gross2018canonical} does not present this result
exactly in this way. Nevertheless, it is known that theta functions
are pointed at the tropical points by \cite{gross2018canonical},
and the claim follows.

To prove the statement, one will need the ``CPS Lemma'' \cite[Section 4]{carl2010tropical}
which says the theta functions are sent to theta functions by wall
crossings, as well as \cite[Theorem 3.5 Proposition 3.6 Proposition 4.3 Theorem 4.4]{gross2018canonical}.
These results together tell us the construction of theta functions
is compatible with monomial automorphisms $\tau_{k,\epsilon}$, Hamiltonian
automorphisms (wall-crossings) $\rho_{k,\epsilon}$, and the tropical
transformation $T_{k}=\phi_{\mu_{k}t_{0},t_{0}}$ associated to the
mutation of the initial seeds. Then it is also compatible with mutations
because $\mu_{k}^{*}=\rho_{k,\epsilon}\tau_{k,\epsilon}$.

\end{proof}

\begin{Thm}\cite[Proposition 8.25]{gross2018canonical}\label{thm:theta_basis}

Let there be given an injective-reachable initial seed $t_{0}$. Then
the theta functions $\theta_{g},\ \forall g\in\Mc(t_{0})$, are
pointed Laurent polynomials in $\LP(t_{0})$. In addition, they form
a basis of the upper cluster algebra $\upClAlg$, called the theta
basis.

\end{Thm}

\begin{proof}

Because $t_{0}$ is injective-reachable, the cluster algebra has large
cluster complex in the sense of \cite[Definition 8.23]{gross2018canonical}.
In particular, it verifies the EGM condition (enough global monomials).
The claims follow from \cite[Proposition 8.25]{gross2018canonical}.

\end{proof}

\subsection{Weak genteelness and the proofs\label{sec:proof_weak-genteelness}}

We shall show that, by combining known results from the cluster theory,
the scattering diagrams and some theta functions for skew-symmetric
injective-reachable seeds have a representation theoretic description
by Bridgeland \cite[Theorem 1.4]{bridgeland2017scattering}, see also
\cite{cheung2016tropical}. Related definitions could be found in
Section \ref{sec:Theta-functions}.

Let there be given an injective-reachable skew-symmetric initial seed
$t_{0}$ subject to the full rank assumption. We take the corresponding principal quiver with a non-degenerate
potential $(Q,W)$. Omit the symbol $t_{0}$ for simplicity.

We take the stability scattering diagram $\frD_{\ufv}^{(st)}$ constructed
by integrating moduli of semistable modules of $J_{(Q,W)}$ introduced
by \cite[Section 11]{bridgeland2017scattering} . The walls $(\frd,\frp_{\frd})$
of $\frD_{\ufv}^{(st)}$ live in $N_{\ufv}(t_{0})_{\R}^{*}=\Hom_{\Z}(N_{\ufv}(t_{0}),\R)$.
We define the stability scattering diagram $\frD^{(st)}$ to be the
collection of walls $(\frd\oplus\R^{I_{\fv}},\frp_{\frd})$ which
live in $M(t_{0})_{\R}$. As in Section \ref{sec:Theta-functions},
we define the action\footnote{Our action is slightly different than the one in \cite[Section 10.3]{bridgeland2017scattering}
so that it is faithful.} of the Poisson algebra $A=\Z[\N^{I_{\ufv}}]$ on $\widehat{\LP(t_{0})}$
such that $\{y^{n},x^{m}\}=\langle m,n\rangle x{}^{m+\tB n}$. Then
the corresponding group $G$ and its action on $\widehat{\LP(t_{0})}$
are given as in Section \ref{sec:Theta-functions}. 

The scattering diagram $\frD^{(st)}$ can be described via representation
theory \cite[Theorem 1.1 Theorem 1.3]{bridgeland2017scattering}.
Moreover, Bridgeland has the following description of theta functions
in $\frD^{(st)}$ \cite[Theorem 1.4]{bridgeland2017scattering}:

\begin{align*}
\theta_{Q,m}^{(st),t_{0}} & =x^{m}\cdot(\sum K(n,m,Q)\cdot x^{\tB\cdot n})
\end{align*}
where the base point $Q$ does not belong to any non-trivial wall,
$m\in\N^{I_{\ufv}}$, and $K(n,m,Q)$ is the Euler characteristic
of the quotient module Grassmannian $\Quot_{n}U(m,Q)$ consisting
of $n$-dimensional quotient modules of certain module $U(m,Q)$ in
a tilted heart, see \cite[Section 8.4]{bridgeland2017scattering}
for details. A representation theoretic formula for other theta functions
is unknown at the moment. In particular, by taking $Q$ to be a generic
point in $\cC^{-}$ and $m=f_{i}$, the formula reads

\begin{align*}
\theta_{t_{0}[1],f_{i}}^{(st),t_{0}} & =\begin{cases}
x_{i}\cdot(\sum\chi(\Quot_{n}P_{i})\cdot x^{\tB\cdot n}) & i\in I_{\ufv}\\
x_{i} & i\in I_{\fv}
\end{cases}
\end{align*}
where $P_{i}$, $i\in I_{\ufv}$, corresponds to the $i$-th projective
module of $J_{(Q,W)}$.

\begin{Def}[Genteelness]\cite{bridgeland2017scattering}

We say the Jacobian algebra $J_{(Q,W)}$ is genteel, if the only modules
$V$ such that $V$ are $p^{*}(\dim V)$-stable are those simples
$S_{k}$, $k\in I_{\ufv}$.

\end{Def}

Let $\frD$ denote the cluster scattering diagram associated to $t_{0}$,
see Section \ref{sec:Theta-functions}. The following property is
a weaker version of the genteelness.

\begin{Def}[Weak genteelness]

We say the Jacobian algebra $J_{(Q,W)}$ is weakly
genteel, if $\frD^{(st)}$ and $\frD$ are equivalent.

\end{Def}

Given a consistent scattering diagram $\frD$ live in $\R^{I}$, let
us construct the opposite scattering diagram $\frD^{op}$ in $\R^{I}$,
see Example \ref{eg:opposite_scattering_diagram}.

Recall that $A=\Z[y^{n}]_{n\in\N^{I_{\ufv}}}$ is a Poisson algebra
such that $\{y^{n},y^{n'}\}=-\{n,n'\}y^{n+n'}$ and $\frg=A_{>0}$,
see Section \ref{sec:Theta-functions}. Define the opposite Poisson
algebra $A^{op}=\Z[y^{n}]$ with the Poisson bracket $\{\ ,\ \}^{op}=-\{\ ,\ \}$
and Lie algebra $\frg{}^{op}=A_{>0}^{op}$. We have $\iota:A\simeq A{}^{op}$
as $\Z$-modules such that $\iota(y^{n})=y^{n}$.

\begin{Lem}\label{lem:opposite_group}

For any $u,v,w\in\frg$ such that $\exp w=\exp u\cdot\exp v$ we have
$\exp\iota w=\exp\iota v\cdot\exp\iota u$ in $G^{op}:=\exp\hat{\frg}^{op}$ 

\end{Lem}

\begin{proof}The claim follows from the Baker–Campbell–Hausdorff
formula which defines the group multiplication on $G$ and $G^{op}$.

\end{proof}

Let $\kappa$ denote the isomorphism $m\mapsto-m$ on $\R^{I}$ as
well as the induced automorphism $\kappa(x^{m})=x^{\kappa m}$ on
$\widehat{\Z[x^{m}]_{m\in\Z^{I}}}$. The opposite scattering diagram
$\frD^{op}$ in $\R^{I}$ is defined to be the collection of walls
$(\kappa\frd\subset n_{\frd}^{\bot},\exp\iota u)$ for any wall $(\frd\subset n_{\frd}^{\bot},\exp u)\in\frD$.
Given any path $\gamma$, let $\frp_{\gamma}^{op}$ denote the corresponding
wall crossing operator in $\frD^{op}$. 

\begin{Lem}\label{lem:compare_wall_crossing}

(1) If $\frp_{\gamma^{-1}}=\exp w$, then $\frp_{\kappa\gamma}^{op}=\exp\iota w$.

(2) $\frD^{op}$ is consistent.

\end{Lem}

\begin{proof}

(1) For any given generic path $\gamma$ such that $\frp_{\gamma}=\frp_{\frd_{r}}^{\varepsilon_{r}}\cdots\frp_{\frd_{1}}^{\varepsilon_{1}}$,
the wall crossing operator in $\frD^{op}$ along $\kappa\gamma$ is
$$\frp_{\kappa\gamma}^{op}=\exp(-\varepsilon_{r}\iota\log\frp_{\frd_{r}})\cdots\exp(-\varepsilon_{1}\iota\log\frp_{\frd_{1}}).$$
The claim follows from the equality $\frp_{\frd_{1}}^{-\varepsilon_{1}}\cdots\frp_{\frd_{r}}^{-\varepsilon_{r}}=\frp_{\gamma}^{-1}=\frp_{\gamma^{-1}}$
and Lemma \ref{lem:opposite_group}.

(2) The claim follows from (1) by taking all paths.

\end{proof}

\begin{Prop}

Let $t_{0}^{op}$ denote the seed opposite to $t_{0}$ such that $\tB(t_{0}^{op})=-\tB$
and we take the same strictly positive integers $d_{i}$. Let $\frD(t_{0})$
and $\frD(t_{0}^{op})$ denote the cluster scattering diagrams associated
to $t_{0},t_{0}^{op}$ respectively. Then $\frD(t_{0}^{op})$ is equivalent
to the opposite scattering diagram $\frD(t_{0})^{op}$, where we identify
$\Mc(t_{0})\simeq\Z^{I}\simeq\Mc(t_{0}^{op})$ such that $f_{i}(t_{0})\mapsto f_{i}(t_{0}^{op})$
and $N(t_{0})\simeq Z^{I}\simeq N(t_{0}^{op})$ such that $e_{i}(t_{0})\mapsto e_{i}(t_{0}^{op})$.

\end{Prop}

\begin{proof}

Notice that the bilinear form on $N(t_{0}^{op})$ is opposite to that
of $N(t_{0})$ under the identification. So we can view $A^{op}$
and $\frg^{op}$ in the construction of $\frD(t_{0})^{op}$ as $A(t_{0}^{op})$
and $\frg(t_{0}^{op})$ associated to $t_{0}^{op}$. Furthermore,
$\frD(t_{0})^{op}$ is consistent with the incoming walls are $(e_{k}^{\bot},\exp(-d_{k}\mathrm{Li}_{2}(-y_{k})))$.
Therefore, $\frD^{op}(t_{0})$ is equivalent to the cluster scattering
diagram $\frD(t_{0}^{op})$.

\end{proof}

The actions of $A$ and $A^{op}$ on $\widehat{\Z[x^{m}]_{m\in\Z^{I}}}$
are defined as in Section \ref{sec:Theta-functions} using the scattering
diagrams associated to the seeds $t_{0}$, $t_{0}^{op}$ respectively.

\begin{Lem}\label{lem:opposite_wall_crossing}

We have $\frp_{\kappa\gamma}^{op}(\kappa x^{m})=\kappa\frp_{\gamma}x^{m}$
for any path $\gamma$. 

\end{Lem}

\begin{proof}

Recall that the action of $A$ satisfy $\{y^{n},x^{m}\}:=\langle m,n\rangle x^{m+\tB n}$
and the action of $A^{op}$ satisfy $\{y^{n},x^{m}\}^{op}:=\langle m,n\rangle x^{m-\tB\cdot n}$.
Then, we have $\{\iota y^{n},\kappa x^{m}\}^{op}=-\langle m,n\rangle x^{-m-\tB n}=\kappa\{-y^{n},x^{m}\}$.
Therefore, $\exp(\iota w)(\kappa x^{m})=\kappa(\exp(-w)(x^{m}))$.
The claim follows from Lemma \ref{lem:compare_wall_crossing}(1).

\end{proof}

\begin{Eg}\label{eg:opposite_scattering_diagram}

Let $I=I_{\ufv}=\{1,2\}$ and $B(t_{0})=\left(\begin{array}{cc}
0 & -1\\
1 & 0
\end{array}\right)$, $\epsilon=-B(t_{0})$. The cluster scattering diagram $\frD=\frD(t_{0})$
in $M(t_{0})_{\R}=\R f_{1}\oplus\R f_{2}\simeq\R^{2}$ is given by
\begin{align*}
\frD & =\{(e_{1}^{\bot},\exp(-\mathrm{Li}_{2}(-y_{1})\},(e_{2}^{\bot},\exp(-\mathrm{Li}_{2}(-y_{2})),(\R_{\geq0}(1,-1),\exp(-\mathrm{Li}_{2}(-y_{1}y_{2}))\}
\end{align*}
where the Poisson bracket on $A=\Z[y_{1},y_{2}]$ satisfies $\{y_{i},y_{j}\}=-\epsilon_{ij}y_{i}y_{j}$.
By \cite{gross2018canonical}, we have $\exp(-\mathrm{Li}_{2}(y^{n}))(x^{m})=x^{m}(1+y^{n})^{\langle m,n\rangle}$.
Let $\gamma$ denote a path from $\cC^{+}$ to $\cC^{-}$. One checks
that, for $v_{i}=\tB e_{i}$,
\begin{align*}
\frp_{\gamma}x_{1} & =x_{1}(1+x^{v_{1}}+x^{v_{1}+v_{2}})\\
\frp_{\gamma}x_{2} & =x_{2}(1+x^{v_{2}})\\
\frp_{\gamma^{-1}}x_{1}^{-1} & =x_{1}^{-1}(1+x^{v_{1}})\\
\frp_{\gamma^{-1}}x_{2}^{-1} & =x_{2}^{-1}(1+x^{v_{2}}+x^{v_{1}+v_{2}})
\end{align*}

The opposite scattering diagram is given by

\begin{align*}
\frD^{op} & =\{(e_{1}^{\bot},\exp(-\mathrm{Li}_{2}(-y_{1})\},(e_{2}^{\bot},\exp(-\mathrm{Li}_{2}(-y_{2})),(\R_{\geq0}(-1,1),\exp(-\mathrm{Li}_{2}(-y_{1}y_{2}))\}
\end{align*}
and the Poisson bracket on $A^{\op}=\Z[y_{1},y_{2}]$ satisfies $\{y_{i},y_{j}\}=\epsilon_{ij}y_{i}y_{j}$.
The opposite seed $t_{0}^{op}$ has $B(t_{0}^{op})=-B(t_{0})$, $\epsilon(t_{0}^{op})=-B(t_{0}^{op})$.
The corresponding cluster scattering diagram is just $\frD^{op}$.
One checks that 
\begin{align*}
\frp_{\kappa\gamma}^{op}x_{1}^{-1} & =x_{1}^{-1}(1+x^{-v_{1}}+x^{-v_{1}-v_{2}})\\
\frp_{\kappa\gamma}^{op}x_{2}^{-1} & =x_{2}^{-1}(1+x^{-v_{2}})
\end{align*}

\end{Eg}

\begin{proof}[Proof of Theorem \ref{thm:genteel}]

We refer the reader to \cite{gross2018canonical}\cite{bridgeland2017scattering}
and \cite{Nagao10} for details of the related notions below.

As in \cite{gross2018canonical}, replacing $t_0$ by a principal coefficient seed $t_0^{\prin}$ by adding principal framing vertices $i'$ for all $i\in I$ if necessary, we first assume that the seed $t_0$ satisfies the full rank assumption. 

Recall that the equivalence classes of the consistent scattering diagrams
$\frD$, $\frD^{(st)}$ are determined by the corresponding wall crossing
operators $\mathfrak{p}_{t_{0}[1],t_{0}}$ and $\mathfrak{p}_{t_{0}[1],t_{0}}^{(st)}$
respectively. Because $G$ acts faithfully on $\widehat{\LP(\Mc(t_{0}))}$,
it suffices to show that $\mathfrak{p}_{t_{0}[1],t_{0}}$ and $\mathfrak{p}_{t_{0}[1],t_{0}}^{(st)}$
have the same action.

Because $\frD$ is a cluster scattering diagram, for any index $i$,
$\theta_{-f_{i}}$ agrees with the localized cluster variable 
\begin{align*}
\theta_{-f_{i}}=\begin{cases}
x_{i}^{-1}\cdot\sum_{n}(\chi(\Gr_{n}I_{i})\cdot x^{\tB\cdot n}) & i\in I_{\ufv}\\
x_{i}^{-1} & i\in I_{\fv}
\end{cases}
\end{align*}
where $I_{i}\in\cC_{(Q,W)}$ corresponds to the $i$-th injective
module of the Jacobian algebra $J_{(Q,W)}$ and $f_{i}$ denote the
$i$-th unit vector.

As a conceptual proof, we observe that the theta functions in $\frD^{(st)}(t)$
can be calculated by using the tilting theory as in the work of Nagao
\cite[Section 8.3]{bridgeland2017scattering}\cite{Nagao10}. Moreover,
the main result of Nagao \cite{Nagao10} is the deduction of the Caldero-Chapoton
type formula for cluster monomials from the the tilting theory. By
the main result of Nagao, the theta function $\theta_{-f_{i}}^{(st)}$
in $\frD^{(st)}$ must agree with the localized cluster variable with
degree $-f_{i}$. Therefore, we obtain $\mathfrak{p}_{t_{0},t_{0}[1]}(x_{i}^{-1})=\theta_{-f_{i}}=\theta_{-f_{i}}^{(st)}=\mathfrak{p}_{t_{0},t_{0}[1]}^{(st)}(x_{i}^{-1})$
for any $i\in I_{\ufv}$. The faithfulness of $G$ implies $\mathfrak{p}_{t_{0},t_{0}[1]}=\mathfrak{p}_{t_{0},t_{0}[1]}^{(st)}$
and, consequently, $\mathfrak{p}_{t_{0}[1],t_{0}}=\mathfrak{p}_{t_{0}[1],t_{0}}^{(st)}$.
We refer the reader to Mou's upcoming work \cite{mou2019scattering}
for a detailed treatment (and a quantized version) in terms of the
Hall algebras.

Instead of re-examining the arguments of \cite{Nagao10} in the setting
of \cite{bridgeland2017scattering}, we give an alternative proof
by using the scattering diagram $\frD^{op}$ opposite to $\frD$.

Choose any generic smooth path $\gamma$ from $\cC^{+}$ to $\cC^{-}$
in $\R^{I}$. Assume $\frp_{\gamma}x_{k}=x_{k}\cdot f$, then $\frp_{\kappa\gamma}^{op}x_{k}^{-1}=x_{k}^{-1}\cdot\kappa f$
by Lemma \ref{lem:opposite_wall_crossing}. Because $\frD^{op}\simeq\frD(t_{0}^{op})$
and $\kappa\gamma$ is a path from $\cC^{-}$ to $\cC^{+}$, we obtain
the cluster expansion formula for cluster variables associated to
$t_{0}^{\op}[1]$:
\begin{align*}
\frp_{\kappa\gamma}^{op}x_{k}^{-1} & =x_{k}^{-1}\sum_{n}\chi(\Gr_{n}I_{k}^{op})\cdot x^{-\tB n}
\end{align*}
where $k\in I_{\ufv}$, $I_{k}^{op}$ is the $k$-th injective module
associated to the opposite algebra $J_{(Q,W)}^{op}$. By the natural
isomorphism $\Quot_{n}(P_{k})\simeq\Gr_{n}I_{k}^{op}$, we obtain

\begin{align*}
\frp_{t_{0}[1],t_{0}}x_{k} & =\frp_{\gamma}x_{k}=x_{k}(\sum_{n}\chi(\Quot^{n}P_{k})x^{\tB n}).
\end{align*}
In addition, it trivially holds that $\frp_{t_{0}[1],t_{0}}x_{i}=x_{i}$
for any $i\in I_{\fv}$. Therefore, $\frp_{t_{0}[1],t_{0}}$ and $\frp_{t_{0}[1],t_{0}}^{(st)}$
have the same action on $\widehat{\LP(t_{0})}$.

Finally, if the original seed $t_0$ does not satisfy the full rank assumption and we have worked with its principal coefficient seed $t_0^{\prin}$ as in \cite{gross2018canonical}, we can consider the natural projection $\proj$ from $\Z^{I(t_0^{\prin})}$ to $\Z^{I(t_0)}$ and the induced $\Z$-linear projection $\proj$ from $\LP(t_0^{\prin})$ to $\LP(t)$. By applying the projections, we recover the theta functions and scattering diagrams for $t_0$ from those for $t_0^{\prin}$, see \cite{gross2018canonical} for details. The desired claim follows.
\end{proof}

\begin{Rem}

By \cite{qin2017triangular}, a seed is injective-reachable if and
only if it is ``projective reachable''. Recall that projective modules
of $J=J_{(Q,W)}$ can be identified with injective modules of $J^{op}=J_{(Q^{op},W^{op})}$.
We deduce that if $t$ is injective-reachable, then so is $t^{op}$.
Consequently, if $J$ is weakly genteel, then so is $J^{op}$.

\end{Rem}






\newcommand{\etalchar}[1]{$^{#1}$}
\def\cprime{$'$}
\providecommand{\bysame}{\leavevmode\hbox to3em{\hrulefill}\thinspace}
\providecommand{\MR}{\relax\ifhmode\unskip\space\fi MR }
\providecommand{\MRhref}[2]{%
	\href{http://www.ams.org/mathscinet-getitem?mr=#1}{#2}
}
\providecommand{\href}[2]{#2}

\end{document}